\PassOptionsToPackage{colorlinks,
     linkcolor={red!50!black},
     citecolor={blue!50!black},
     urlcolor={blue!80!black}}{hyperref}
\documentclass{IEEEtran}
\usepackage{amsmath,amsfonts,amssymb,mathrsfs,amsthm}
\usepackage{algorithm,algorithmic}
\usepackage[square]{natbib}
\usepackage[utf8]{inputenc}
\usepackage[english]{babel}
\usepackage{graphicx,tikz,pgfplots}

\usepackage{xcolor}

\ifCLASSOPTIONcompsoc
\usepackage[caption=false,font=normalsize,labelfon
t=sf,textfont=sf]{subfig}
\else
\usepackage[caption=false,font=footnotesize]{subfi
g}
\fi

\newcommand{\nl}{\vskip 0.3cm}

\newcommand{\C}{{\mathbb C}}
\newcommand{\R}{{\mathbb R}}
\newcommand{\Z}{{\mathbb Z}}
\newcommand{\N}{{\mathbb N}}

\newcommand{\scal}[2]{\langle #1,#2\rangle}

\renewcommand{\Im}{\mbox{\rm Im}\,}
\renewcommand{\Re}{\mbox{\rm Re}\,}

\newtheorem{thm}{Theorem}[section]
\newtheorem*{thm*}{Theorem}
\newtheorem{cor}[thm]{Corollary}
\newtheorem{lem}[thm]{Lemma}
\newtheorem*{lem*}{Lemma}
\newtheorem{prop}[thm]{Proposition}
\newtheorem*{prop*}{Proposition}

\allowdisplaybreaks

\title{Phase retrieval for wavelet transforms}
\author{Ir\`ene Waldspurger
\IEEEcompsocitemizethanks{\IEEEcompsocthanksitem The author is with CNRS and CEREMADE, Université Paris Dauphine (e-mail: waldspurger@ceremade.dauphine.fr).}}
\date{}

\begin{document}

\maketitle
\begin{abstract}
This article describes a new algorithm that solves a particular phase retrieval problem, with important applications in audio processing: the reconstruction of a function from its scalogram, that is, from the modulus of its wavelet transform.

It is a multiscale iterative algorithm, that reconstructs the signal from low to high frequencies. It relies on a new reformulation of the phase retrieval problem, that involves the holomorphic extension of the wavelet transform. This reformulation allows to propagate phase information from low to high frequencies.

Numerical results, on audio and non-audio signals, show that reconstruction is precise and stable to noise. The complexity of the algorithm is linear in the size of the signal, up to logarithmic factors. It can thus be applied to large signals.
\end{abstract}


\section{Introduction}

The spectrogram is an ubiquitous tool in audio analysis and processing, eventually after being transformed into mel-frequency cepstrum coefficients (MFCC) by an averaging along frequency bands. A very similar operator, yielding the same kind of results, is the modulus of the wavelet transform, sometimes called scalogram.

The phase of the wavelet transform (or the windowed Fourier transform in the case of the spectrogram) contains information that cannot be deduced from the single modulus, like the relative phase of two notes with different frequencies, played simultaneously. However, this information does not seem to be relevant to understand the perceptual content of audio signals \citep{balan,risset}, and only the modulus is used in applications. To clearly understand which information about the audio signal is kept or lost when the phase is discarded, it is natural to consider the corresponding inverse problem: to what extent is it possible to reconstruct a function from the modulus of its wavelet transform? The study of this problem mostly begun in the early 80's \citep{griffin_lim,nawab}.

On the applications side, solving this problem allows to resynthesize sounds after some transformation has been applied to their scalogram. Examples include blind source separation \citep{virtanen} or audio texture synthesis \citep{bruna_texture}.

\nl
The reconstruction of a function from the modulus of its wavelet transform is an instance of the class of \emph{phase retrieval problems}, where one aims at reconstructing an unknown signal $x\in\C^n$ from linear measurements $Ax\in\C^m$, whose phase has been lost and whose modulus only is available, $|Ax|$. These problems are known to be difficult to solve.

Two main families of algorithms exist. The first one consists of iterative algorithms, like gradient descents or alternate projections \citep{fienup,gerchberg}. In the case of the spectrogram, the oldest such algorithm is due to Griffin and Lim \citep{griffin_lim}. These methods are simple to implement but, because the phase retrieval problem is non-convex, they are not guaranteed to converge towards the correct solution, and often get stuck in local minima. For measurement matrices $A$ that are chosen at random (according to a suitable distribution), this problem can be overcome with a careful initialization of the algorithm \citep{netrapalli,candes_wirtinger,candes_wirtinger2}. However, these methods do not apply when measurements are not random. In the case of the spectrogram or scalogram, the reconstructed signals tend to present auditive artifacts. For the spectrogram, the performances can be improved by applying the algorithm window by window and not to the whole signal at the same time \citep{bouvrie}. If additional information on the nature of the audio signal is available, it can also be taken into account in the algorithm \citep{achan,eldar_sparse}. Nevertheless, in the general case, the reconstruction results are still perfectible.

More recently, convexification methods have been proposed \citep{chai,candes2}. For generic phase retrieval problems, these methods are guaranteed to return the true solution with high probability when the measurement matrix $A$ is chosen at random. In the case of the spectrogram or scalogram, the matrix is not random and the proof does not hold. However, numerical experiments on small signals indicate that the reconstruction is in general correct \citep{sun_smith,maxcut}. Unfortunately, these methods have a high complexity, making them difficult to use for phase retrieval problems whose size exceeds a few hundred.

\nl
In this article, we present a new algorithm for the reconstruction of a function from its scalogram. As convexification methods, it offers a reconstruction of high quality. However, it has the complexity of an iterative method (roughly proportional to the size of the signal, up to logarithmic factors) and can be applied to large signals. The memory it requires is also proportional to the size of the signal.

The algorithm is multiscale: it performs the reconstruction frequency band by frequency band, from the lowest frequencies up to the highest ones.

\nl
The main idea of this algorithm is to introduce an equivalent formulation of the phase retrieval problem (by using the analytic extension of the wavelet transform). This reformulation has two advantages.
\begin{itemize}
\item It gives an explicit method to propagate towards higher scales the phase information reconstructed at lower scales. Thus, at each scale, we have a good initial guess for the subproblem that we have to solve (in the spirit of \citep{netrapalli,candes_wirtinger}, but for non-random measurements).
\item The local optimization algorithm naturally derived from the reformulation, although non-convex, seems very robust to the problem of local minima. In a particular case, if we allow for a small modification, we can even prove that it has no spurious local minimum.
\end{itemize}

\nl
Additionally,
\begin{itemize}
\item we introduce a multigrid error correction method, to detect and correct eventual errors of the reconstruction algorithm afterwards;
\item we use our algorithm to numerically study the intrinsic stability of the phase retrieval problem, and highlight the role played by the sparsity or non-sparsity of the wavelet coefficients.
\end{itemize}


\nl
Section \ref{s:notations} is devoted to definitions and notations. In Section \ref{s:well_posedness}, we discuss the well-posedness of the phase retrieval problem. We explain our reformulation in Section \ref{s:reformulation}, prove its equivalence with the classical formulation, and describe its advantages. In Section \ref{s:algorithm}, we describe the resulting algorithm. In Section \ref{s:multiscale}, we discuss the superiority of multiscale algorithms over non-multiscale ones. Finally, in Section \ref{s:numerical_results}, we give numerical results, and empirically study the stability of the underlying phase retrieval problem.

The source code, and several reconstruction examples, are available at:

{\small http://www-math.mit.edu/~waldspur/wavelets\_phase\_retrieval.html}

\section{Definitions and assumptions\label{s:notations}}

All signals $f[n]$ are of finite length $N$. Their discrete Fourier transform is defined by
\begin{equation*}
\hat f[k]=\underset{n=0}{\overset{N-1}{\sum}}f[n]e^{-2\pi i\frac{kn}{N}},
\quad\quad k=0,\dots,N-1.
\end{equation*}
The circular convolution of two signals $f$ and $g$ is
\begin{equation*}
f\star g[k]=\sum_{n=0}^{N-1}f[n]g[k-n],
\quad\quad k=0,\dots,N-1,
\end{equation*}
where we set $g[k-n]=g[k-n+N]$ when $k-n<0$.

We define a family of wavelets $(\psi_j)_{0\leq j\leq J}$ by
\begin{equation*}
\hat\psi_j[k]=\hat\psi(a^jk),
\quad\quad k=0,...,N-1.
\end{equation*}
where the \textit{dilation factor} $a$ can be any number in $]1;+\infty[$ and $\psi:\R\to\C$ is a fixed \textit{mother wavelet}. We assume that $J$ is sufficiently large so that $\hat\psi_J$ is negligible outside a small set of points. An example is shown in Figure \ref{fig:discrete_wavelets}.
\begin{figure}
\begin{center}
\begin{minipage}[b]{0.48\textwidth}
  \centering
%
%
\begin{tikzpicture}[scale=0.65]

\begin{axis}[%
scale only axis,
every outer x axis line/.append style={black},
every x tick label/.append style={font=\color{black}},
xmin=0,
xmax=10,
xlabel={\Large $k$},
every outer y axis line/.append style={black},
every y tick label/.append style={font=\color{black}},
ymin=0,
ymax=1,
ylabel={\Large$\psi_j[k]$},
axis background/.style={fill=white},
axis x line*=bottom,
axis y line*=left
]
\addplot [color=black,solid,line width=3.0pt,forget plot]
  table[row sep=crcr]{%
0	0\\
1	1\\
2	0.398296546942912\\
3	0.0669263087699917\\
4	0.00789822746154749\\
5	0.000768026544166026\\
6	6.60749012283944e-05\\
7	5.22388305243643e-06\\
8	3.8822709390909e-07\\
9	2.75207308277946e-08\\
10	1.87952881653908e-09\\
};
\addplot [color=gray,solid,line width=3.0pt,forget plot]
  table[row sep=crcr]{%
0	0\\
1	0.560211133792258\\
2	1\\
3	0.753064290500951\\
4	0.398296546942912\\
5	0.173578070910036\\
6	0.0669263087699917\\
7	0.0237134923700884\\
8	0.00789822746154749\\
9	0.00250925894699571\\
10	0.000768026544166026\\
};
\addplot [color=black,solid,line width=3.0pt,forget plot]
  table[row sep=crcr]{%
0	0\\
1	0.148245872443102\\
2	0.560211133792258\\
3	0.893109382008472\\
4	1\\
5	0.922590923322294\\
6	0.753064290500951\\
7	0.564873969136242\\
8	0.398296546942912\\
9	0.267881823891104\\
10	0.173578070910036\\
};
\addplot [color=gray,solid,line width=3.0pt,forget plot]
  table[row sep=crcr]{%
0	0\\
1	0.0269620589571623\\
2	0.148245872443102\\
3	0.343871320798658\\
4	0.560211133792258\\
5	0.752006066630379\\
6	0.893109382008472\\
7	0.974730576589928\\
8	1\\
9	0.978581805153552\\
10	0.922590923322294\\
};
\addplot [color=black,solid,line width=3.0pt,forget plot]
  table[row sep=crcr]{%
0	0\\
1	0.00406530638760013\\
2	0.0269620589571623\\
3	0.0754391203703468\\
4	0.148245872443102\\
5	0.240039345518338\\
6	0.343871320798658\\
7	0.452695430010601\\
8	0.560211133792258\\
9	0.661270407216568\\
10	0.752006066630379\\
};
\end{axis}
\end{tikzpicture}%
  \caption{Example of wavelets; the figure shows $\psi_j$ for $J-5\leq j\leq J$ in the Fourier domain. Only the real part is displayed.\label{fig:discrete_wavelets}}
\end{minipage}
\end{center}
\end{figure}
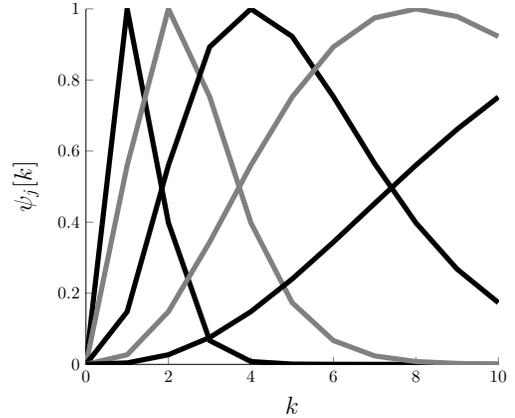

The \textit{wavelet transform} is defined by:
\begin{equation*}
\forall f\in\C^N,\quad Wf=\{f\star\psi_j\}_{0\leq j\leq J}.
\end{equation*}
The problem we consider here consists in reconstructing functions from the modulus of their wavelet transform:
\begin{equation}\label{eq:def_problem}
\mbox{Reconstruct }f\mbox{ from }\{|f\star\psi_j|\}_{0\leq j\leq J}.
\end{equation}
Multiplying a function by a unitary complex does not change the modulus of its wavelet transform, so we only aim at reconstructing functions up to multiplication by a unitary complex, that is \emph{up to a global phase}.

All signals are assumed to be analytic:
\begin{equation}\label{eq:analyticity_discrete}
\hat f[k]=0\quad\mbox{when }N/2<k\leq N-1.
\end{equation}
Equivalently, we could assume the signals to be real but set the $\hat\psi_j[k]$ to zero for $N/2<k\leq N-1$.

\section{Well-posedness of the inverse problem\label{s:well_posedness}}

Developing an algorithm to solve Problem \eqref{eq:def_problem} makes sense only if the problem is well-posed: $f$ must be uniquely determined from $\{|f\star\psi_j|\}_{0\leq j\leq J}$, up to a global phase, and reconstruction must satisfy a form of stability to noise.

Unfortunately, determining whether a given phase retrieval problem is well-posed or not is a notoriously difficult question that, despite decades of works, has been solved only in a handful of cases. In the case of the scalogram (or spectrogram), a theoretical study of the well-posedness, for relatively general wavelets, seems out of reach. However, we can reasonably conjecture that the inverse problem is well-posed for \textit{generic} wavelet families.

Indeed, it has been proven in several cases that reconstruction is unique, as soon as the sampling rate is high enough (for Cauchy wavelets \citep{waldspurger} and for real band-limited wavelets \citep{alaifari}; in the case of the spectrogram, it is known, for general windows, that at least almost all signals are uniquely determined \citep{jaganathan}). By contrast, there is no case in which uniqueness is known not to hold, except for families which exhibit a severe lack of redundancy (Shannon wavelets, for example).

It has been established in \citep{waldspurger} that, when reconstruction is unique, it is also stable, in the sense that the reconstruction operator is continuous (although it may not be uniformly continuous). Moreover, in the two cases where this question of stability could be studied in more detail (the Cauchy wavelet transform in \citep{waldspurger} and the Gabor spectrogram in \citep{alaifari_2}), it was shown that the reconstruction operator actually enjoyed a stronger property of \textit{local stability} (the stability constants depend on the wavelet family; reconstruction is more stable when the wavelets are more redundant, that is, when the overlap between their Fourier supports is large).

In addition to these theoretical works, various numerical studies support our conjecture regarding the well-posedness of the phase retrieval problem. In \citep[IV-C]{nawab}, for example, it is observed that, when the reconstruction algorithm succeeds in finding a signal whose spectrogram matches the target, then this signal is indistinguishable from the correct solution. The experiments in \citep{griffin_lim,virtanen} show that this is also true when the spectrogram has been subject to various deformations, indicating that reconstruction enjoys a form of stability. For short signals, stability is also numerically established in \citep{jaganathan}. Our own experiments in Section \ref{s:numerical_results} are in line with these works: we never observe uniqueness issues, and we empirically demonstrate in Paragraph \ref{ss:stability} that reconstruction satisfies the previously mentioned \textit{local stability} property.

\section{Reformulation of the phase retrieval problem\label{s:reformulation}}

In the first part of this section, we reformulate the phase retrieval problem for the wavelet transform, by introducing two auxiliary wavelet families.

We then describe the two main advantages of this reformulation. First, it allows to propagate the phase information from the low-frequencies to the high ones, and so enables us to perform the reconstruction scale by scale. Second, from this reformulation, we can define a natural objective function to locally optimize approximate solutions. Although non-convex, this function has few local minima; hence, the local optimization algorithm is efficient.

\subsection{Introduction of auxiliary wavelets and reformulation}

\begin{figure}
\begin{center}
\begin{minipage}[b]{0.48\textwidth}
\begin{center}
\input{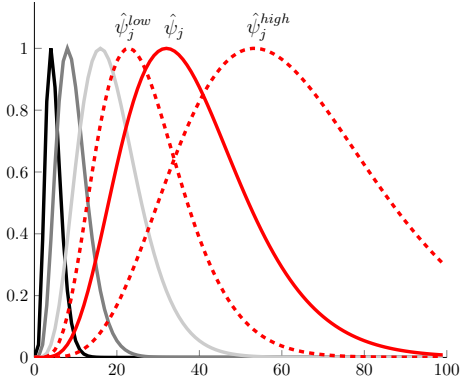}
\caption{$\psi_J,...,\psi_{j+1},\psi_j$ (in the Fourier domain), along with $\psi_j^{low}$ and $\psi_j^{high}$ (dashed lines), renormalized so as to have a maximum value of $1$. \label{fig:psi_low_high}}
\end{center}
\end{minipage}
\end{center}
\end{figure}

Let us fix $r\in]0;1[$, and set $r_j=r^{a^j}$. We define:
\begin{align*}
\forall k=0,...,N-1,\quad\quad
\hat\psi_j^{low}[k]&=\hat\psi_j[k]r_j^k;\\
\hat\psi_j^{high}[k]&=\hat\psi_j[k]r_j^{-k}.
\end{align*}
This definition is illustrated by Figure \ref{fig:psi_low_high}. The wavelet $\hat\psi_j^{low}$ has a lower characteristic frequency than $\hat\psi_j$ and $\hat\psi_j^{high}$ a higher one. The following theorem explains how to rewrite a condition on the modulus of $f\star\psi_j$ as a condition on $f\star\psi_j^{low}$ and $f\star\psi_j^{high}$.
\begin{thm}\label{thm:reformulation}
Let $j\in\{0,...,J\}$ and $g_j\in(\R^+)^N$ be fixed. Let $Q_j$ be the function whose Fourier transform is
\begin{gather}\label{eq:def_Qj}
\hat Q_j[k] = r_j^k \widehat{g_j^2}[k],\\
\forall k=\left\lfloor\frac{N}{2}\right\rfloor-N+1,...,\left\lfloor\frac{N}{2}\right\rfloor.\nonumber
\end{gather}
For any $f\in\C^N$ satisfying the analyticity condition \eqref{eq:analyticity_discrete}, the following two properties are equivalent:
\begin{enumerate}
\item $|f\star\psi_j|=g_j$;
\item $(f\star\psi_j^{low})\overline{(f\star\psi_j^{high})} = Q_j$.
\end{enumerate}
\end{thm}

\begin{proof}
The proof consists in showing that the second equality is the analytic extension of the first one, in a sense that will be precisely defined.

For any function $h:\{0,...,N-1\}\to\C$, let $P(h)$ be
\begin{equation*}
\forall z\in\C,\quad\quad
P(h)(z)=\underset{k=\left\lfloor\frac{N}{2}\right\rfloor-N+1}{\overset{\left\lfloor\frac{N}{2}\right\rfloor}{\sum}}\hat h[k]z^k.
\end{equation*}
Up to a change of coordinates, $P({|f\star\psi_j|^2})$ and $P({g_j^2})$ are equal to $P({(f\star\psi_j^{low})\overline{(f\star\psi_j^{high})}})$ and $P({Q_j})$:
\begin{lem}\label{lem:extension}
For any $f$ satisfying the analyticity condition \eqref{eq:analyticity_discrete}, and for any $z\in\C$:
\begin{align*}
&P({|f\star\psi_j|^2})(r_jz)=P({(f\star\psi_j^{low})\overline{(f\star\psi_j^{high})}})(z)\\
\mbox{and}\quad&P({g_j^2})(r_jz)=P({Q_j})(z).
\end{align*}
\end{lem}
This lemma is proved in the appendix \ref{app:extension}. It implies the result because then,
\begin{align*}
&|f\star\psi_j|=g_j\\
\iff& |f\star\psi_j|^2=g_j^2\\
\iff&\forall z, P({|f\star\psi_j|^2})(z)=P({g_j^2})(z)\\
\iff&\forall z, P({|f\star\psi_j|^2})(r_jz)=P({g_j^2})(r_jz)\\
\iff&\forall z,P({(f\star\psi_j^{low})\overline{(f\star\psi_j^{high})}})(z)=P({Q_j})(z)\\
\iff& (f\star\psi_j^{low})\overline{(f\star\psi_j^{high})}=Q_j.
\end{align*}
\end{proof}

By applying simultaneously Theorem \ref{thm:reformulation} to all indexes $j$, we can reformulate the phase retrieval problem $|f\star\psi_j|=g_j,\forall j$ in terms of the $f\star\psi_j^{low}$'s and $f\star\psi_j^{high}$'s.
\begin{cor}[Reformulation of the phase retrieval problem]
Let $(g_j)_{0\leq j\leq J}$ be a family of signals in $(\R^+)^N$. For each $j$, let $Q_j$ be defined as in \eqref{eq:def_Qj}. Then the following two problems are equivalent.
\begin{align*}
&\begin{minipage}{0.42\textwidth}
\begin{center}
Find $f$ satisfying \eqref{eq:analyticity_discrete} such that:
\begin{equation*}
\forall j,\quad
|f\star\psi_j|=g_j.
\end{equation*}
\end{center}
\end{minipage}\\
&\\
\iff&
\begin{minipage}{0.42\textwidth}
\begin{center}
Find $f$ satisfying \eqref{eq:analyticity_discrete} such that:
\begin{equation}\label{eq:reformulation}
\forall j,\quad
(f\star\psi_j^{low})\overline{(f\star\psi_j^{high})} = Q_j.
\end{equation}
\end{center}
\end{minipage}
\end{align*}

\end{cor}

\subsection{Phase propagation across scales\label{ss:phase_propagation}}

This new formulation yields a natural multiscale reconstruction algorithm, in which one reconstructs $f$ frequency band by frequency band, starting from the low frequencies.

Indeed, once $f\star\psi_J,...,f\star\psi_{j+1}$ have been reconstructed, it is possible to estimate $f\star\psi_j^{low}$ by deconvolution. This deconvolution is stable to noise because, if $r$ is sufficiently small, then the frequency band covered by $\psi_j^{low}$ is almost included in the frequency range covered by $\psi_J,...,\psi_{j+1}$ (see Figure \ref{fig:psi_low_high}). From $f\star\psi_j^{low}$, one can reconstruct $f\star\psi_j^{high}$, using \eqref{eq:reformulation}:
\begin{equation}\label{eq:rec_high}
f\star\psi_j^{high}=\frac{\overline{Q}_j}{\overline{f\star\psi_j^{low}}}.
\end{equation}
Finally, one reconstructs $f\star\psi_j$ from $f\star\psi_j^{high}$ and $f\star\psi_j^{low}$.

The classical formulation of the phase retrieval problem does not allow the conception of such a multiscale algorithm. Indeed, from $f\star\psi_J,...,f\star\psi_{j+1}$, it is not possible to directly estimate $f\star\psi_j$: it would require performing a highly unstable deconvolution. The introduction of the two auxiliary wavelet families is essential.

\subsection{Local optimization of approximate solutions\label{ss:local_optimization}}

From the reformulation \eqref{eq:reformulation}, we can define a natural objective function for the local optimization of approximate solutions to the phase retrieval problem. Despite its non-convexity, this function empirically seems to have less local minima than more classical objective functions. It is thus easier to globally minimize.

\nl
The objective function has $2J+3$ variables: $(h_j^{low})_{0\leq j\leq J}$, $(h_j^{high})_{0\leq j\leq J}$ and the analytic function $f$. The intuition is that $f$ is the signal we aim at reconstructing and the $h_j^{low},h_j^{high}$ correspond to the $f\star\psi_j^{low}$'s and $f\star\psi_j^{high}$'s. The objective function is
\begin{align}
&\mbox{obj}(h_J^{low},...,h_0^{low},h_J^{high},...,h_0^{high},f)\nonumber\\
=&\hskip 0.7cm
\underset{j=0}{\overset{J}{\sum}}\,||h_j^{low}\overline{h_j^{high}}-Q_j||_2^2\nonumber\\
&+\lambda\,\underset{j=0}{\overset{J}{\sum}}\left(||f\star\psi_j^{low}-h_j^{low}||_2^2+||f\star\psi_j^{high}-h_j^{high}||_2^2\right).\label{eq:objective}
\end{align}
We additionally constrain the variables $(h_j^{low})_{0\leq j\leq J}$ and $(h_j^{high})_{0\leq j\leq J}$ to satisfy
\begin{equation}\label{eq:additional_constraint}
\forall j=0,...,J-1,
\quad\quad h_j^{low}\star\psi_{j+1}^{high}=h_{j+1}^{high}\star\psi_j^{low}.
\end{equation}
The first term of the objective ensures that the equalities \eqref{eq:reformulation} are satisfied, while the second term and the additional constraint \eqref{eq:additional_constraint} enforce the fact that the $h_j^{low}$'s and $h_j^{high}$'s must be the wavelet transforms of the \emph{same} function $f$.

The parameter $\lambda$ is a positive real number. In our implementation, we choose a small $\lambda$, so that the first term dominates over the second one.
\nl

If we had not introduced auxiliary wavelets, an objective function that we could have considered would have been
\begin{align}
&\mbox{classical obj}(h_J,...,h_0,f)\nonumber\\
&\quad\quad=\,
\underset{j=0}{\overset{J}{\sum}}\,||\,|h_j|^2-g_j^2||_2^2
+\lambda\,\underset{j=0}{\overset{J}{\sum}}||f\star\psi_j-h_j||_2^2.\label{eq:classical_objective}
\end{align}
However, this function seems to have more local minima, and, when minimized with a first-order method, does not produce reconstruction results of the same quality as Equation \eqref{eq:objective}. This phenomenon is illustrated by Figure \ref{fig:critical_points}: starting from random initializations, we have used a first-order method to minimize one or the other objective function, and have recorded how many times we have hit a local minimum, instead of the global one. Provided that $r$ is well chosen, we see that our new formulation performs better than the classical one.

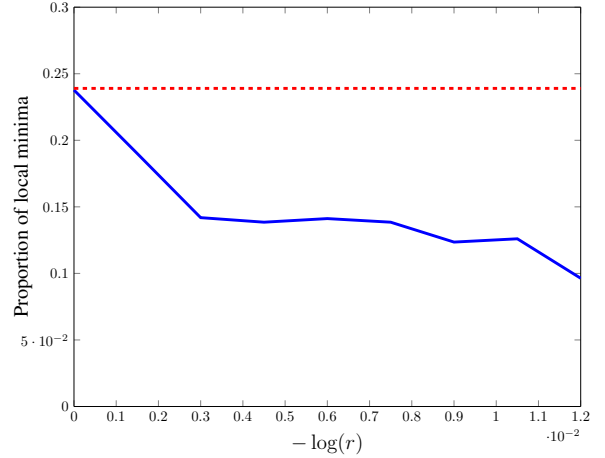
\begin{figure}
\begin{center}
\begin{minipage}[b]{0.48\textwidth}
\begin{center}
%
%
\begin{tikzpicture}[scale=0.55]

\begin{axis}[%
width=4.822in,
height=3.803in,
at={(0.809in,0.513in)},
scale only axis,
separate axis lines,
every outer x axis line/.append style={black},
every x tick label/.append style={font=\color{black}},
xmin=0,
xmax=0.012,
every outer y axis line/.append style={black},
every y tick label/.append style={font=\color{black}},
ymin=0,
ymax=0.3,
axis background/.style={fill=white},
xlabel={\Large $-\log(r)$},
ylabel={\Large Proportion of local minima}
]
\addplot [color=blue,solid,line width=2.0pt,forget plot]
  table[row sep=crcr]{%
0	0.237764705882353\\
0.0015	0.19\\
0.003	0.141882352941176\\
0.0045	0.138470588235294\\
0.006	0.141176470588235\\
0.0075	0.138470588235294\\
0.009	0.123529411764706\\
0.0105	0.126\\
0.012	0.0963529411764706\\
};
\addplot [color=red,dashed,line width=2.0pt,forget plot]
  table[row sep=crcr]{%
0	0.239\\
0.0015	0.239\\
0.003	0.239\\
0.0045	0.239\\
0.006	0.239\\
0.0075	0.239\\
0.009	0.239\\
0.0105	0.239\\
0.012	0.239\\
};
\end{axis}
\end{tikzpicture}%
\caption{Proportion of local minima, for the objective functions \eqref{eq:objective} (solid line) and \eqref{eq:classical_objective} (dashed line), as a function of $-\log(r)$. The second proportion is independent from $r$. Signals were realizations of Gaussian random processes, with $500$ points; the wavelets were dyadic Morlet and there was no noise.\label{fig:critical_points}}
\end{center}
\end{minipage}
\end{center}
\end{figure}

\subsection{Analysis for Cauchy wavelets\label{ss:analysis_cauchy}}

In this section, we give a partial theoretical justification of the fact that introducing auxiliary wavelets reduces the number of local minima. Indeed, in the case of Cauchy wavelets and in the absence of noise, we show that a modified version of our objective function \eqref{eq:objective} has no spurious critical point, in the limit $\lambda\to 0^+$.

The idea is that we have an explicit expression for the manifold of stable critical points of the first part of the objective function. This expression is sufficiently simple so that we can easily compute the critical points of the (modified) second part, when restricted to that manifold.

Cauchy wavelets are parametrized by numbers $p_1,p_2>0$:
\begin{equation*}
\hat \psi(x)=x^{p_1}e^{-p_2 x}\quad\quad \forall x\in\R^+.
\end{equation*}

We define $(\psi_j^{low})_{0\leq j\leq J},(\psi_j^{high})_{0\leq j\leq J}$ with the following choice for $r$:
\begin{equation*}
r=e^{-p_2\left(\frac{a-1}{a+1}\right)}.
\end{equation*}
With this choice, we see that
\begin{equation*}
\psi_{j+1}^{high}=a^{p_1}\psi_j^{low}\quad\quad j=0,\dots,J-1,
\end{equation*}
and we naturally replace Condition \eqref{eq:additional_constraint} by
\begin{equation}\label{eq:additional_constraint_mod}
\forall j=0,\dots,J-1,\quad\quad
h_{j+1}^{high}=a^{p_1}h_j^{low}.
\end{equation}
In the limit $\lambda\to 0^+$, the critical points of the objective function are included in the critical points of the first term
\begin{align*}
\mathrm{obj}_1(h_J^{low},&\dots,h_0^{low},h_J^{high},\dots,h_0^{high})\\
&=\sum_{j=0}^J||h_j^{low}\overline{h_j^{high}}-Q_j||^2_2.
\end{align*}
We say that $\mathrm{obj}_1$ is \textit{strict saddle} at a critical point if its Hessian has at least one strictly negative eigenvalue. Strict saddle critical points can be escaped from by first-order methods \citep{ge,lee}, so their presence is not too harmful.

The critical points that are not strict saddle have a simple structure, described by the following proposition (proven in Appendix \ref{s:critical_points}).
\begin{prop}\label{prop:critical_points_1}
Let $f_{rec}$ be the function that we aim at reconstructing. We assume that, for any $j=0,\dots,J$, the functions $f_{rec}\star\psi_j^{low}$ and $f_{rec}\star\psi_j^{high}$ have no zero entry. Then, on the vector space defined by Equation \eqref{eq:additional_constraint_mod}, the critical points of $\mathrm{obj}_1$ that are not strict saddle are
\begin{align*}
\Big\{
\Big(&\gamma (f_{rec}\star \psi_{J}^{low}),\frac{1}{\overline{\gamma}}(f_{rec}\star \psi_{J-1}^{low}),\gamma (f_{rec}\star \psi_{J-2}^{low}),\dots,\\
&\frac{1}{\overline{\gamma}}(f_{rec}\star\psi_J^{high}),\gamma (f_{rec}\star\psi_{J-1}^{high}),\frac{1}{\overline{\gamma}}(f_{rec}\star\psi_{J-2}^{high}),\dots\Big),\\
&\quad\quad\mbox{such that } \gamma\in (\C-\{0\})^N\Big\}
\quad\overset{def}{=}\quad\mathcal{C}_1
.
\end{align*}
\end{prop}
The second member of the objective function \eqref{eq:objective} is difficult to directly analyze, but a modified version is approachable. We define
\begin{align*}
&\mathrm{obj}_2(h_J^{low},\dots,h_0^{high},f)\\
&=\sum_{j\equiv J[2]}||f\star\psi_j^{low}-h_j^{low}||_2^2+\sum_{j\not\equiv J[2]}||f\star\psi_j^{high}-h_j^{high}||_2^2.
\end{align*}
In the limit $\lambda\to 0^+$, the minimization problem becomes:
\begin{align}
\min\quad &\mathrm{obj}_2(h_J^{low},\dots,h_0^{high},f)\nonumber\\
\mbox{s.t.} \quad&(h_J^{low},\dots,h_0^{high},f)\in\mathcal{C}_1\times \mathcal{A},
\label{eq:objective_mod}
\end{align}
where $\mathcal{A}$ is the set of signals satisfying the analyticity conditions \eqref{eq:analyticity_discrete}.

The following theorem (proven in Appendix \ref{s:critical_points}), ensures that all critical points of this problem are the correct solution, up to affine scaling.
\begin{thm}\label{thm:critical_points_2}
We assume $J\geq 2$, and keep the hypotheses of Proposition \ref{prop:critical_points_1}.

Let $(h_J^{low},\dots,h_0^{low},h_J^{high},\dots,h_0^{high},f)$ be a critical point of $\mathrm{obj}_2$ on the real manifold $\mathcal{C}_1\times \mathcal{A}$. There exists $\alpha\in\C-\{0\},b\in\C$ such that
\begin{equation*}
f=\alpha f_{rec}+b.
\end{equation*}
\end{thm}

The modified problem \eqref{eq:objective_mod} is however less stable to noise that the original version \eqref{eq:objective}. In applications, we thus stick to the original formulation.


\section{Description of the algorithm\label{s:algorithm}}

In this section, we describe our implementation of the multiscale reconstruction algorithm introduced in Section \ref{s:reformulation}. We explain the general organization in Paragraph \ref{ss:organization}. We then describe our exhaustive search method for solving phase retrieval problems of very small size (paragraph \ref{ss:exhaustive_search}), which our algorithm uses to initialize the multiscale reconstruction. In Paragraph \ref{ss:error_correction}, we describe an additional multigrid correction step.

\subsection{Organization of the algorithm\label{ss:organization}}

We start by reconstructing $f\star\psi_J$ from $|f\star\psi_J|$ and $|f\star\psi_{J-1}|$. We use an exhaustive search method, described in the next paragraph \ref{ss:exhaustive_search}, which takes advantage of the fact that $\hat\psi_J$ and $\hat\psi_{J-1}$ have very small supports.

\nl
We then reconstruct the components of the wavelet transform scale by scale, as described in Section \ref{s:reformulation}.

At each scale, we reconstruct $f\star\psi_j^{low}$ by propagating the phase information coming from $f\star\psi_J,...,f\star\psi_{j+1}$ (as explained in Paragraph \ref{ss:phase_propagation}). This estimation can be imprecise, so we refine it by local optimization, using the objective function defined in Paragraph \ref{ss:local_optimization}, from which we drop all the terms with higher scales than $j$. The local optimization algorithm we use in the implementation is L-BFGS \citep{nocedal}, a low-memory approximation of a second order method.

We then reconstruct $f\star\psi_j^{high}$ by the equation \eqref{eq:rec_high}.

\nl
At the end of the reconstruction, we run a few steps of the classical Gerchberg-Saxton algorithm to further refine the estimation.

\nl
The pseudo-code \ref{alg:overview} summarizes the structure of the implementation.

\begin{algorithm}
\caption{overview of the algorithm\label{alg:overview}}
\begin{algorithmic}[1]
\REQUIRE $\{|f\star\psi_j|\}_{0\leq j\leq J}$
\STATE Initialization: reconstruct $f\star\psi_J$ by exhaustive search.
\FORALL{$j=J:(-1):0$}
\STATE Estimate $f\star\psi_j^{low}$ by phase propagation.
\STATE Refine the values of $f\star\psi_J^{low},...,f\star\psi_j^{low},f\star\psi_J^{high},...,f\star\psi_{j+1}^{high}$ by local optimization.
\STATE Do an error correction step.
\STATE Refine again.
\STATE Compute $f\star\psi_j^{high}$ by $f\star\psi_j^{high}=\overline{Q}_j/\overline{f\star\psi_j^{low}}$.
\ENDFOR
\STATE Compute $f$.
\STATE Refine $f$ with Gerchberg-Saxton.
\ENSURE $f$
\end{algorithmic}
\end{algorithm}

\subsection{Reconstruction by exhaustive search for small problems\label{ss:exhaustive_search}}

In this paragraph, we explain how to reconstruct $f\star\psi_j$ from $|f\star\psi_j|$ and $|f\star\psi_{j-1}|$ by exhaustive search, in the case where the support of $\hat\psi_j$ and $\hat\psi_{j-1}$ is small.

This is the method we use to initialize our multiscale algorithm. It is also useful for the multigrid error correction step described in the next paragraph \ref{ss:error_correction}.

\begin{lem}\label{lem:rec_small}
Let $m\in\R^N$ and $K\in\N^*$ be fixed. We consider the problem
\begin{align*}
\mbox{Find }g\in\C^N\mbox{ s.t. }&|g|=m\\
\mbox{and }&\mbox{Supp}(\hat g)\subset\{1,...,K\}.
\end{align*}
This problem has at most $2^{K-1}$ solutions, up to a global phase, and there exist a simple algorithm which, from $m$ and $K$, returns the list of all possible solutions.
\end{lem}
\begin{proof}
This lemma is a consequence of classical results about the phase retrieval problem for the Fourier transform. It can for example be derived from \citep{hayes}. We give a proof in Appendix \ref{app:rec_small}.
\end{proof}
We apply this lemma to $m=|f\star\psi_j|$ and $|f\star\psi_{j-1}|$, and construct the lists of all possible $f\star\psi_j$'s and of all possible $f\star\psi_{j-1}$'s. The true $f\star\psi_j$ and $f\star\psi_{j-1}$ are the only pair in these two lists which satisfy the equality
\begin{equation*}
(f\star\psi_j)\star\psi_{j-1}=(f\star\psi_{j-1})\star\psi_j.
\end{equation*}
This solves the problem.

\nl
The number of elements in the lists is exponential in the size of the supports of $\hat\psi_j$ and $\hat\psi_{j-1}$, so this algorithm has a prohibitive complexity when the supports become large. Otherwise, our numerical experiments show that it works well.

\subsection{Error correction\label{ss:error_correction}}

When the modulus are noisy, there can be errors during the phase propagation step. The local optimization generally corrects them, if run for a sufficient amount of time, but, for the case where some errors are left, we add, at each scale, a multigrid error correction step. This step does not totally remove the errors but greatly reduces their amplitude.

\subsubsection{Principle}

First, we determine the values of $n$ for which $f\star\psi_j^{low}[n]$ and $f\star\psi_{j+1}^{high}[n]$ seem to have been incorrectly reconstructed. We use the fact that $f\star\psi_j^{low}$ and $f\star\psi_{j+1}^{high}$ must satisfy
\begin{equation*}
(f\star\psi_j^{low})\star\psi_{j+1}^{high}=(f\star\psi_{j+1}^{high})\star\psi_j^{low}.
\end{equation*}
The points where this equality does not hold provide a good estimation of the places where the values of $f\star\psi_j^{low}$ and $f\star\psi_{j+1}^{high}$ are erroneous.

\nl
\begin{figure}
\begin{center}
\begin{minipage}[b]{0.4\textwidth}
\begin{center}
\input{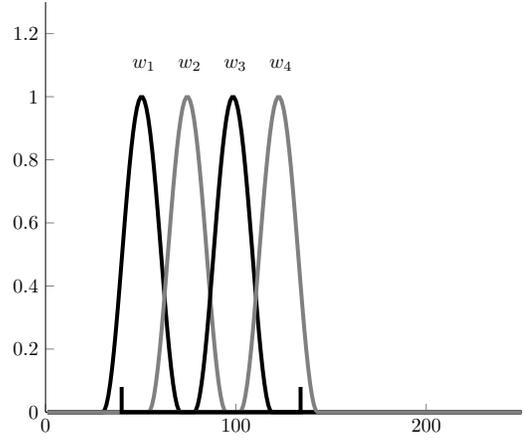}
\caption{Four window signals, whose supports cover the interval in which errors have been detected\label{fig:windows}}
\end{center}
\end{minipage}
\end{center}
\end{figure}

We then construct a set of smooth ``windows'' $w_1,...,w_S$, whose supports cover the interval on which errors have been found (see figure \ref{fig:windows}), such that each window has a small support. For each $s$, we reconstruct $(f\star\psi_j^{low}).w_s$ and $(f\star\psi_{j+1}^{high}).w_s$, by expressing these functions as the solutions to phase retrieval problems of small size, which we can solve by the exhaustive search method described in Paragraph \ref{ss:exhaustive_search}.

If $w_s$ is sufficiently smooth so that it can be considered constant on the support of $\psi_j$, we have, for $k=0,\dots,N-1$,
\begin{align*}
\left|(f.w_s)\star\psi_j[k]\right|
&=\left|\sum_{n=0}^{N-1}\psi_j[n]w_s[k-n]f[k-n] \right|\\
&\approx\left|\sum_{n=0}^{N-1}\psi_j[n]w_s[k]f[k-n] \right|\\
&=w_s[k]\left|f\star\psi_j[k]\right|.
\end{align*}
The same reasoning applies for $\psi_{j+1}$:
\begin{gather*}
|(f.w_s)\star\psi_j|\approx w_s|f\star\psi_j|;\\
|(f.w_s)\star\psi_{j+1}|\approx w_s|f\star\psi_{j+1}|.
\end{gather*}
The wavelets $\psi_j$ and $\psi_{j+1}$ have a small support in the Fourier domain, if we truncate them to the support of $w_s$, so we can solve this problem by exhaustive search, and reconstruct $(f.w_s)\star\psi_j$ and $(f.w_s)\star\psi_{j+1}$.

From $(f.w_s)\star\psi_j$ and $(f.w_s)\star\psi_{j+1}$, we reconstruct $(f\star\psi_j^{low}).w_s\approx (f.w_s)\star\psi_j^{low}$ and $(f\star\psi_{j+1}^{high}).w_s\approx (f.w_s)\star\psi_{j+1}^{high}$ by deconvolution.

There is a trade-off to reach in the size of the support of the $w_s$: the approximation error in $|(f.w_s)\star\psi_j|\approx w_s|f\star\psi_j|$ is inversely proportional to this size, but the complexity of the exhaustive search becomes prohibitive when the size is too large. In our implementation, we adjust it so that the Fourier support of $\psi_j$ (truncated to the support of $w_s$) is of size $12$.

\subsubsection{Usefulness of the error correction step}

The error correction step does not perfectly correct the errors, but greatly reduces the amplitude of large ones.

Figure \ref{fig:reestimation_example} shows an example of this phenomenon. It deals with the reconstruction of a difficult audio signal, representing a human voice saying ``I'm sorry''. Figure \ref{fig:reestimation_example_a} shows $f\star\psi_7^{low}$ after the multiscale reconstruction at scale $7$, but before the error correction step. The reconstruction presents large errors. Figure \ref{fig:reestimation_example_b} shows the value after the error correction step. It is still not perfect but much closer to the ground truth.

\begin{figure}
\centering
\subfloat[]{\includegraphics[width=0.4\textwidth]{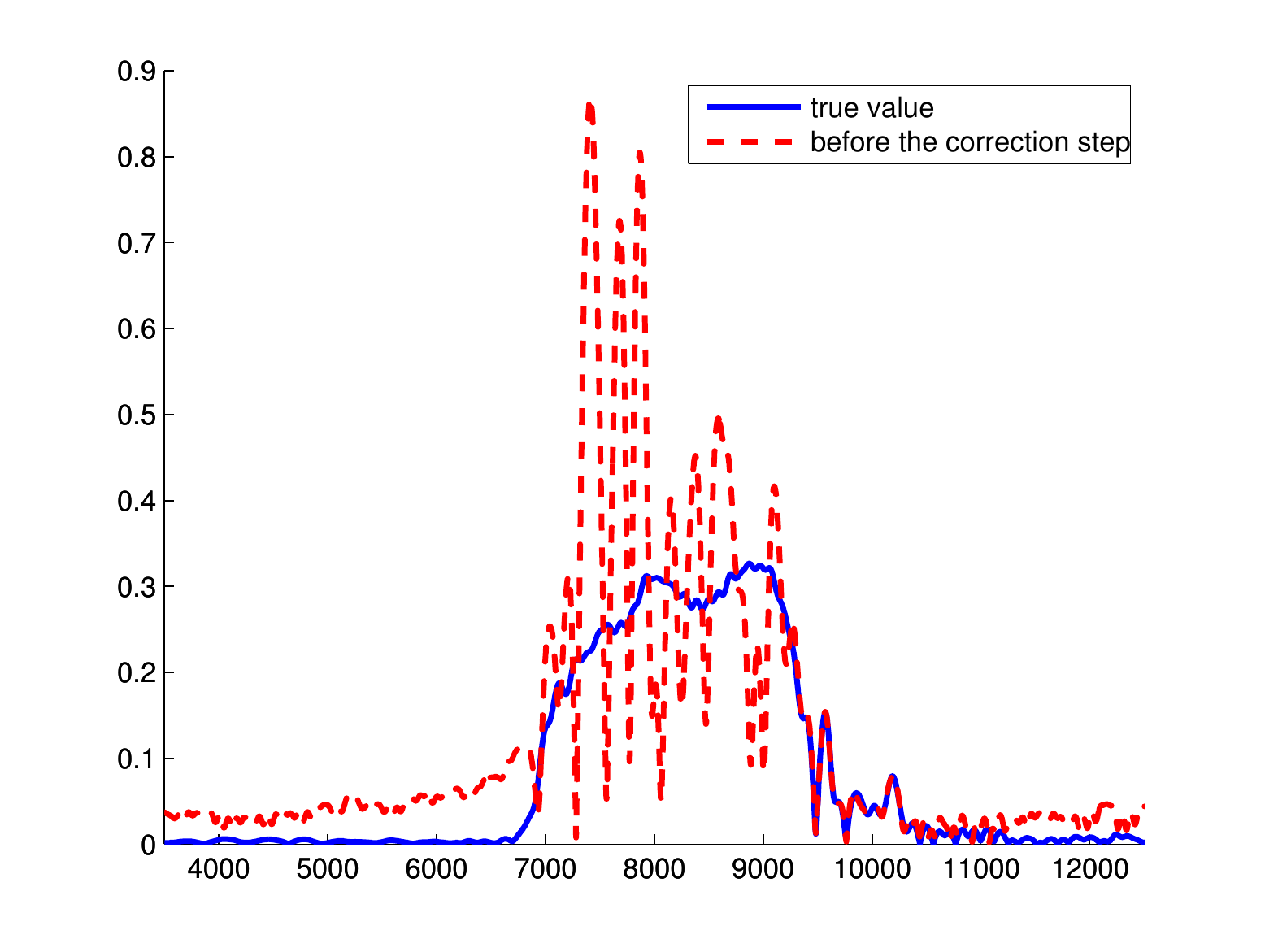}\label{fig:reestimation_example_a}}
\hfil
\subfloat[]{\includegraphics[width=0.4\textwidth]{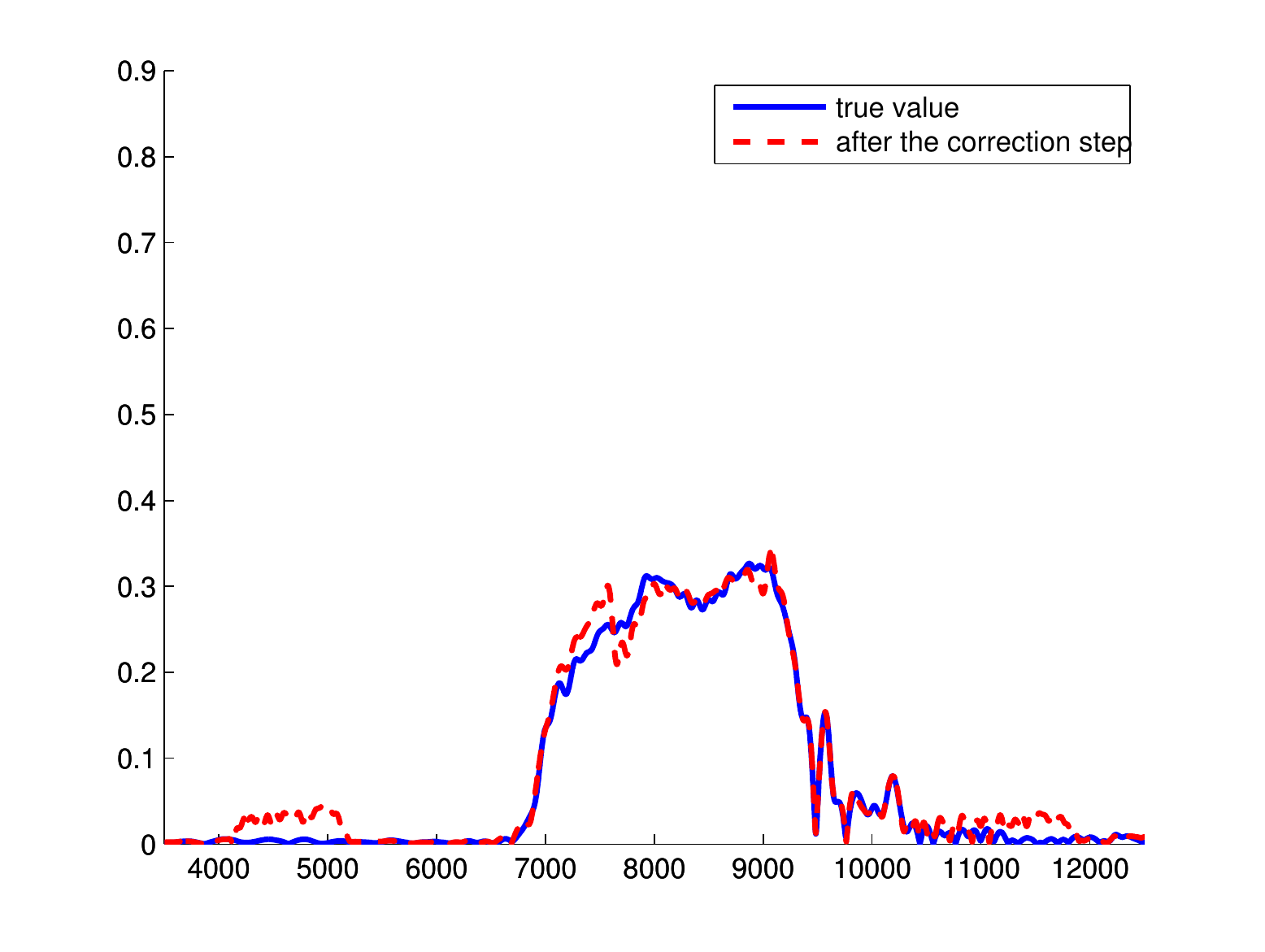}\label{fig:reestimation_example_b}}
\caption{For an audio signal, the reconstructed value of $f\star\psi^{low}_7$ at the scale $7$ of the multiscale algorithm, in modulus (dashed line); the solid line represents the ground true. (a) Before the error correction step (b) After the error correction step
\label{fig:reestimation_example}}
\end{figure}

\nl
So the error correction step must be used when large errors are susceptible to occur, and turned off otherwise: it makes the algorithm faster without reducing its precision.

Figure \ref{fig:reestimation} illustrates this affirmation by showing the mean reconstruction error for the same audio signal as previously. When $200$ iterations only are allowed at each local optimization step, there are large errors in the multiscale reconstruction; the error correction step significantly reduces the reconstruction error. When $2000$ iterations are allowed, all the large errors can be corrected during the local optimization steps and the error correction step is not useful.

\begin{figure}
\centering
\subfloat[]{\includegraphics[width=0.4\textwidth]{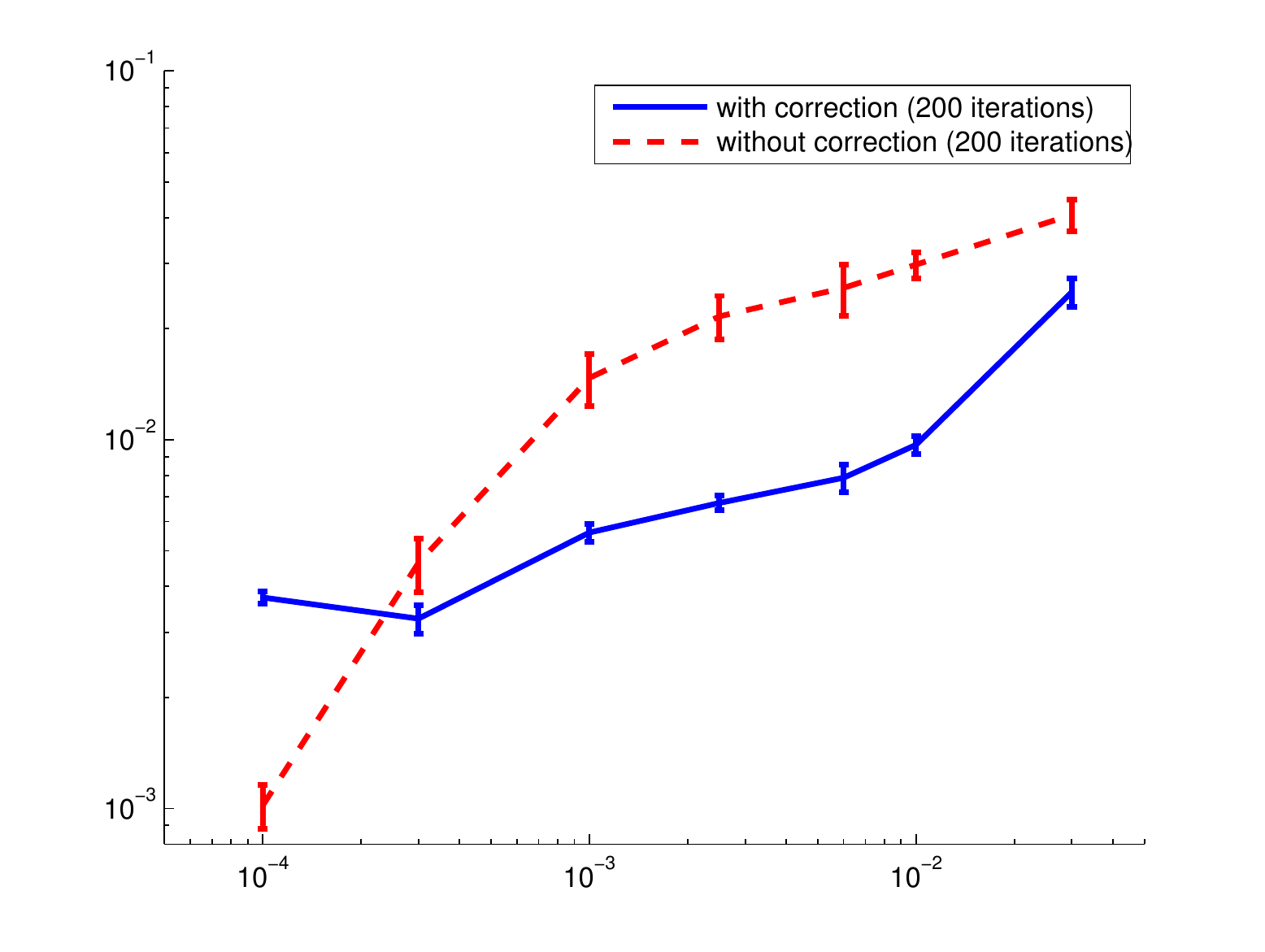}}
\hfil
\subfloat[]{\includegraphics[width=0.4\textwidth]{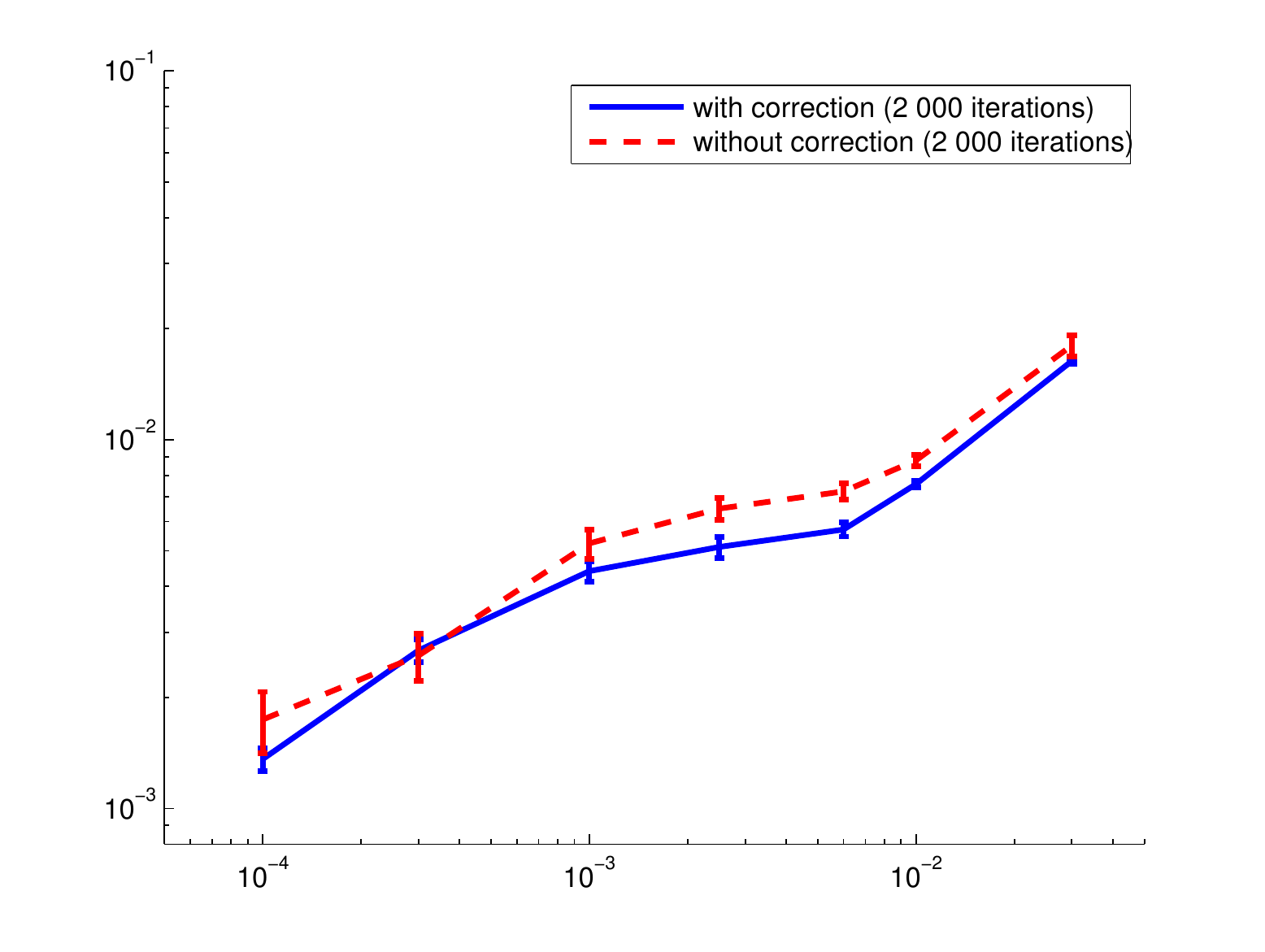}}
\caption{Mean reconstruction error \eqref{eq:reconstruction_error} as a function of the noise, for an audio signal representing a human voice. (a) Maximal number of iterations per local optimization step equal to $200$ (b) Maximal number equal to $2000$.
\label{fig:reestimation}}
\end{figure}

\section{Multiscale versus non-multiscale\label{s:multiscale}}

Our reconstruction algorithm has very good reconstruction performances, mainly because it uses the reformulation of the phase retrieval problem introduced in Section \ref{s:reformulation}. However, the quality of its results is also due to its multiscale structure. It is indeed known that, for the reconstruction of functions from their spectrogram or scalogram, multiscale algorithms perform better than non-multiscale ones \citep{bouvrie,bruna_phd}.

In this section, we propose two justifications for this phenomenon (paragraph \ref{ss:advantages}). We then introduce a multiscale version of the classical Gerchberg-Saxton algorithm, and numerically verify that it yields better reconstruction results than the usual non-multiscale version (paragraph \ref{ss:multiscale_GS}).

\subsection{Advantages of the multiscale reconstruction\label{ss:advantages}}

At least two factors can explain the superiority of multiscale methods, where the $f\star\psi_j$'s are reconstructed one by one, and not all at the same time.

First, they can partially remedy the possible ill-conditioning of the problem. In particular, if the $f\star\psi_j$'s have very different norms, then a non-multiscale algorithm will be more sensitive to the components with a high norm. It may neglect the information given by $|f\star\psi_j|$, for the values of $j$ such that this function has a small norm. With a multiscale algorithm where all the $|f\star\psi_j|$'s are successively considered, this happens less frequently.

Second, iterative algorithms, like Gerchberg-Saxton, are very sensitive to the choice of their starting point (hence the care given to their initialization in the literature \citep{netrapalli,candes_wirtinger}). If all the components are reconstructed at the same time and the starting point is randomly chosen, the algorithm almost never converges towards the correct solution: it gets stuck in a local minima. In a multiscale algorithm, the starting point at each scale can be chosen so as to be consistent with the values reconstructed at lower scales; it yields much better results.

\subsection{Multiscale Gerchberg-Saxton\label{ss:multiscale_GS}}

To justify the efficiency of the multiscale approach, we introduce a multiscale version of the classical Gerchberg-Saxton algorithm \citep{gerchberg} (by alternate projections) and compare its performances with the non-multiscale algorithm.
\nl

The multiscale algorithm reconstructs $f\star\psi_J$ by exhaustive search (paragraph \ref{ss:exhaustive_search}).

Then, for each $j$, once $f\star\psi_{J},...,f\star\psi_{j+1}$ are reconstructed, an initial guess for $f\star\psi_j$ is computed by deconvolution. The frequencies of $f\star\psi_j$ for which the deconvolution is too unstable are set to zero. The regular Gerchberg-Saxton algorithm is then simultaneously applied to $f\star\psi_J,...,f\star\psi_j$.

\nl
We test this algorithm on realizations of Gaussian random processes (see Section \ref{ss:performances} for details), of various lengths. On Figure \ref{fig:advantages}, we plot the mean reconstruction error obtained with the regular Gerchberg-Saxton algorithm and the error obtained with the multiscale version (see Paragraph \ref{ss:setting} for the definition of the reconstruction error).

None of the algorithms is able to perfectly reconstruct the signals, in particular when their size increases. However, the multiscale algorithm clearly yields better results, with a mean error approximately twice smaller.
\begin{figure}
\centering
\includegraphics[width=0.4\textwidth]{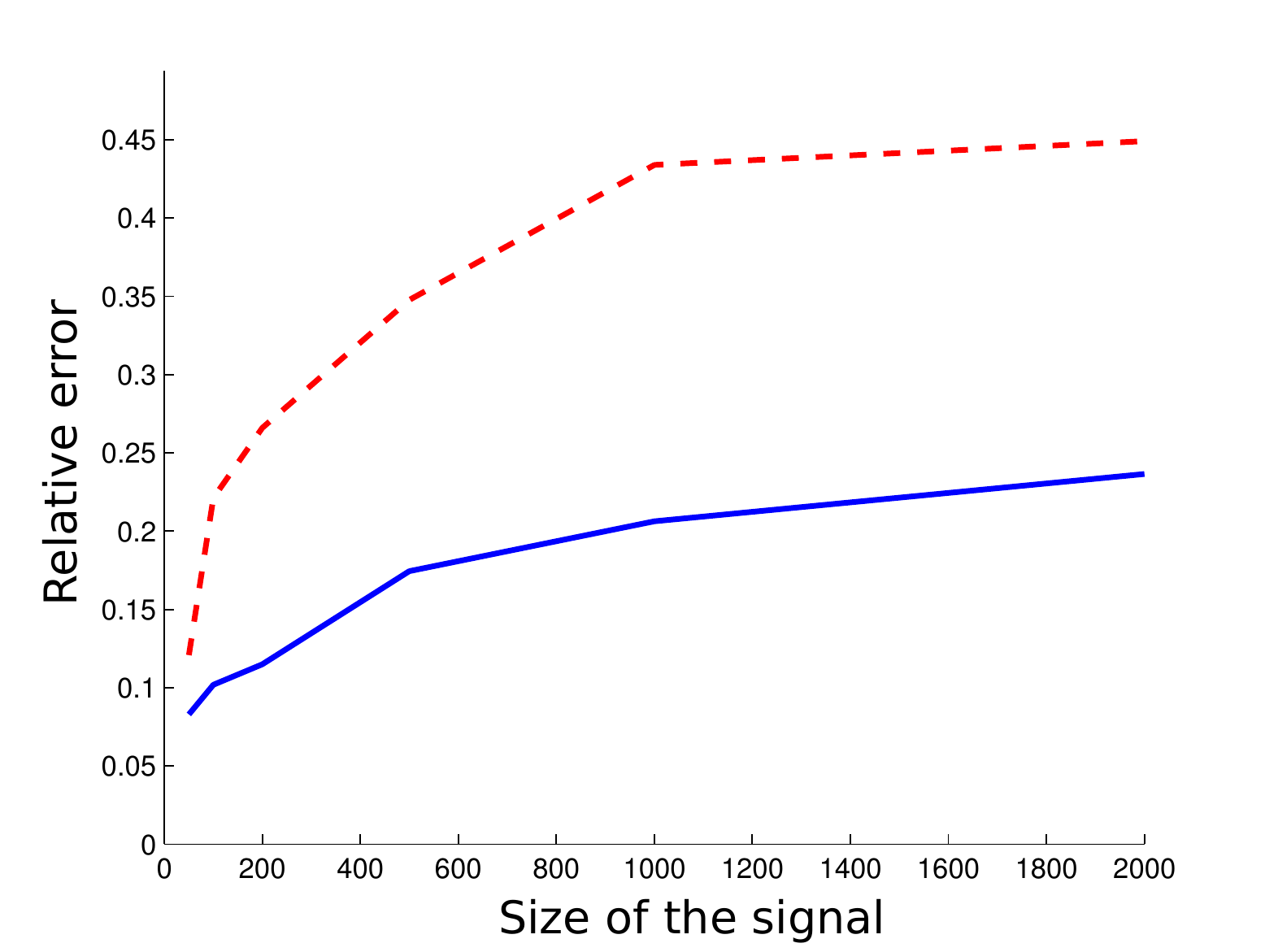}
\caption{Mean reconstruction error, as a function of the size of the signal; the solid blue line corresponds to the multiscale algorithm and the dashed red one to the non-multiscale one.\label{fig:advantages}}
\end{figure}

\section{Numerical results\label{s:numerical_results}}

In this section, we describe the behavior of our algorithm. We compare it with Gerchberg-Saxton and with \emph{PhaseLift}. We show that it is much more precise than Gerchberg-Saxton. It is comparable with \emph{PhaseLift} in terms of precision, but significantly faster, so it allows to reconstruct larger signals.

\nl
The performances strongly depend on the type of signals we consider. The main source of difficulty for our algorithm is the presence of small values in the wavelet transform, especially in the low frequencies.

Indeed, the reconstruction of $f\star\psi_j^{high}$ by Equation \eqref{eq:rec_high} involves a division by $f\star\psi_j^{low}$. When $f\star\psi_j^{low}$ has small values, this operation is unstable and induces errors.

As we will see in Section \ref{ss:stability}, the signals whose wavelet transform has many small values are also the signals for which the phase retrieval problem is the least stable (in the sense that two functions can have wavelet transforms almost equal in modulus without being close in $l^2$-norm). This suggests that this class of functions is intrinsically the most difficult to reconstruct; it is not an artifact of our algorithm.

\nl

We describe our experimental setting in Paragraph \ref{ss:setting}. In Paragraph \ref{ss:performances}, we give detailed numerical results for various types of signals. In Paragraph \ref{ss:stability}, we use our algorithm to investigate the stability to noise of the underlying phase retrieval problem. Finally, in Paragraph \ref{ss:influence}, we study the influence of various parameters on the quality of the reconstruction.

\subsection{Experimental setting\label{ss:setting}}

At each reconstruction trial, we choose a signal $f$ and compute its wavelet transform $\{|f\star\psi_j|\}_{0\leq j\leq J}$. We corrupt it with a random noise $n_j$:
\begin{equation}\label{eq:add_noise}
h_j = |f\star\psi_j| + n_j.
\end{equation}
We measure the amplitude of the noise in $l^2$-norm, relatively to the $l^2$-norm of the wavelet transform:
\begin{equation}\label{eq:amount_of_noise}
\mbox{amount of noise}= \frac{\sqrt{\underset{j}{\sum}||n_j||_2^2}}{\sqrt{\underset{j}{\sum}||f\star\psi_j||_2^2}}.
\end{equation}
In all our experiments, the $n_j$ are realizations of Gaussian white noises.

We run the algorithm on the noisy wavelet transform $\{h_j\}_{0\leq j\leq J}$. It returns a reconstructed signal $f_{rec}$. We quantify the reconstruction error by the difference, in relative $l^2$-norm, between the modulus of the wavelet transform of the original signal $f$ and the modulus of the wavelet transform of the reconstructed signal $f_{rec}$:
\begin{equation}\label{eq:reconstruction_error}
\mbox{reconstruction error}=\frac{\sqrt{\underset{j}{\sum}||\,|f\star\psi_j|-|f_{rec}\star\psi_j|\,||_2^2} }{\sqrt{\underset{j}{\sum}||f\star\psi_j||_2^2} }.
\end{equation}
Alternatively, we could measure the difference between $f$ and $f_{rec}$ (up to a global phase):
\begin{equation}\label{eq:error_on_the_signal}
\mbox{error on the signal}=\inf_{\phi\in\R} \frac{||e^{i\phi}f-f_{rec}||_2}{||f||_2}.
\end{equation}
But we know that the reconstruction of a function from the modulus of its wavelet transform is not stable to noise \citep{waldspurger}. So we do not hope the difference between $f$ and $f_{rec}$ to be small. We just want the algorithm to reconstruct a signal $f_{rec}$ whose wavelet transform is close to the wavelet transform of $f$, in modulus. Thus, the reconstruction error \eqref{eq:reconstruction_error} is more relevant to measure the performances of the algorithm.

\nl
In all the experiments, unless otherwise specified, we use dyadic Morlet wavelets, to which we subtract Gaussian functions of small amplitude so that they have zero mean:
\begin{equation*}
\hat\psi(\omega)=\exp(-p(\omega-1)^2)-\beta\exp(-p\omega^2)
\end{equation*}
where $\beta>0$ is chosen so that $\hat\psi(0)=0$ and the parameter $p$ is arbitrary (it controls the frequency bandwidth of the wavelets). For $N=256$, our family of wavelets contains eight elements, which are plotted on Figure \ref{fig:wav_family_1}. The performances of the algorithm strongly depend on the choice of the wavelet family; this is discussed in Paragraph \ref{sss:choice_family}.

The maximal number of iterations per local optimization step is set to $10000$ (with an additional stopping criterion, so that the $10000$-th iteration is not always reached). We study the influence of this parameter in Paragraph \ref{sss:number_iterations}.

The error correction step described in Paragraph \ref{ss:error_correction} is always turned on.

\nl
Gerchberg-Saxton is applied in a multiscale fashion, as described in Paragraph \ref{ss:multiscale_GS}, which yields better results than the regular implementation.

We use \emph{PhaseLift} \citep{candes} with ten steps of reweighting, followed by $2000$ iterations of the Gerchberg-Saxton algorithm. In our experiments with \emph{PhaseLift}, we only consider signals of size $N=256$. Handling larger signals is difficult with a straightforward Matlab implementation.

\subsection{Results\label{ss:performances}}

We describe four classes of signals, whose wavelet transforms have more or less small values. For each class, we plot the reconstruction error of our algorithm, Gerchberg-Saxton and \emph{PhaseLift} as a function of the noise error.

\subsubsection{Realizations of Gaussian random processes}

\begin{figure}
\centering
\begin{minipage}[c]{0.45\textwidth}
\centering
\includegraphics[width=0.48\textwidth]{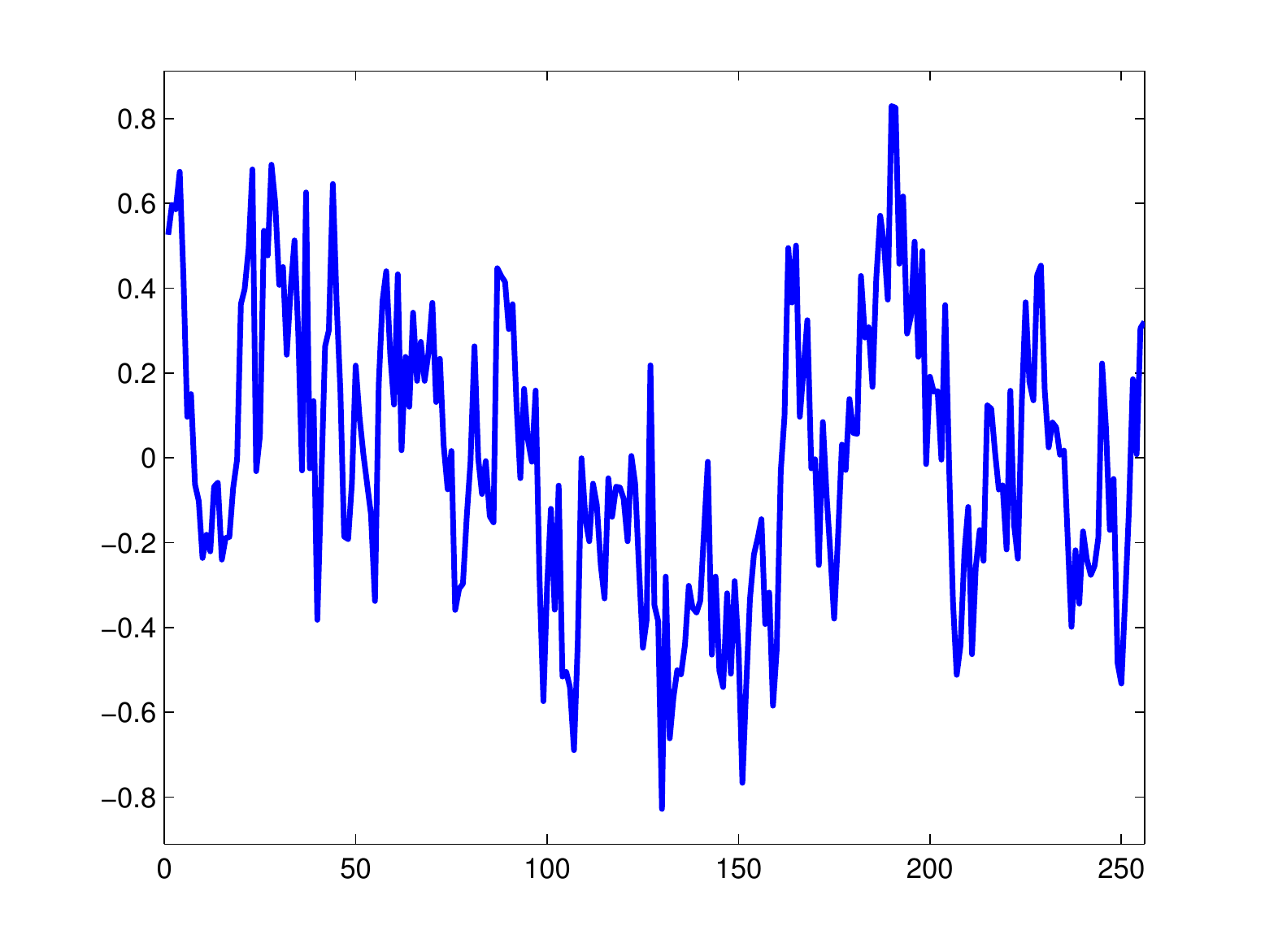}
\includegraphics[width=0.48\textwidth]{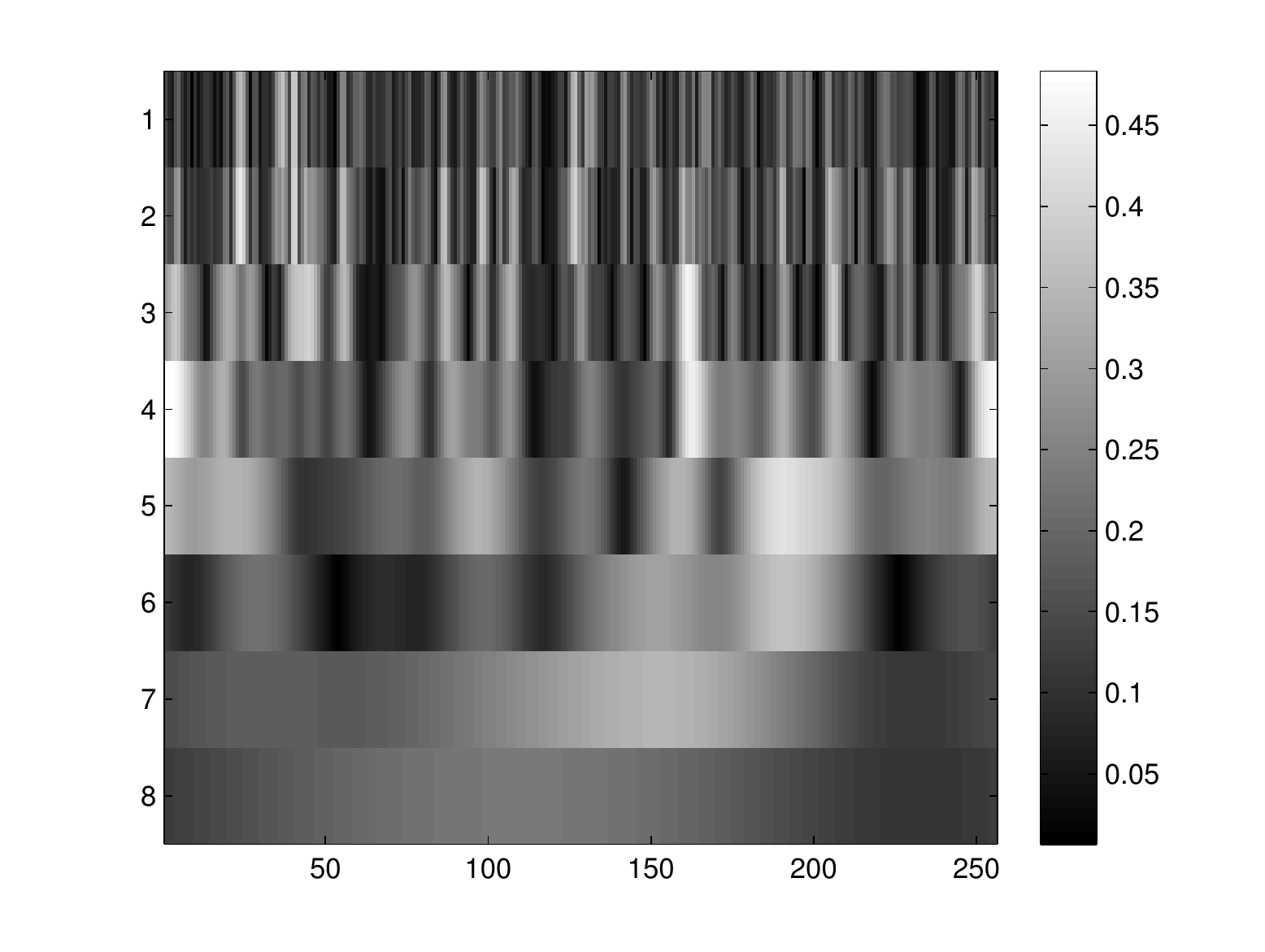}
\end{minipage}
\caption{Realization of a Gaussian process (left) and modulus of its wavelet transform (right)\label{fig:signals_gauss}}
\end{figure}
\begin{figure}
\centering
\includegraphics[width=0.4\textwidth]{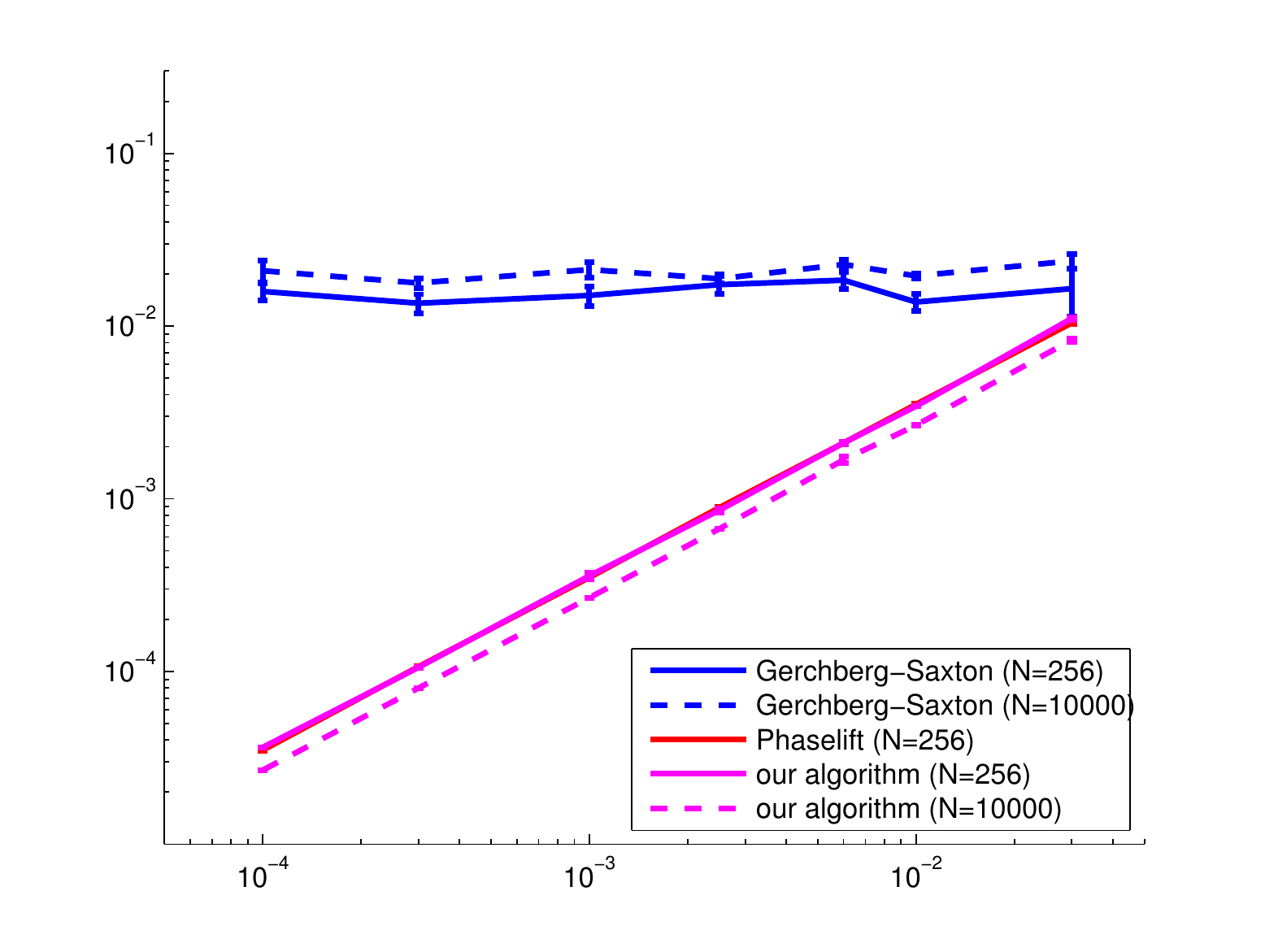}
\caption{Mean reconstruction error as a function of the noise, for Gaussian signals of size $N=256$ or $10000$\label{fig:comp_gauss}}
\end{figure}

We first consider realizations of Gaussian random processes. A signal $f$ in this class is defined by
\begin{alignat*}{3}
\hat f[k]&=\frac{X_k}{\sqrt{k+1}}&&\quad\mbox{if }k\in\{1,...,N/2\};\\
&=0&&\quad\mbox{if not}.
\end{alignat*}
where $X_1,...,X_{N/2}$ are independent realizations of complex Gaussian centered variables. The role of the $\sqrt{k+1}$ is to ensure that all components of the wavelet transform approximately have the same $l^2$-norm (in expectation). An example is displayed on Figure \ref{fig:signals_gauss}, along with the modulus of its wavelet transform.

The wavelet transforms of these signals have few small values, disposed in a seemingly random pattern. This is the most favorable class for our algorithm.

The reconstruction results are shown in Figure \ref{fig:comp_gauss}. Even for large signals ($N=10000$), the mean reconstruction error is proportional to the input noise (generally $2$ or $3$ times smaller); this is the best possible result. The performances of \emph{PhaseLift} are exactly the same, but Gerchberg-Saxton often fails.

\subsubsection{Lines from images}

\begin{figure}
\centering
\begin{minipage}[c]{0.45\textwidth}
\centering
\includegraphics[width=0.48\textwidth]{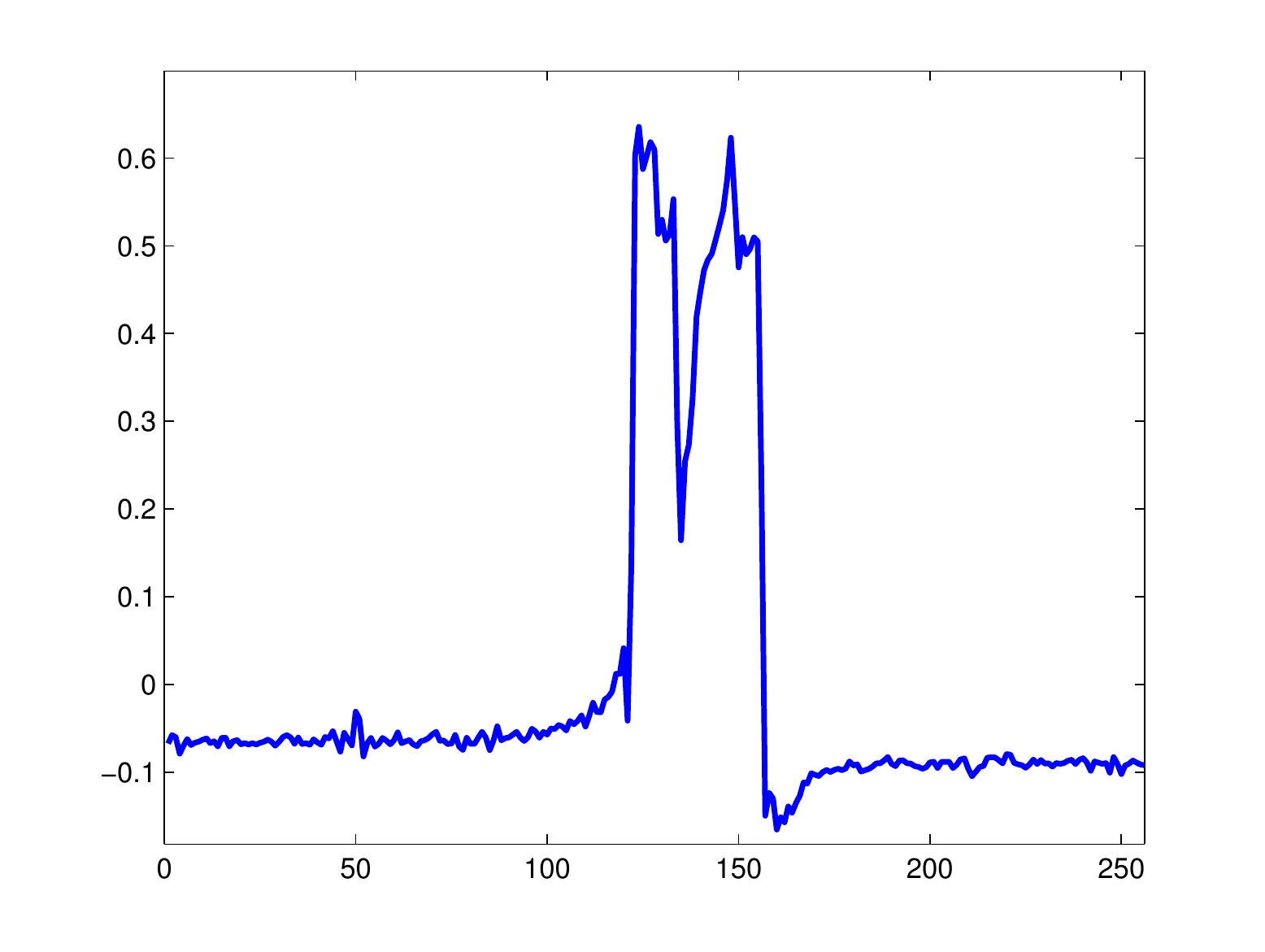}
\includegraphics[width=0.48\textwidth]{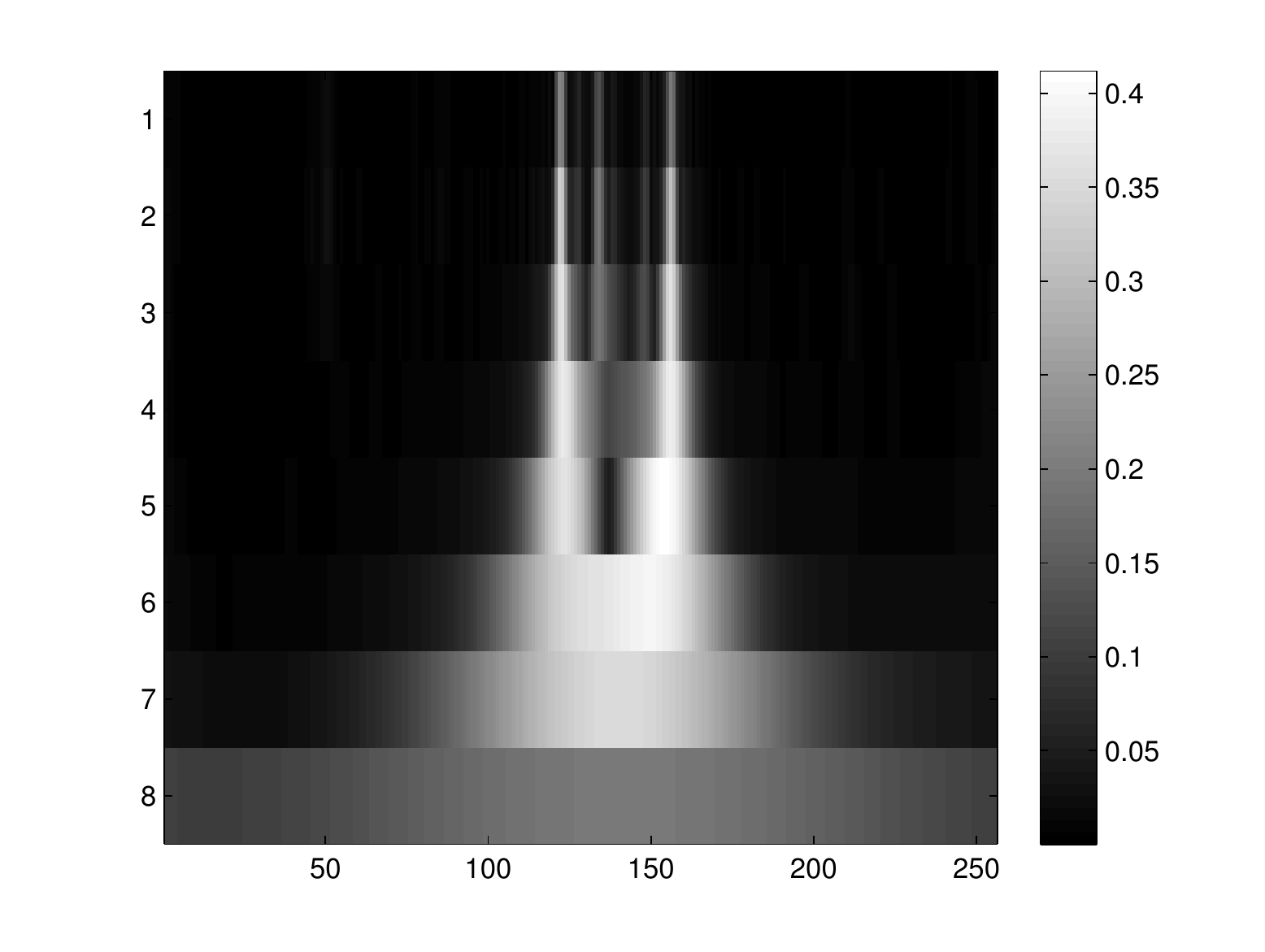}
\end{minipage}
\caption{Line from an image (left) and modulus of its wavelet transform (right)\label{fig:signals_piece}}
\end{figure}
\begin{figure}
\centering
\includegraphics[width=0.4\textwidth]{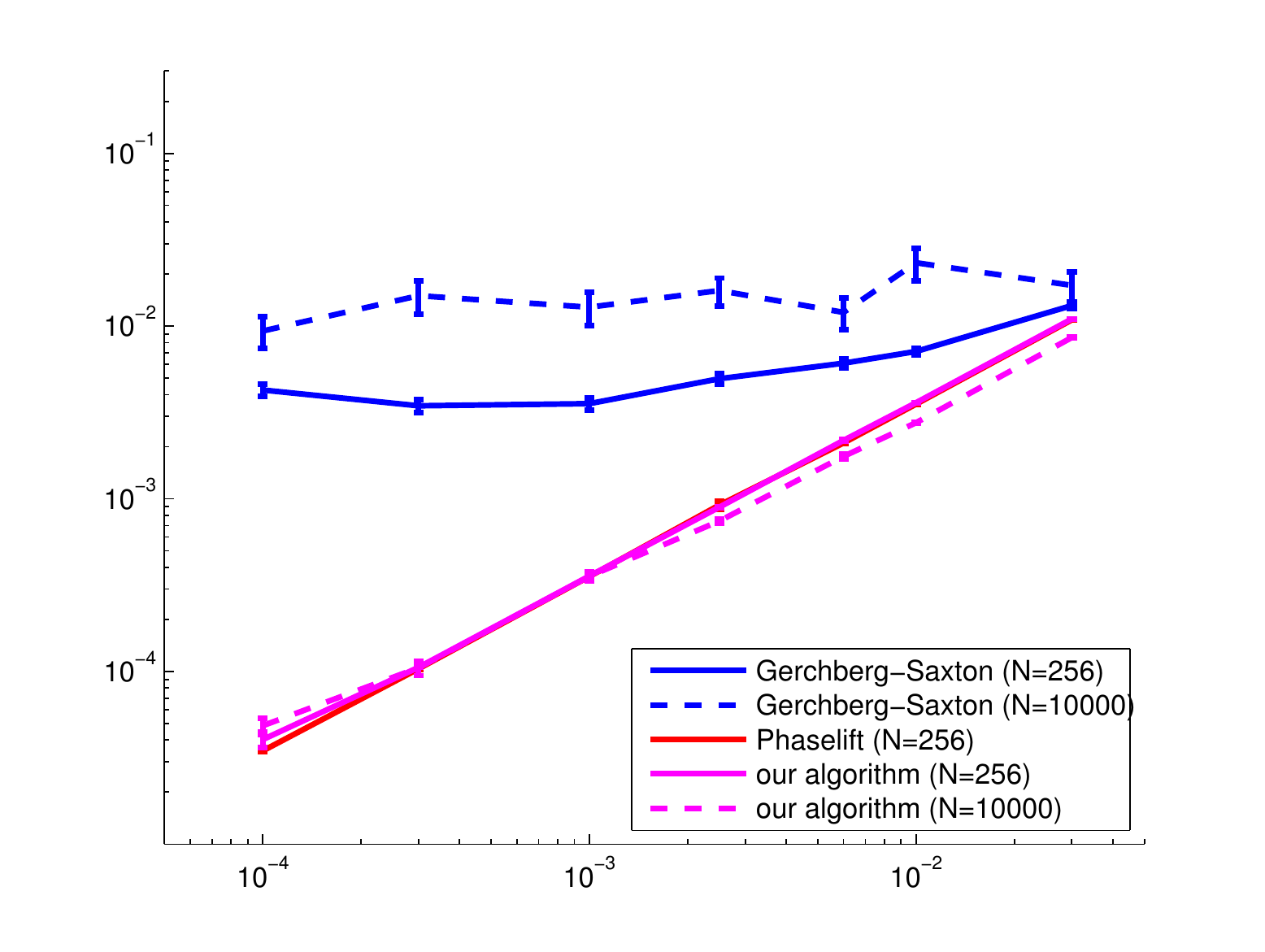}
\caption{Mean reconstruction error as a function of the noise, for lines extracted from images, of size $N=256$ or $10000$\label{fig:comp_piece}}
\end{figure}

The second class consists in lines randomly extracted from photographs. These signals have oscillating parts (corresponding to the texture zones of the initial image) and smooth parts, with large discontinuities in between. Their wavelet transforms generally contain a lot a small values, but, as can be seen in Figure \ref{fig:signals_piece}, the distribution of these small values is particular. They are more numerous at high frequencies and the non-small values tend to concentrate on vertical lines of the time-frequency plane.

This distribution is favorable to our algorithm: small values in the wavelet transform are mostly a problem when they are in the low frequencies and prevent the correct initialization of the reconstruction at medium or high frequencies. Small values at high frequencies are not a problem.

Indeed, as in the case of Gaussian signals, the reconstruction error is proportional to the input noise (figure \ref{fig:comp_piece}). This is also the case for \emph{PhaseLift} but not for Gerchberg-Saxton.

\subsubsection{Sums of a few sinusoids}

\begin{figure}
\centering
\begin{minipage}[c]{0.45\textwidth}
\centering
\includegraphics[width=0.48\textwidth]{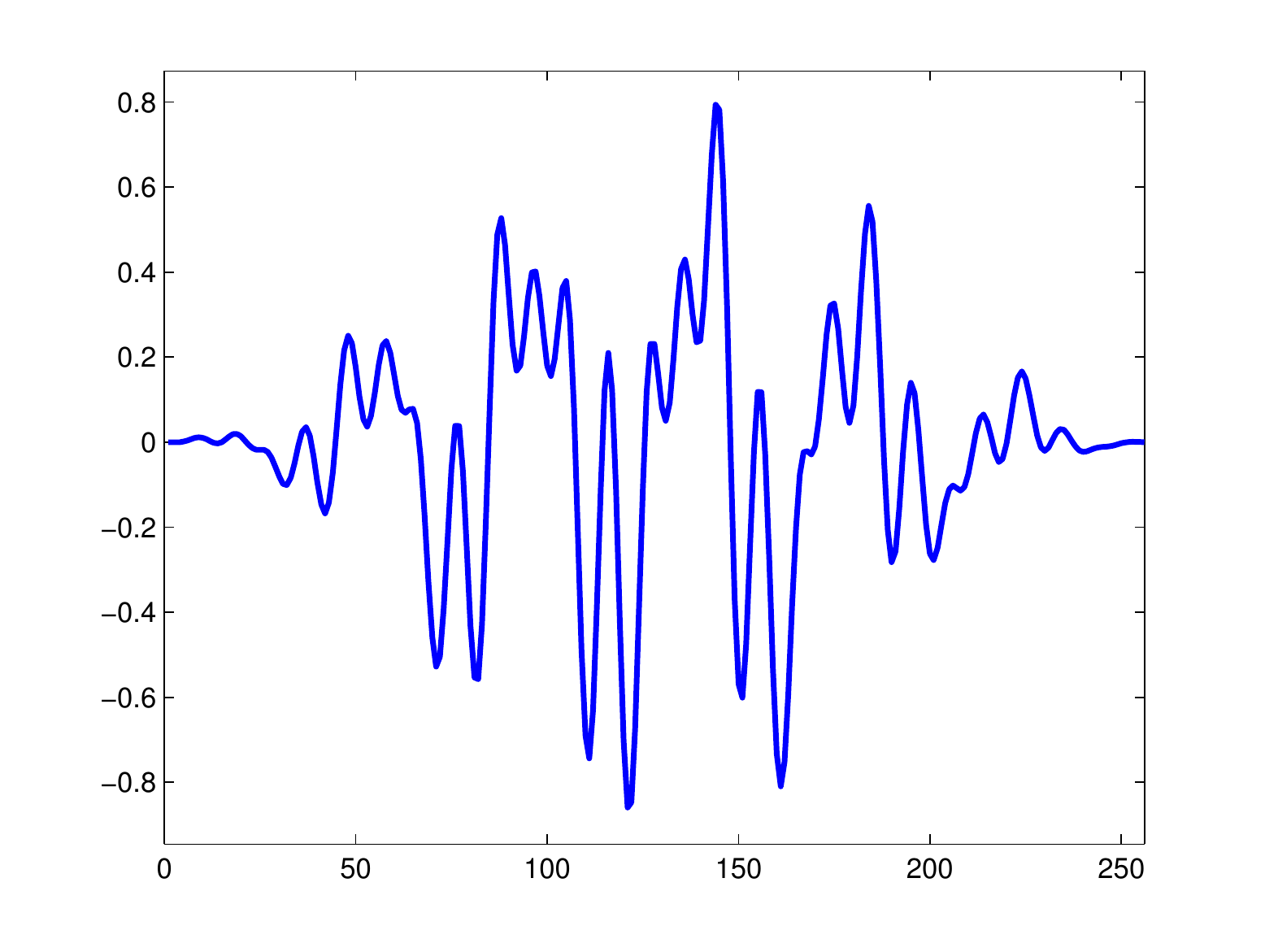}
\includegraphics[width=0.48\textwidth]{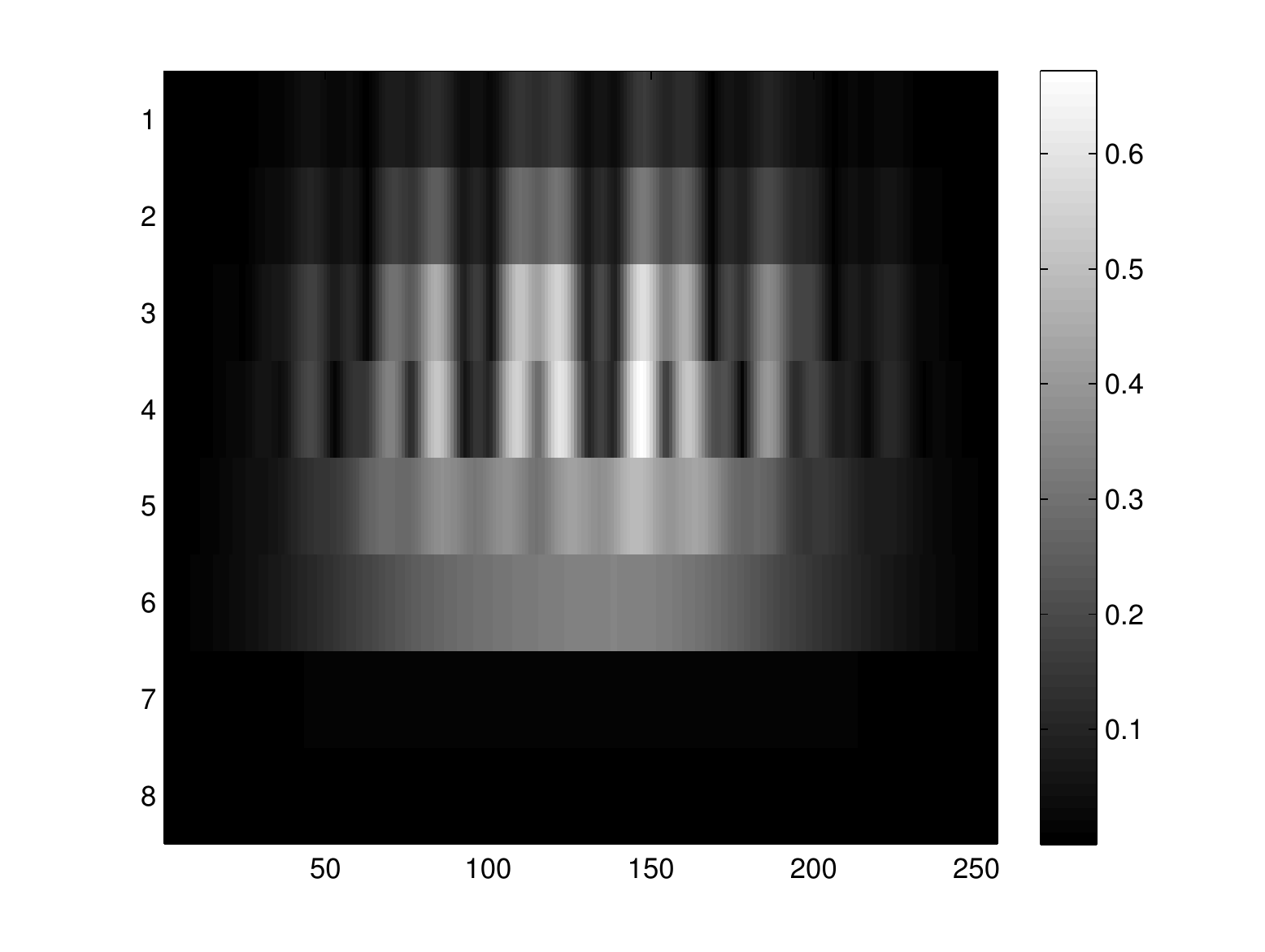}
\end{minipage}
\caption{Random sum of sinusoids, multiplied by a window function (left) and modulus of its wavelet transform (right)\label{fig:signals_comp_sin}}
\end{figure}
\begin{figure}
\centering
\includegraphics[width=0.4\textwidth]{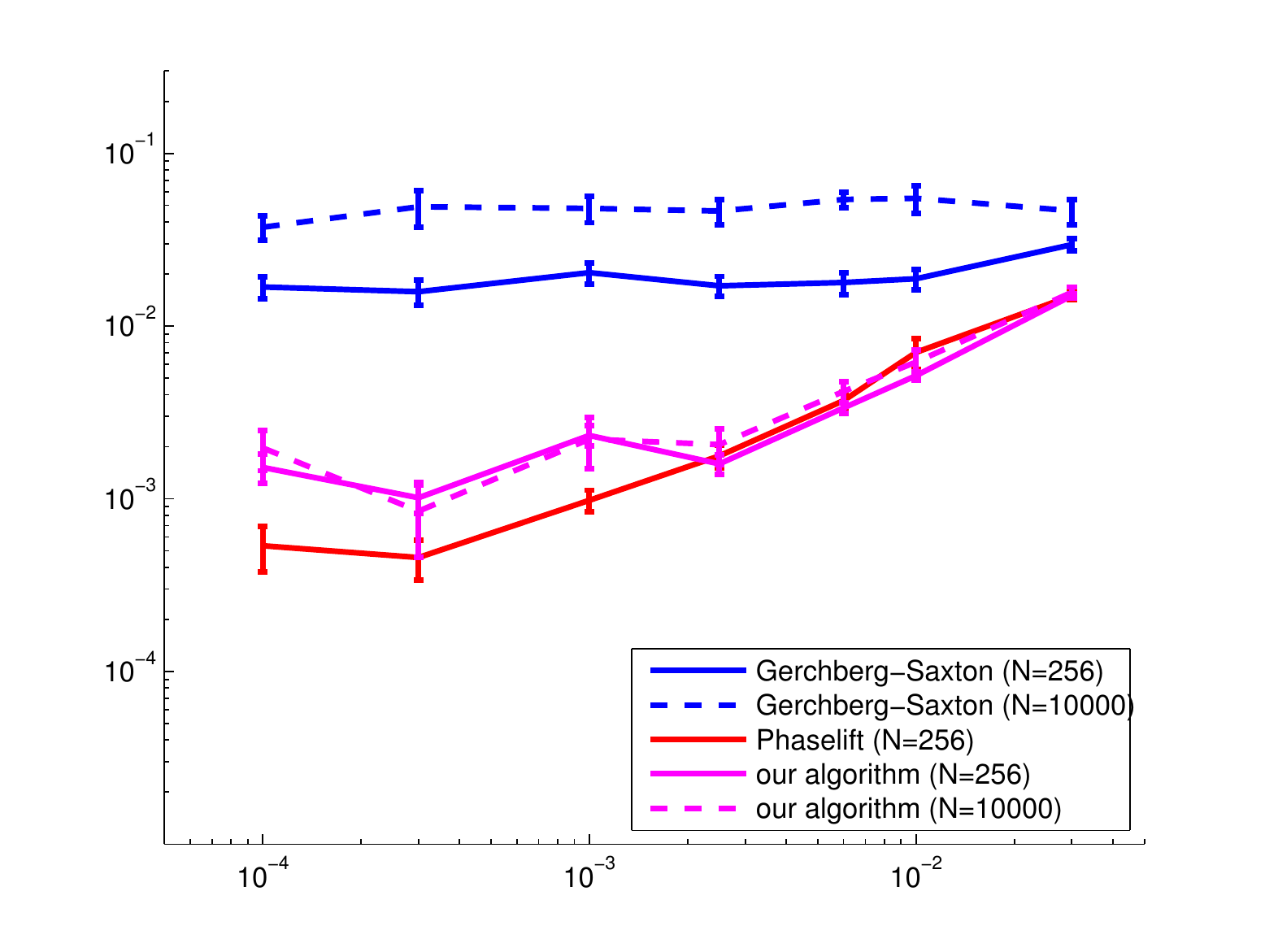}
\caption{Mean reconstruction error as a function of the noise, for random sums of sinusoids multiplied by a window function, of size $N=256$ or $10000$\label{fig:comp_comp_sin}}
\end{figure}

The next class of signals contains sums of a few sinusoids, multiplied by a window function $w$ to avoid boundary effects. Formally, a signal in this class is of the form
\begin{equation*}
f[n] = \left[\underset{k=1}{\overset{N/2}{\sum}}\alpha_k\exp\left(i\frac{2\pi kn}{N}\right)\right] \times w[n].
\end{equation*}
where the $\alpha_k$ are zero with high probability and realizations of complex Gaussian centered variables with small probability.

The wavelet transforms of these signals often have components of very small amplitude, which may be located at any frequential scale (figure \ref{fig:signals_comp_sin}). This can prevent the reconstruction.

\nl
The results are on Figure \ref{fig:comp_comp_sin}. Our algorithm performs much better than Gerchberg-Saxton but the results are not as good as for the two previous classes of signals.

In most reconstruction trials, the signal is correctly reconstructed, up to an error proportional to the noise. But, with a small probability, the reconstruction fails. The same phenomenon occurs for \emph{PhaseLift}.

The probability of failure seems a bit higher for \emph{PhaseLift} than for our algorithm. For example, when the signals are of size $256$ and the noise has a relative norm of $0.01\%$, the reconstruction error is larger than the noise error $20\%$ of the time for \emph{PhaseLift} and only $10\%$ of the time for our algorithm. However, \emph{PhaseLift} has a smaller mean reconstruction error because, in these failure cases, the result it returns, although not perfect, is more often close to the truth: the mean reconstruction error in the failure cases is $0.2\%$ for \emph{PhaseLift} versus $1.7\%$ for our algorithm.

\subsubsection{Audio signals}

\begin{figure*}
\centering
\subfloat[]{\includegraphics[width=0.48\textwidth]{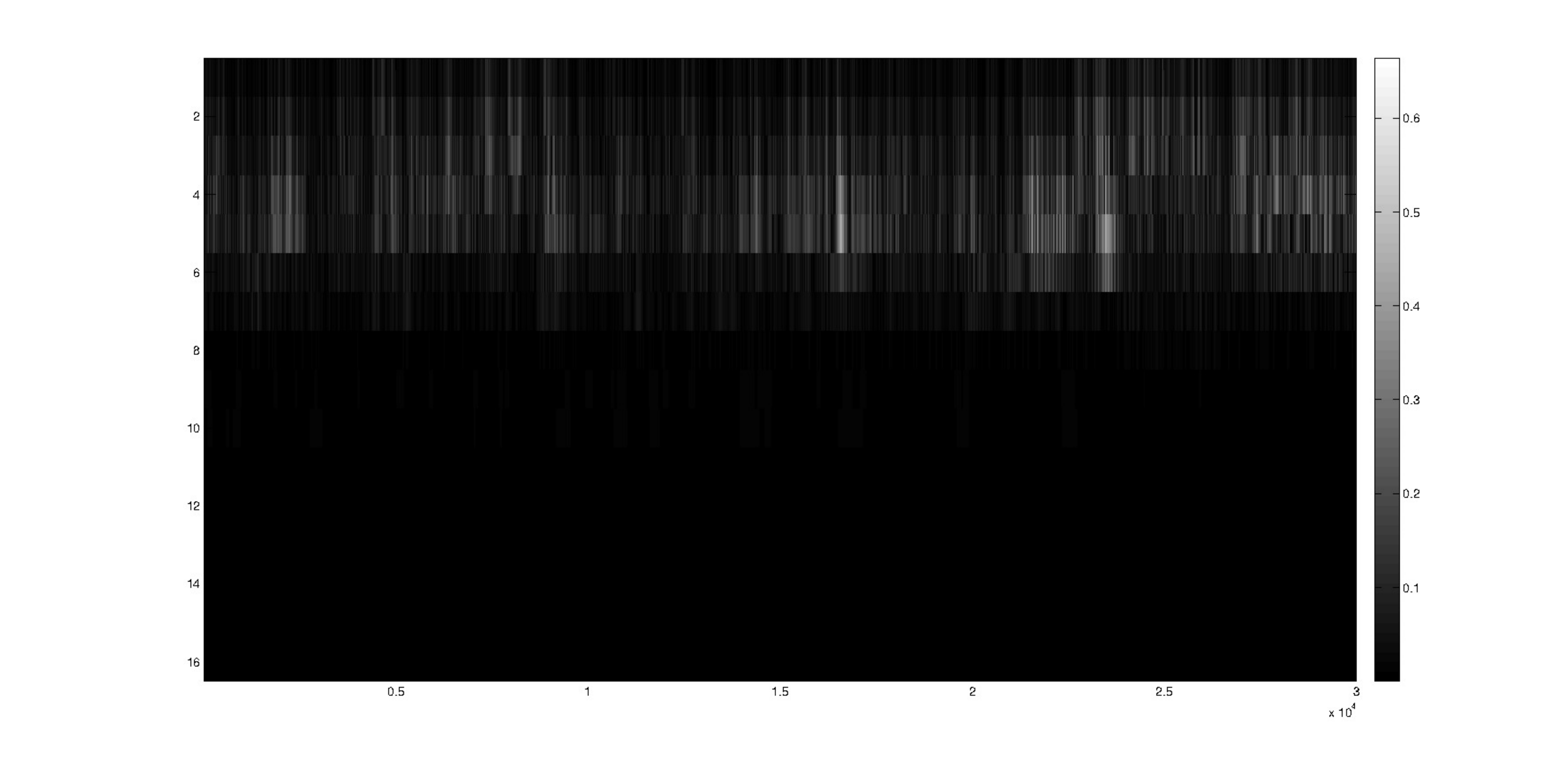}
\label{fig:signals_rimsky}}
\hfil
\subfloat[]{\includegraphics[width=0.48\textwidth]{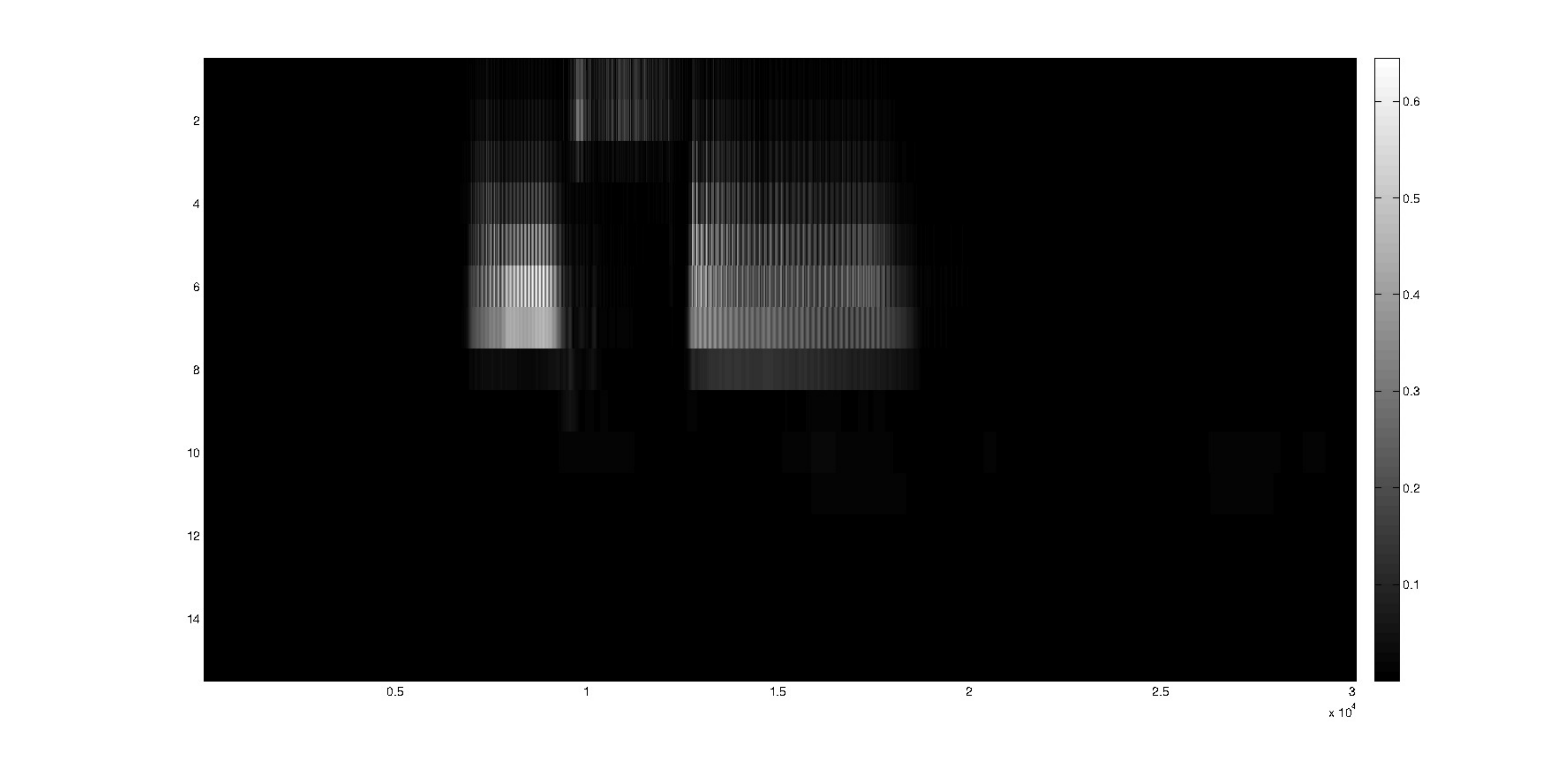}
\label{fig:signals_sorry}}
\caption{Wavelet transforms of the audio signals (a) Rimsky-Korsakov (b) ``I'm sorry''}
\end{figure*}

Finally, we test our algorithm on real audio signals. These signals are difficult to reconstruct because they do not contain very low frequencies (as the human ear cannot hear them, these frequencies are not included in the recordings), so the first components of their wavelet transforms are very small.

The reconstruction results may vary from one audio signal to the other. We focus here on two representative examples.

The first signal is an extract of five seconds of a musical piece played by an orchestra (the \textit{Flight of the Bumblebee}, by Rimsky-Korsakov). Figure \ref{fig:signals_rimsky} shows the modulus of its wavelet transform. It has $16$ components and $9$ of them (the ones with lower characteristic frequencies) seem negligible, compared to the other ones. However, its non-negligible components have a moderate number of small values.

The second signal is a recording of a human voice saying ``I'm sorry'' (figure \ref{fig:signals_sorry}). The low-frequency components of its wavelet transform are also negligible, but even the high-frequency components tend to have small values, which makes the reconstruction even more difficult.

\begin{figure}
\centering
\includegraphics[width=0.4\textwidth]{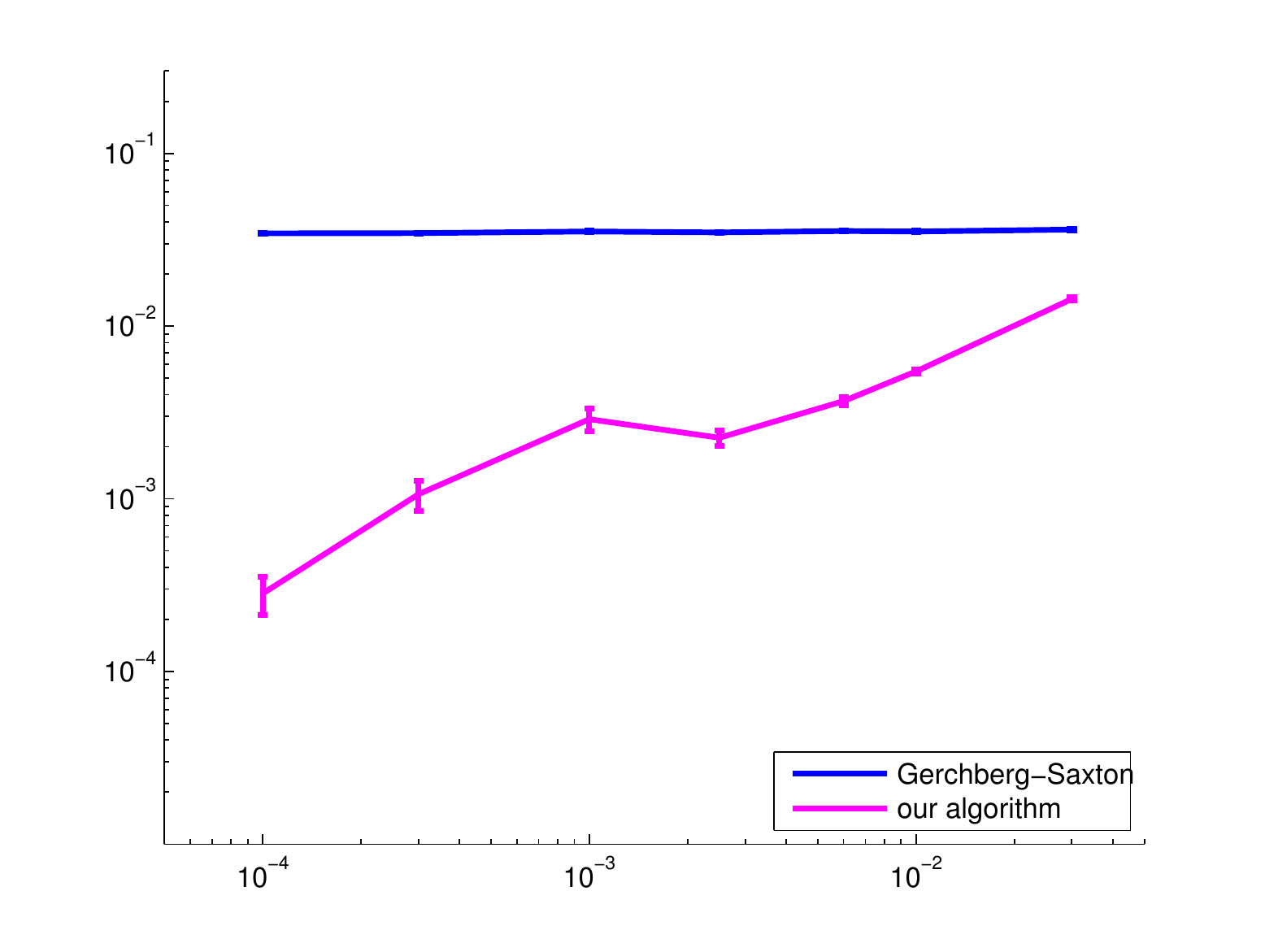}
\caption{mean reconstruction error as a function of the noise, for the audio signal ``Rimsky-Korsakov''\label{fig:comp_rimsky}}
\end{figure}

\begin{figure}
\centering
\includegraphics[width=0.4\textwidth]{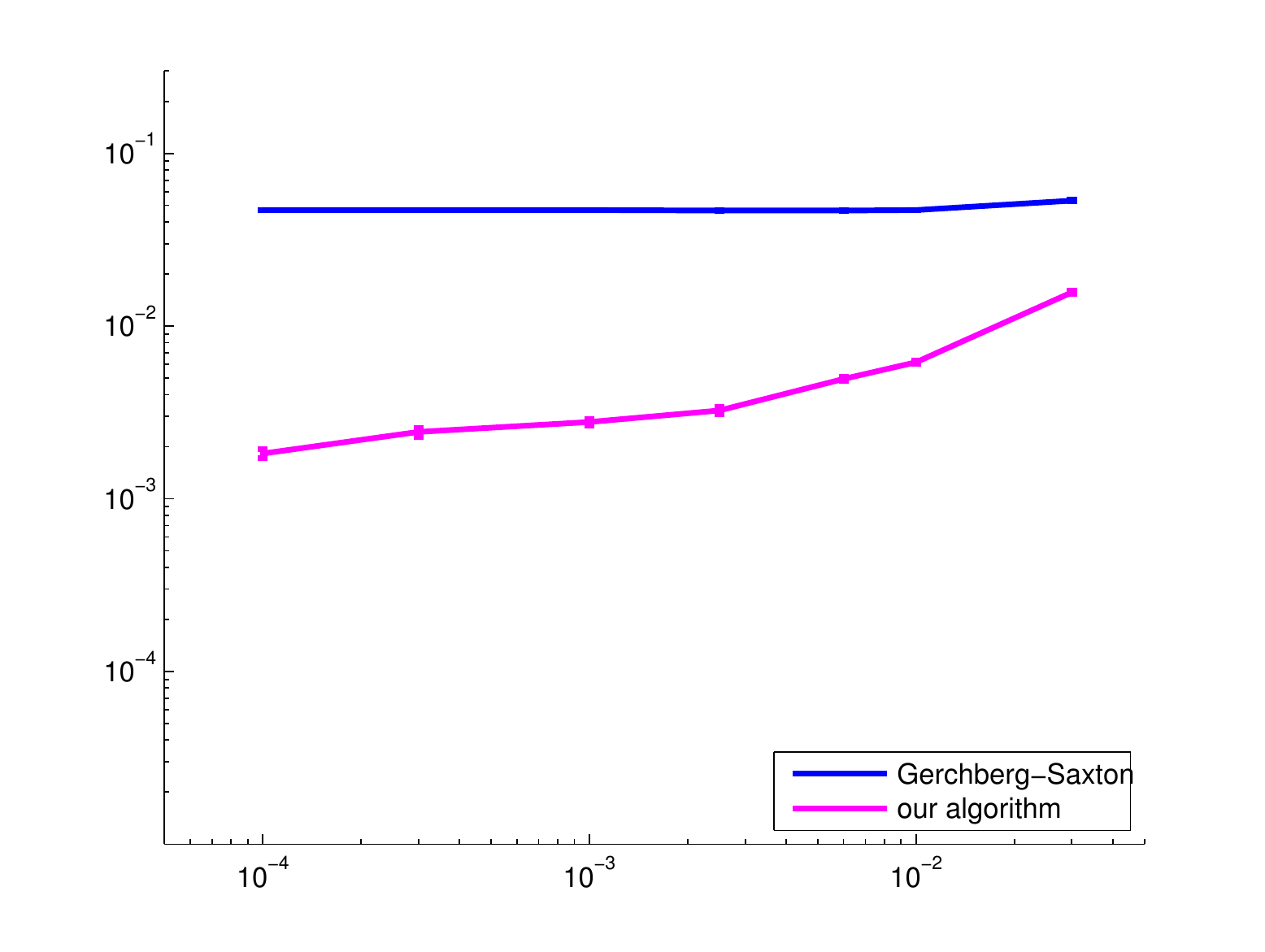}
\caption{mean reconstruction error as a function of the noise, for the audio signal ``I'm sorry''\label{fig:comp_sorry}}
\end{figure}

\nl
The results are presented in Figures \ref{fig:comp_rimsky} and \ref{fig:comp_sorry}. For relatively high levels of noise ($0.5\%$ or higher), the results, in the sense of the $l^2$-norm, are satisfying: the reconstruction error is smaller or equal to the amount of noise.

In the high precision regime (that is, for $0.1\%$ of noise or less), the lack of low frequencies does not allow a perfect reconstruction. Nevertheless, the results are still good: the reconstruction error is of the order of $0.1\%$ or $0.2\%$ when the noise error is below $0.1\%$. More iterations in the optimization steps can further reduce this error. By comparison, the reconstruction error with Gerchberg-Saxton is always several percent, even when the noise is small.


\subsection{Stability of the reconstruction\label{ss:stability}}

In this section, we use our reconstruction algorithm to investigate the stability of the reconstruction. From \citep{waldspurger}, we know that the reconstruction operator is not uniformly continuous: the reconstruction error \eqref{eq:reconstruction_error} can be small (the modulus of the wavelet transform is almost exactly reconstructed), even if the error on the signal \eqref{eq:error_on_the_signal} is not small (the difference between the initial signal and its reconstruction is large).

We show that this phenomenon can occur for all classes of signals, but is all the more frequent when the wavelet transform has a lot of small values, especially in the low frequency components.

We also experimentally show that, when this phenomenon happens, the original and reconstructed signals have their wavelet transforms $\{f\star\psi_j(t)\}_{j\in\Z,t\in\R}$ equal up to multiplication by a phase $\{e^{i\phi_j(t)}\}_{j\in\Z,t\in\R}$, which varies slowly in both $j$ and $t$, except maybe at the points where $f\star\psi_j(t)$ is close to zero. This has been conjectured in \citep{waldspurger}. It is a form of ``local stability'': the signal and its reconstruction are close, up to a global phase, in a neighborhood of each point of the time-frequency plane.

\nl
We perform a large number of reconstruction trials, with various reconstruction parameters. This gives us a large number of pairs $(f,f_{rec})$, such that $\forall j,t,|f\star\psi_j(t)|\approx|f_{rec}\star\psi_j(t)|$. For each one of these pairs, we compute
\begin{align*}
\mbox{error on the modulus}&=\frac{\sqrt{\underset{j}{\sum}||\,|f\star\psi_j|-|f_{rec}\star\psi_j|\,||_2^2} }{\sqrt{\underset{j}{\sum}||f\star\psi_j||_2^2} };
\tag{\ref{eq:reconstruction_error}}\\
\mbox{error on the signal}&=\frac{||f-f_{rec}||_2}{||f||_2}.
\tag{\ref{eq:error_on_the_signal}}
\end{align*}

The results are plotted on Figure \ref{fig:stability}, where each point corresponds to one reconstruction trial. The x-coordinate represents the error on the modulus and the y-coordinate the error on the signal.

\nl
We always have:
\begin{equation*}
\mbox{error on the modulus}\leq C\times(\mbox{error on the function}).
\end{equation*}
with $C$ a constant of the order of $1$. This is not surprising because the modulus of the wavelet transform is a Lipschitz operator, with a constant close to $1$.
\nl

As expected, the converse inequality is not true: the error on the function can be significantly larger than the error on the modulus. For each class, an important number of reconstruction trials yield errors such that:
\begin{equation*}
\mbox{error on the signal}\approx 30\times\mbox{error on the modulus}.
\end{equation*}
For realizations of Gaussian random processes or for lines extracted from images (figures \ref{fig:stability_gauss} and \ref{fig:stability_piece}), the ratio between the two errors never exceeds $30$ (except for one outlier). But for sums of a few sinusoids (\ref{fig:stability_comp_sin}) or audio signals (\ref{fig:stability_sorry}), we may even have:
\begin{equation*}
\mbox{error on the signal}\geq 100\times\mbox{error on the modulus}.
\end{equation*}
So instabilities appear in the reconstruction of all kinds of signals, but are stronger for sums of sinusoids and audio signals, that is for the signals whose wavelet transforms have a lot of small values, especially in the low frequencies.

\nl
These results have a theoretical justification. \citep{waldspurger} explain how, from any signal $f$, it is possible to construct $g$ such that $|f\star\psi_j|\approx |g\star\psi_j|$ for all $j$ but $f\not\approx g$ in the $l^2$-norm sense.

The principle of the construction is to multiply each $f\star\psi_j(t)$ by a phase $e^{i\phi_j(t)}$. The function $(j,t)\to e^{i\phi_j(t)}$ must be chosen so that it varies slowly in both $j$ and $t$, except maybe at the points $(j,t)$ where $f\star\psi_j(t)$ is small. Then there exist a signal $g$ such that $(f\star\psi_j(t))e^{i\phi_j(t)}\approx g\star\psi_j(t)$ for any $j,t$. Taking the modulus of this approximate equality yields
\begin{equation*}
\forall j,t\quad\quad |f\star\psi_j(t)|\approx|g\star\psi_j(t)|.
\end{equation*}
However, we may not have $f\approx g$.

This construction works for any signal $f$ (unless the wavelet transform is very localized in the time frequency domain), but the number of possible $\{e^{i\phi_j(t)}\}_{j,t}$ is larger when the wavelet transform of $f$ has a lot of small values, because the constraint of slow variation is relaxed at the points where the wavelet transform is small (especially when the small values are in the low frequencies). This is probably why instabilities occur for all kinds of signals, but more frequently when the wavelet transforms have a lot of zeroes.

\begin{figure}
\centering
\subfloat[]{
\includegraphics[width=0.23\textwidth]{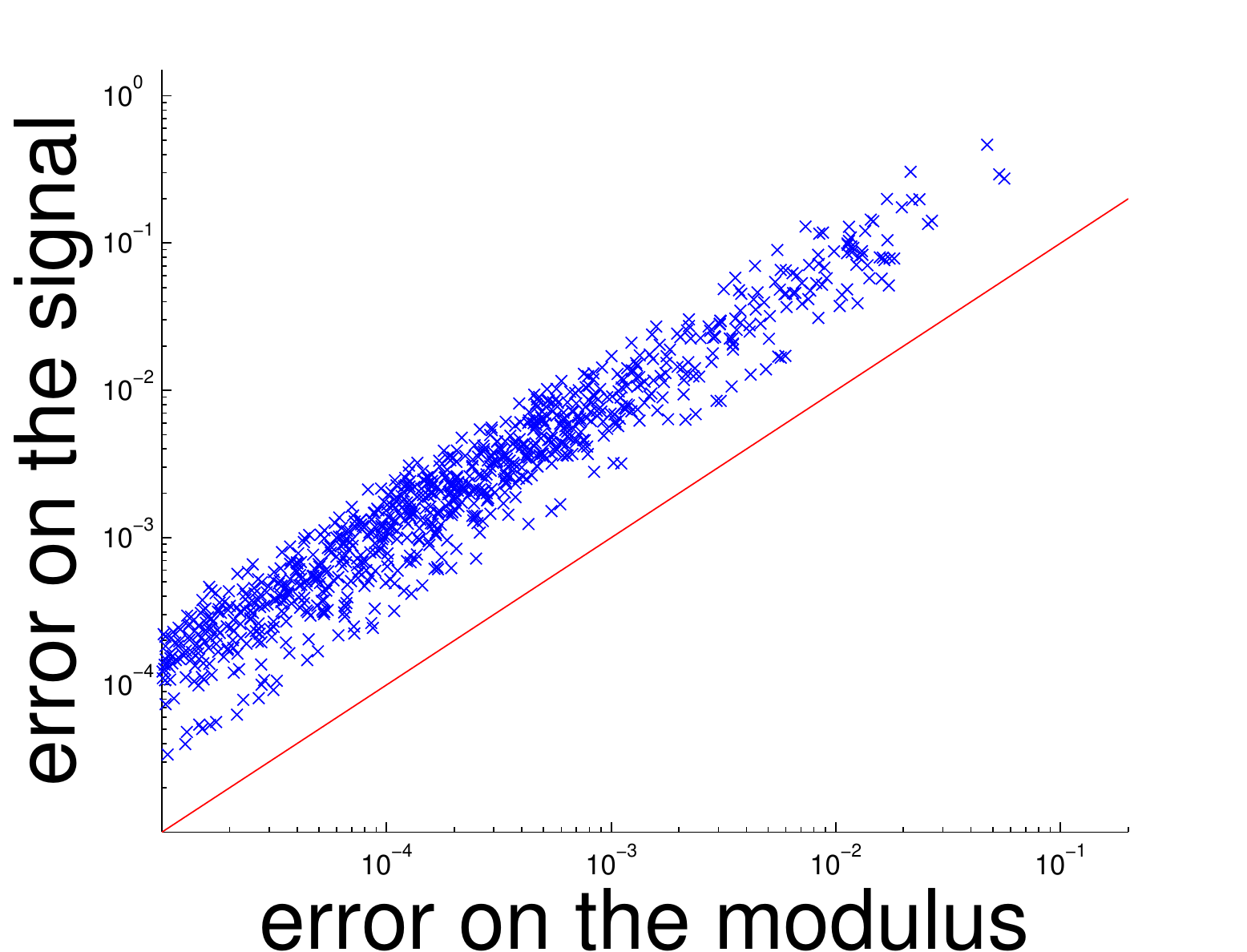}
\label{fig:stability_gauss}}
\hfil
\subfloat[]{
\includegraphics[width=0.23\textwidth]{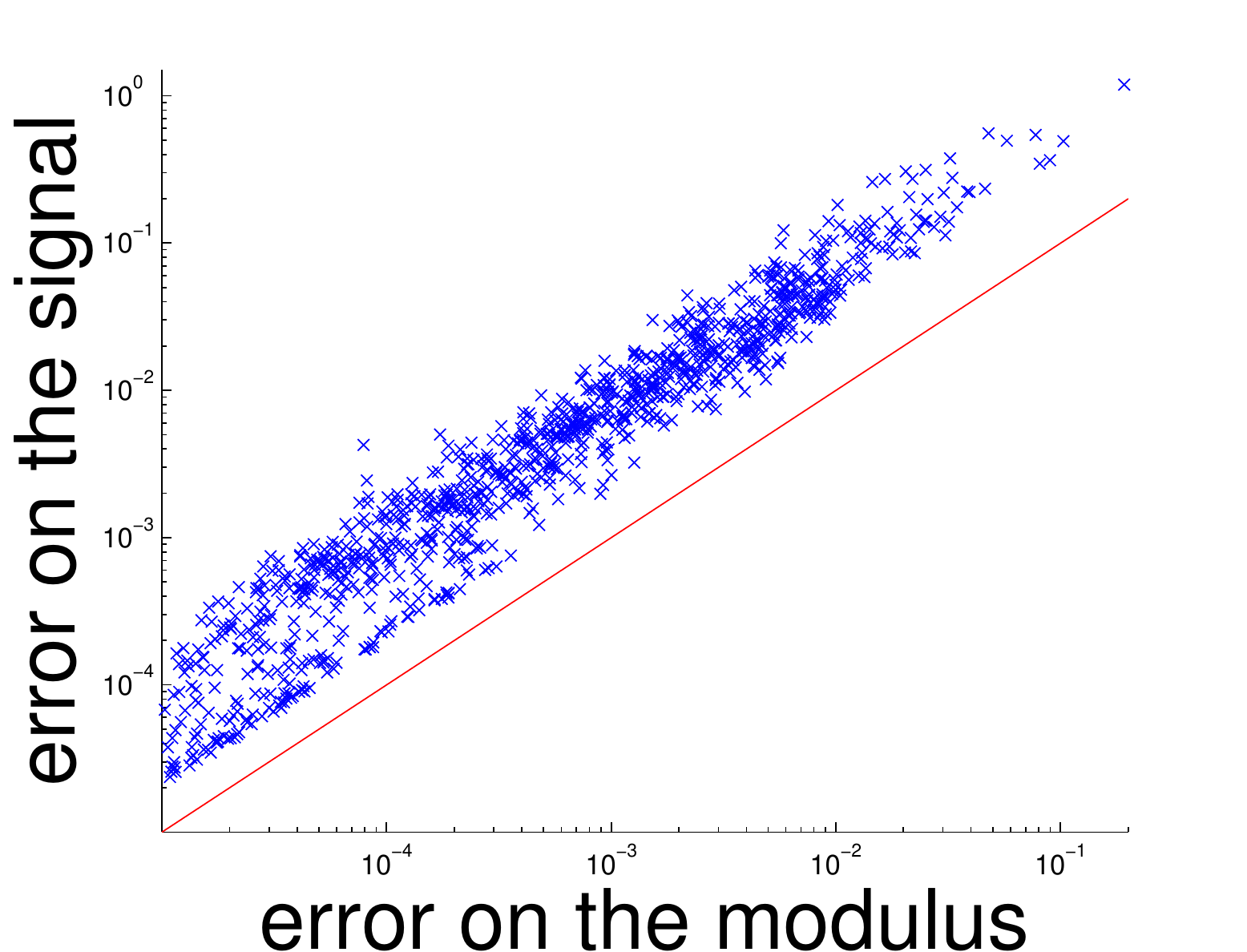}
\label{fig:stability_piece}}
\hfil
\subfloat[]{
\includegraphics[width=0.23\textwidth]{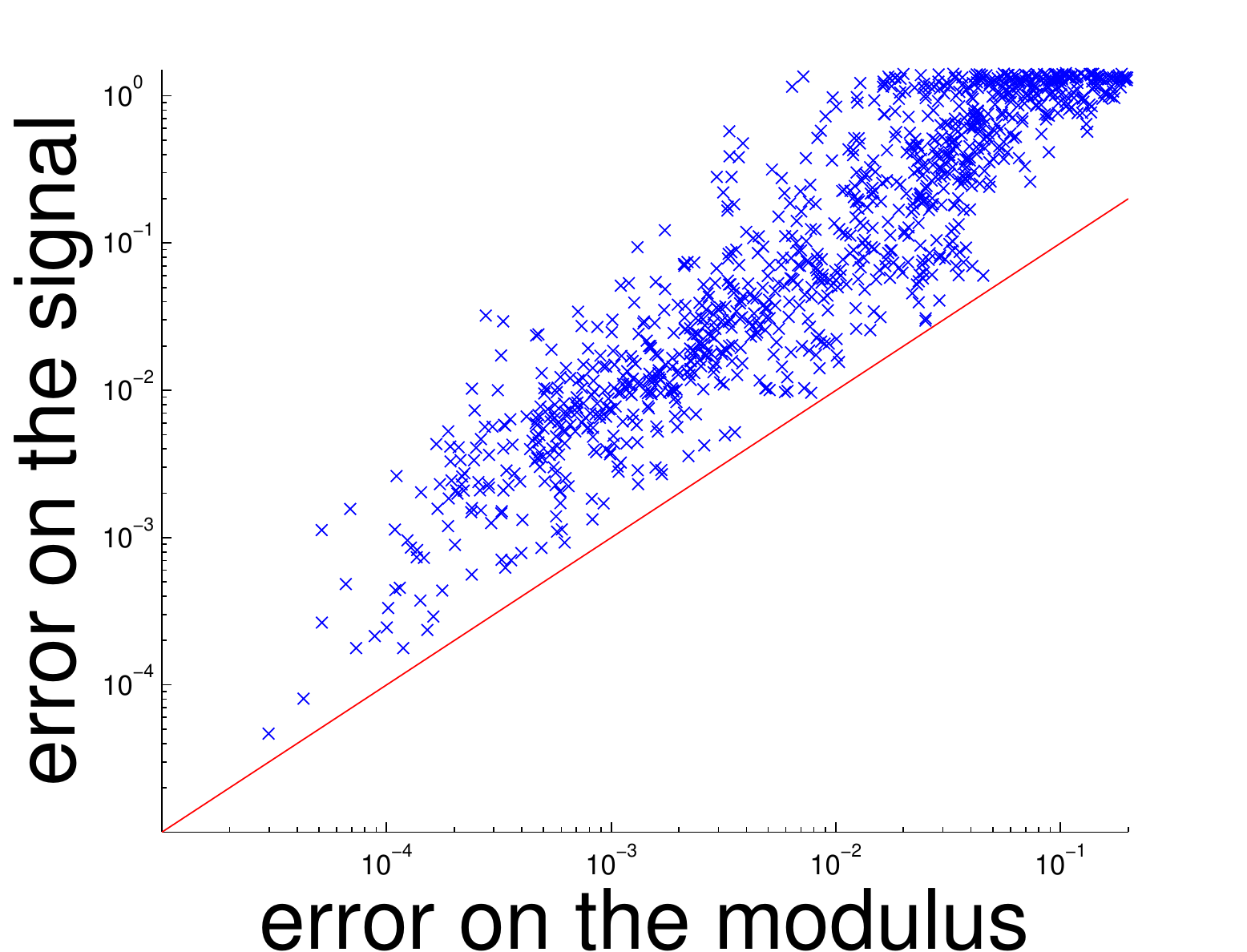}
\label{fig:stability_comp_sin}}
\hfil
\subfloat[]{
\includegraphics[width=0.23\textwidth]{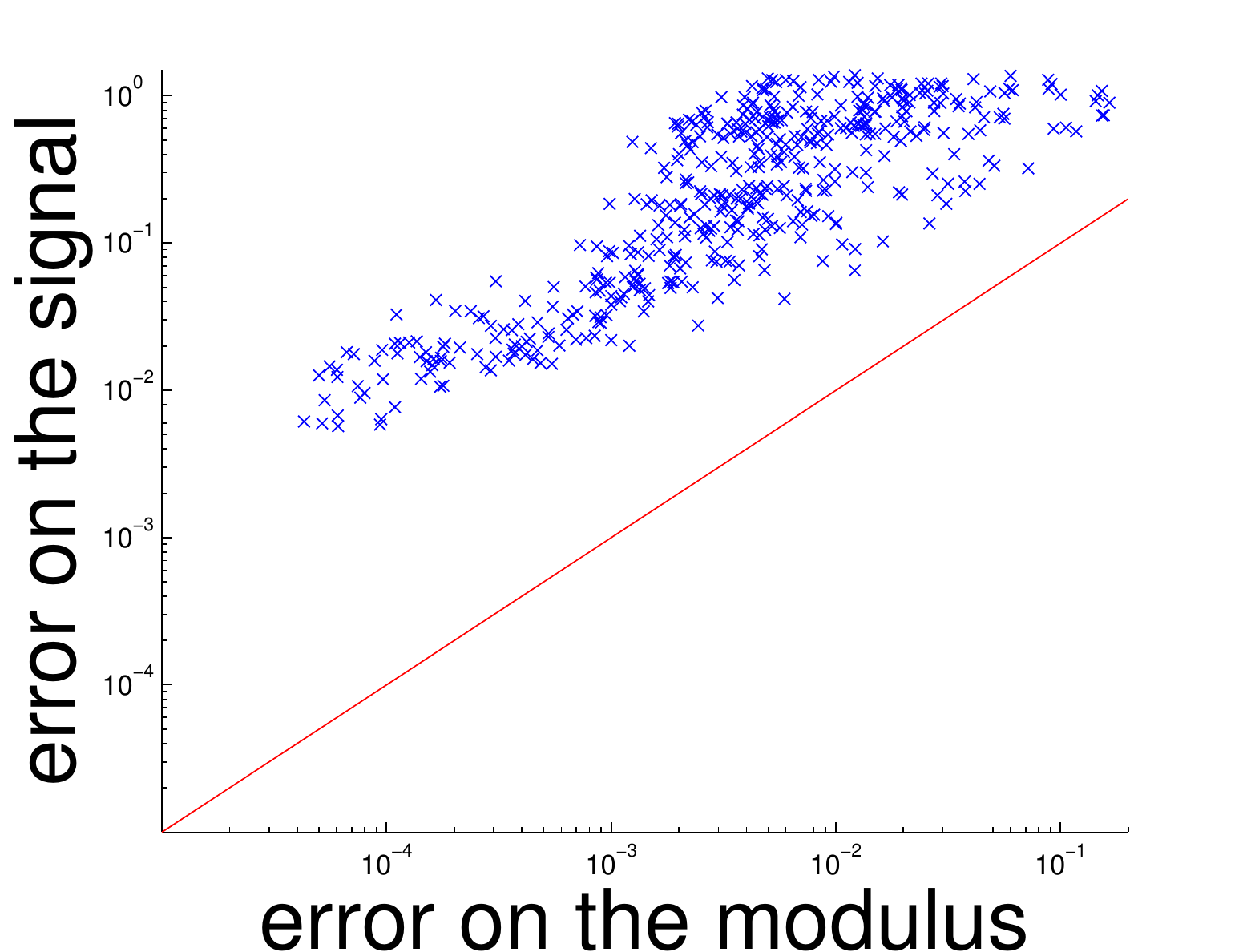}
\label{fig:stability_sorry}}
\caption{error on the signal \eqref{eq:error_on_the_signal} as a function of the error on the modulus of the wavelet transform \eqref{eq:reconstruction_error}, for several reconstruction trials; the red line $y=x$ is here to serve as a reference (a) Gaussian signals (b) lines from images (c) sums of sinusoids (d) audio signal ``I'm sorry'' \label{fig:stability}}
\end{figure}

\nl
From our experiments, it seems that the previous construction describes all the instabilities: when the wavelet transforms of $f$ and $f_{rec}$ have almost the same modulus and $f$ is not close to $f_{rec}$, then the wavelet transforms of $f$ and $f_{rec}$ are equal up to slow-varying phases $\{e^{i\phi_j(t)}\}_{j,t}$.

Figure \ref{fig:diff_phase} shows an example. The signal is a sum of sinusoids. The relative difference between the modulus is $0.3\%$, but the difference between the initial and reconstructed signals is more than a hundred times larger; it is $46\%$. The second subfigure shows the difference between the phases of the two wavelet transforms. It indeed varies slowly, in both time and frequency (actually, it is almost constant along the frequency axis), and a bit faster at the extremities, where the wavelet transform is closer to zero.

\begin{figure}
\centering
\subfloat[]{
\includegraphics[width=0.4\textwidth]{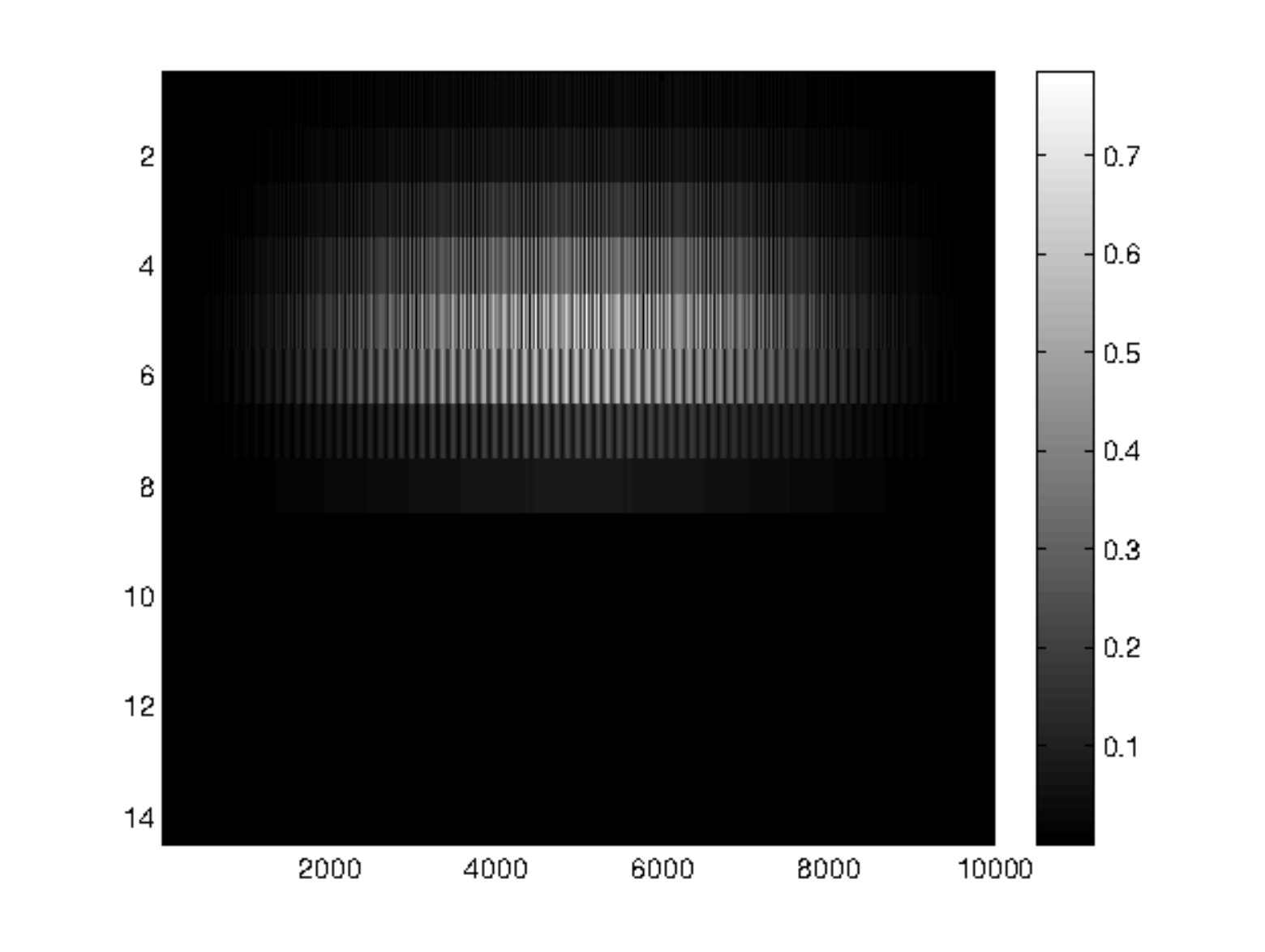}}
\hfil
\subfloat[]{
\includegraphics[width=0.4\textwidth]{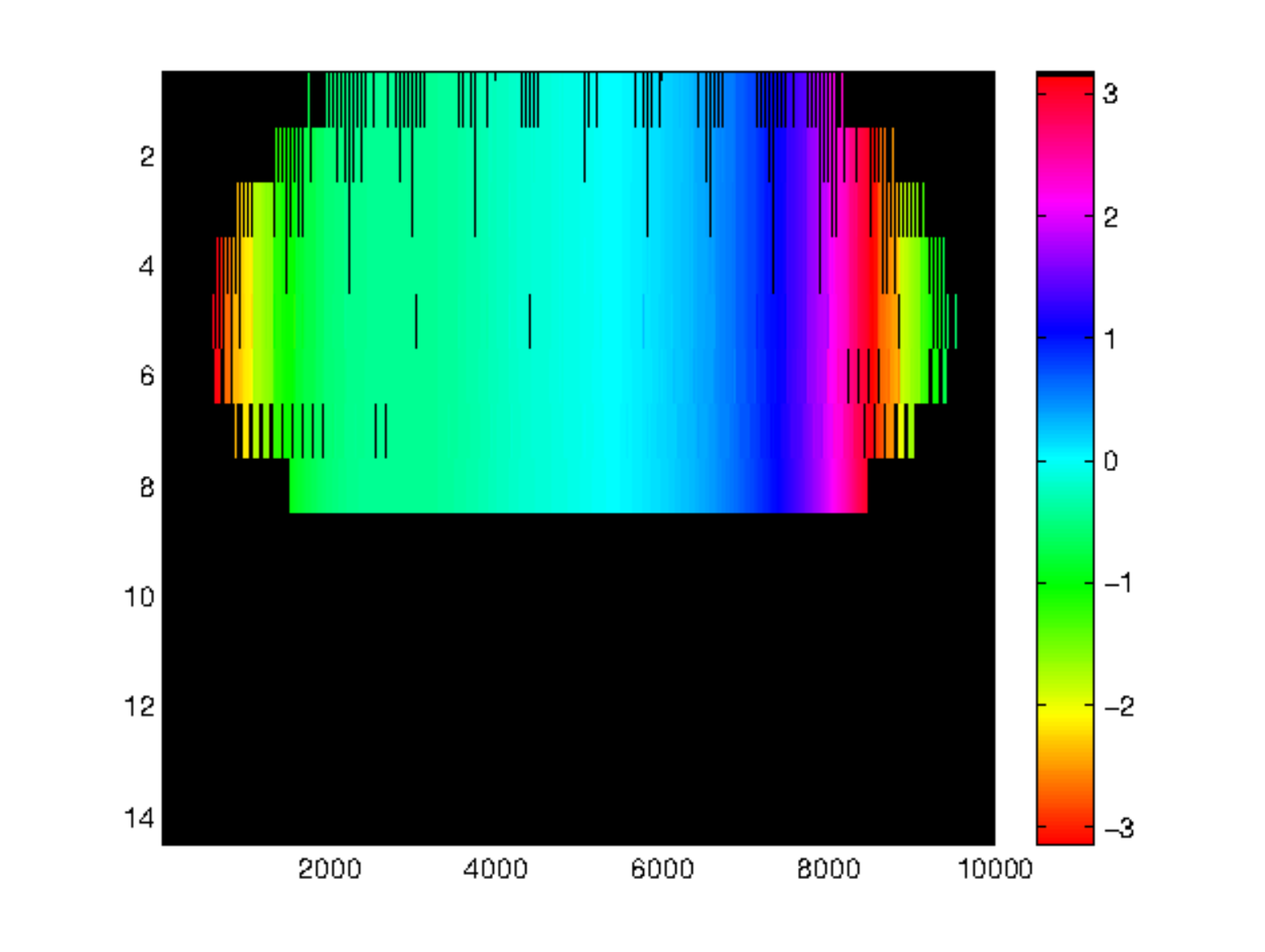}}
\caption{(a) modulus of the wavelet transform of a signal (b) phase difference between the wavelet transform of this signal and of its reconstruction (black points correspond to places where the modulus is too small for the phase to be meaningful)\label{fig:diff_phase}}
\end{figure}

\subsection{Influence of the parameters\label{ss:influence}}

In this paragraph, we analyze the importance of the two main parameters of the algorithm: the choice of the wavelets (paragraph \ref{sss:choice_family}) and the number of iterations allowed per local optimization step (paragraph \ref{sss:number_iterations}).

\subsubsection{Choice of the wavelets\label{sss:choice_family}}

Two properties of the wavelets are especially important: the exponential decay of the wavelets in the Fourier domain (so that the $Q_j$'s \eqref{eq:def_Qj} are correctly computed) and the amount of overlap between two neighboring wavelets (if the overlap is too small, then $f\star\psi_J,...,f\star\psi_{j+1}$ contain not much information about $f\star\psi_j$ and the multiscale approach is less efficient).

\begin{figure}
\centering
\subfloat[]{
\includegraphics[width=0.23\textwidth]{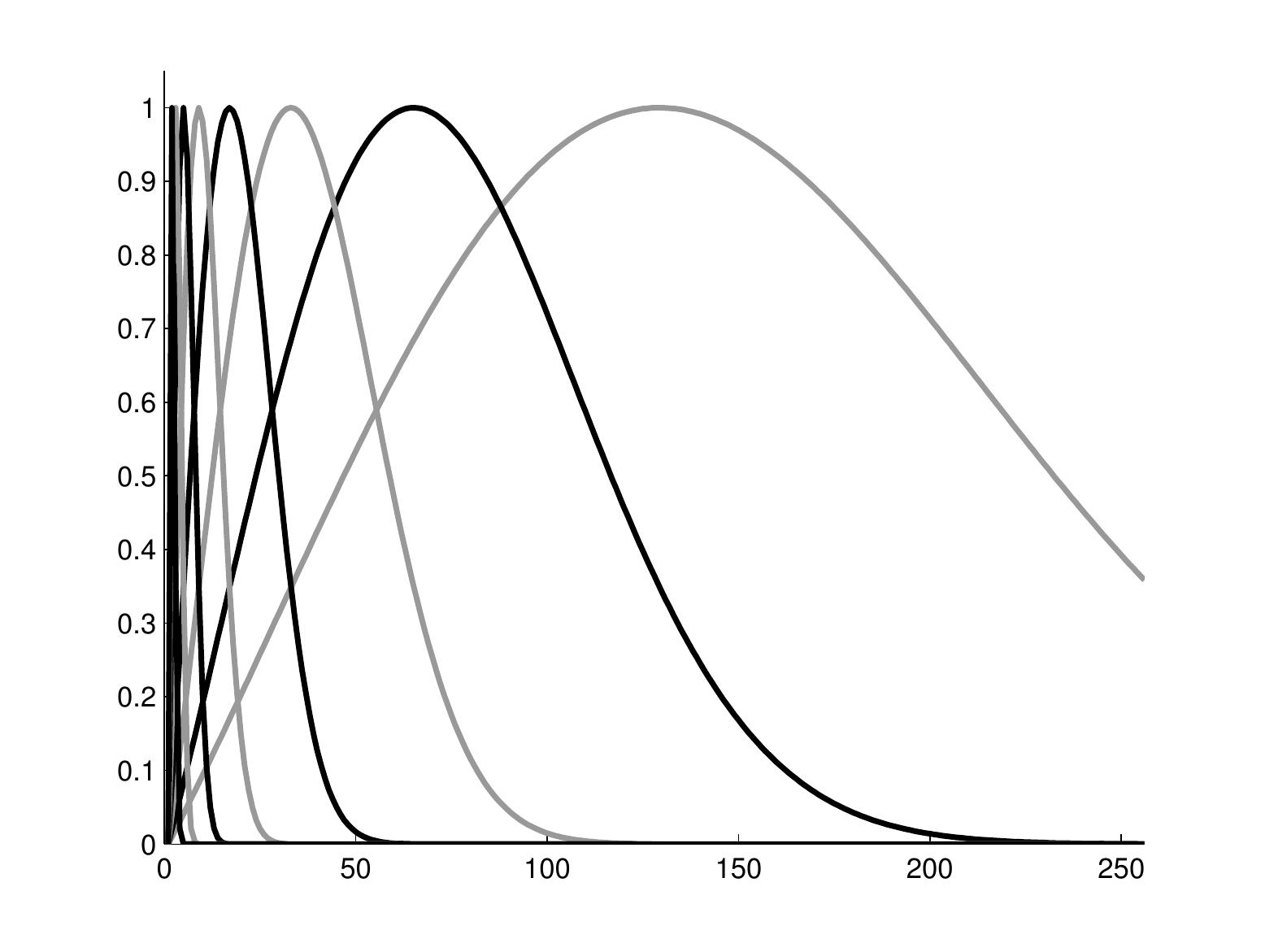}
\label{fig:wav_family_1}}
\hfil
\subfloat[]{
\includegraphics[width=0.23\textwidth]{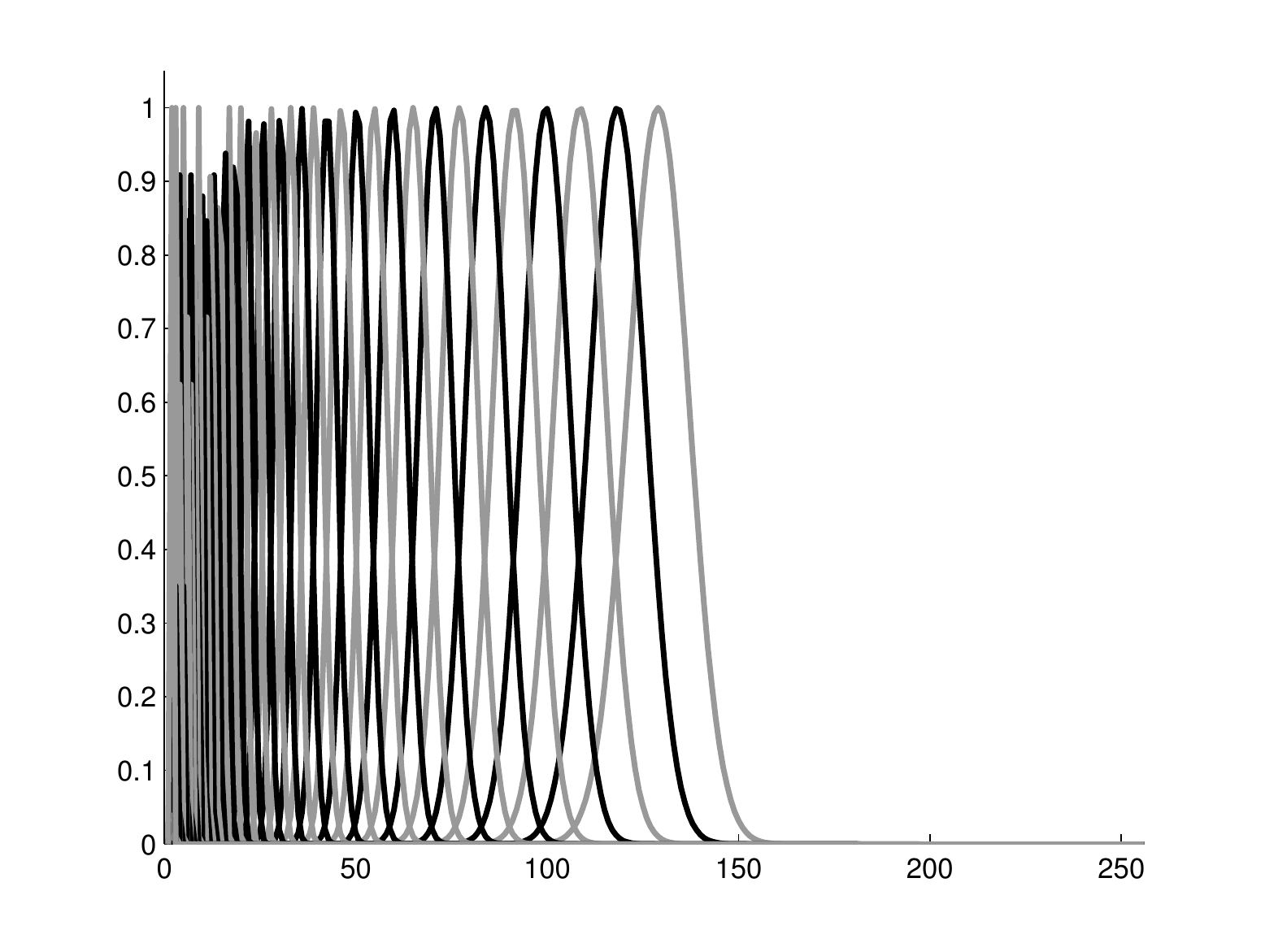}
\label{fig:wav_family_2}}
\hfil
\subfloat[]{
\includegraphics[width=0.23\textwidth]{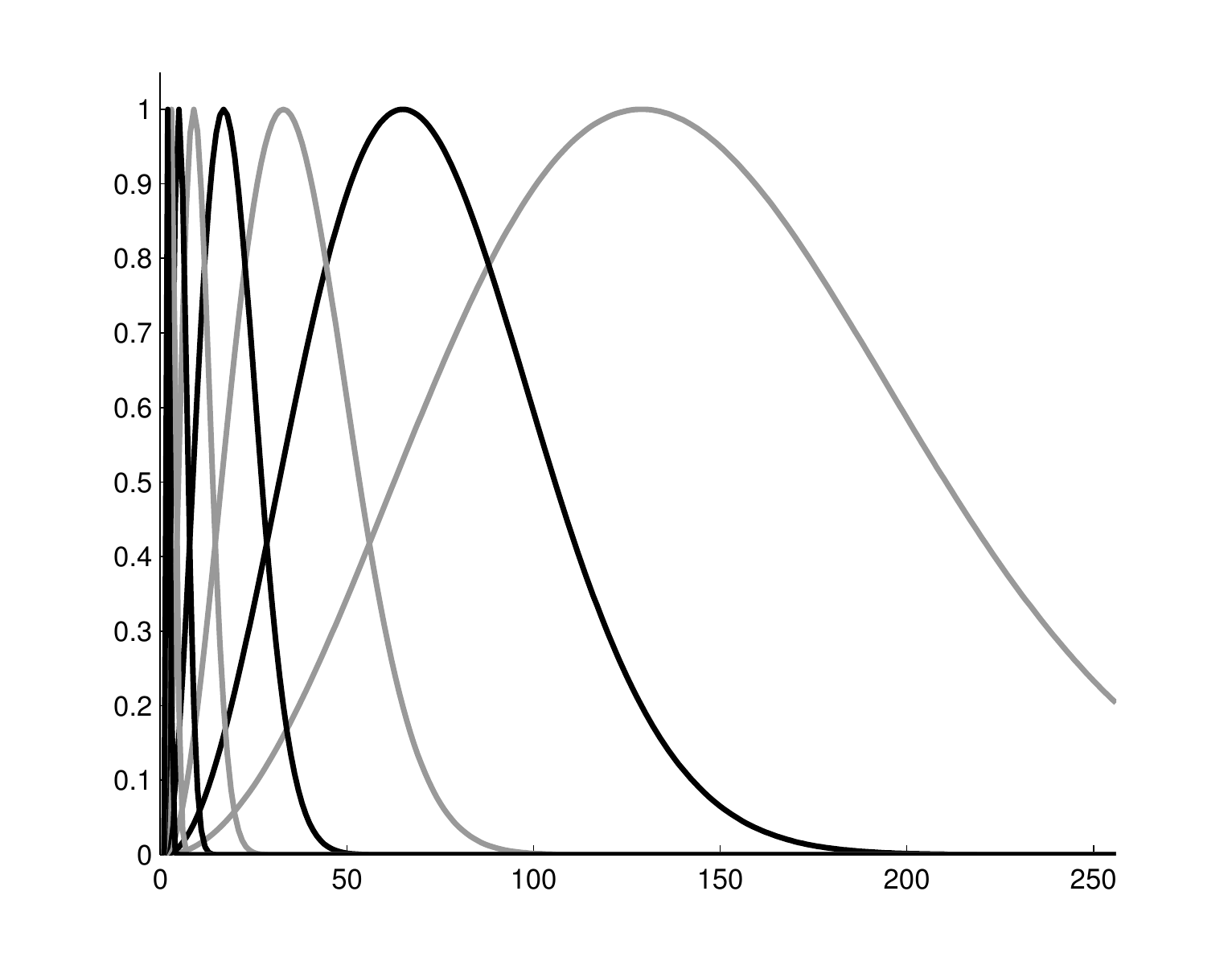}
\label{fig:wav_family_3}}
\hfil
\subfloat[]{
\includegraphics[width=0.23\textwidth]{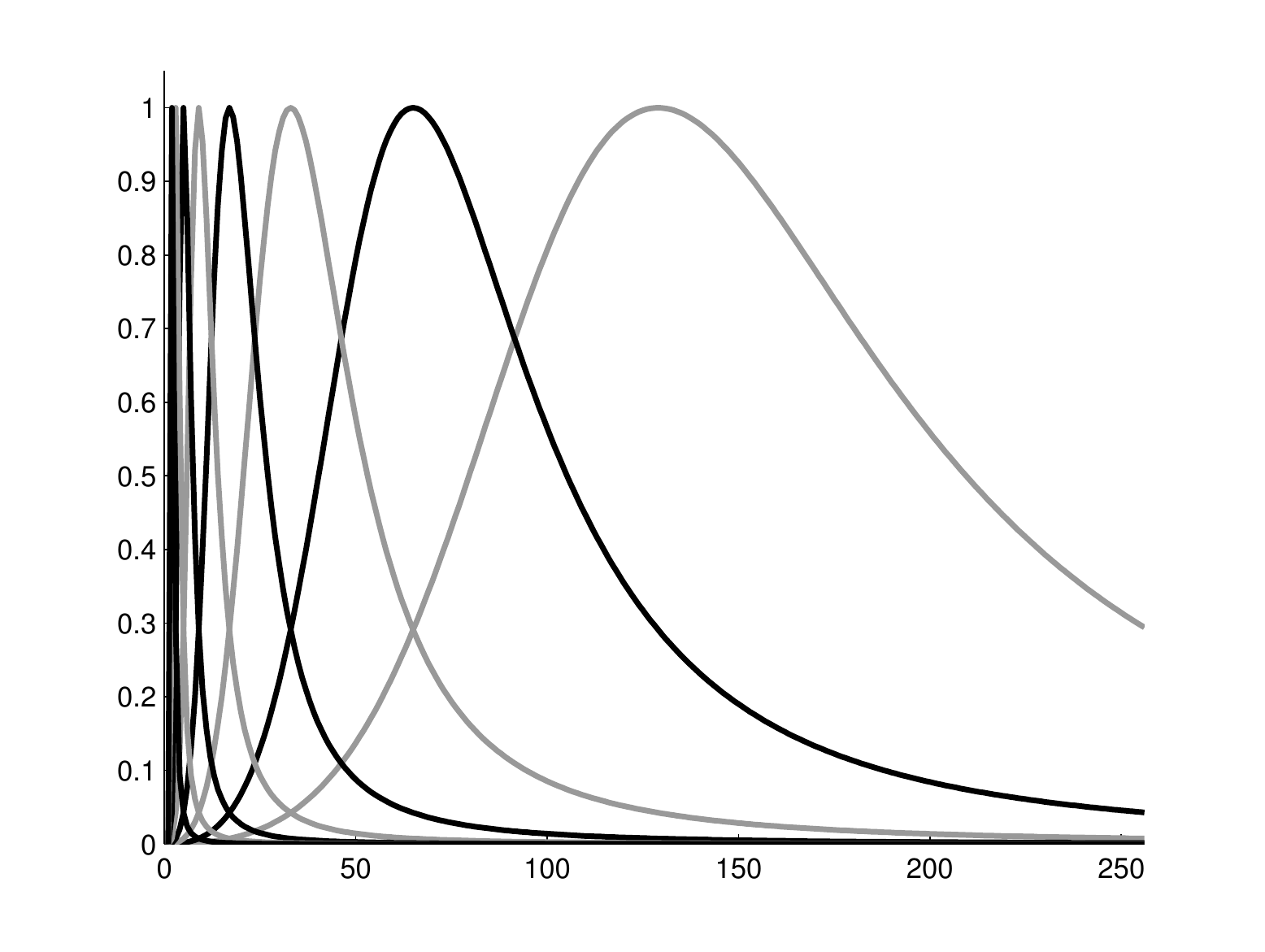}
\label{fig:wav_family_4}}
\caption{Four wavelet families. (a) Morlet (b) Morlet with dilation factor $2^{1/8}$ (c) Laplacian (d) Gammatone \label{fig:wav_family}}
\end{figure}

\nl
We compare the reconstruction results for four families of wavelets.

The first family (figure \ref{fig:wav_family_1}) is the one we used in all the previous experiments. It contains dyadic Morlet wavelets. The second family (figure \ref{fig:wav_family_2}) also contains Morlet wavelets, with a smaller bandwidth ($Q$-factor $\approx 8$) and a dilation factor of $2^{1/8}$ instead of $2$. This is the kind of wavelets used in audio processing. The third family (figure \ref{fig:wav_family_3}) consists in dyadic Laplacian wavelets $\hat\psi(\omega)=\omega^2e^{1-\omega^2}$. Finally, the wavelets of the fourth family (figure \ref{fig:wav_family_4}) are (derivatives of) Gammatone wavelets.

\begin{figure}
\centering
\subfloat[]{
\includegraphics[width=0.4\textwidth]{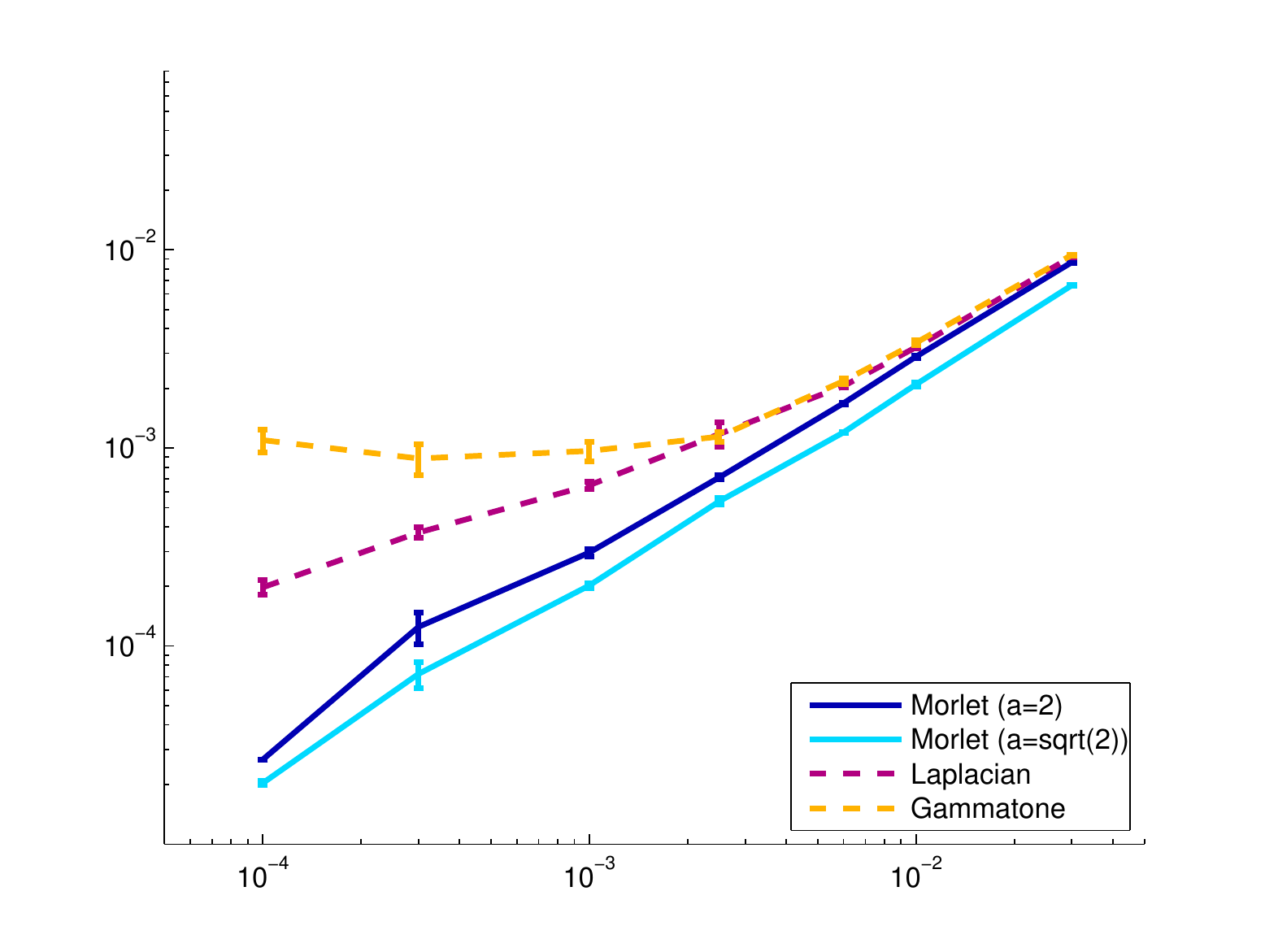}
\label{fig:family_piece}}
\hfil
\subfloat[]{
\includegraphics[width=0.4\textwidth]{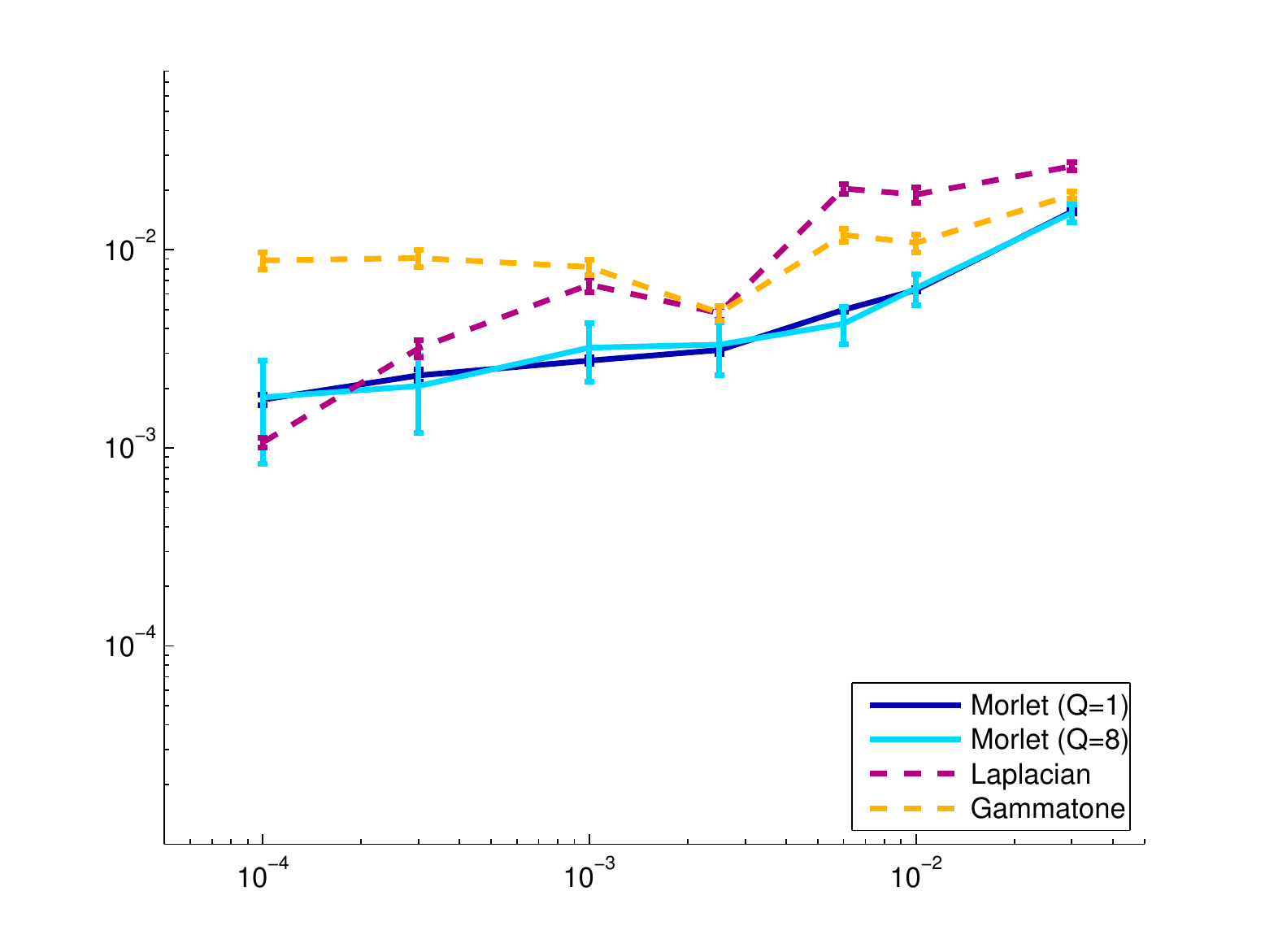}
\label{fig:family_sorry}}
\caption{Mean reconstruction error as a function of the noise for the four wavelet families displayed in \ref{fig:wav_family}. (a) Lines from images (b) Audio signal ``I'm sorry''\label{fig:family}}
\end{figure}

\nl
Figure \ref{fig:family} displays the mean reconstruction error as a function of the noise, for two classes of signals: lines randomly extracted from natural images and audio signals.

Morlet wavelets have a fast decay and consecutive wavelets overlap well. This does not depend upon the dilation factor so the reconstruction performances are similar for the two Morlet families (figures \ref{fig:family_piece} and \ref{fig:family_sorry}).

Laplacian wavelets are similar, but the overlap between consecutive wavelets is not as good. So Laplacian wavelets globally have the same behavior as Morlet wavelets but require significantly more computational effort to reach the same precision. Figures \ref{fig:family_piece} and \ref{fig:family_sorry} have been generated with a maximal number of iterations per optimization step equal to $30000$ (instead of $10000$) and the reconstruction error is still larger.

The decay of Gammatone wavelets is polynomial instead of exponential. The products $Q_j$ cannot be efficiently estimated and our method performs significantly worse. In the case of lines extracted from images (\ref{fig:family_piece}), the reconstruction error stagnates at $0.1\%$, even when the noise is of the order of $0.01\%$. For audio signals (\ref{fig:family_sorry}), it is around $1\%$ for any amount of noise.


\subsubsection{Number of iterations in the optimization step\label{sss:number_iterations}}

The maximal number of iterations allowed per local optimization step (paragraph \ref{ss:local_optimization}) can have a huge impact on the quality of the reconstruction.

Figure \ref{fig:iterations} represents, for an audio signal, the reconstruction error as a function of this number of iterations. As the objective functions are not convex, there are no guarantees on the speed of the decay when the number of iterations increases. It can be slow and even non-monotonic. Nevertheless, it clearly globally decays.

The execution time is roughly proportional to the number of iterations. It is thus important to adapt this number to the desired application, so as to reach the necessary precision level without making the algorithm exaggeratedly slow.

\begin{figure}
\centering
\subfloat[]{
\includegraphics[width=0.23\textwidth]{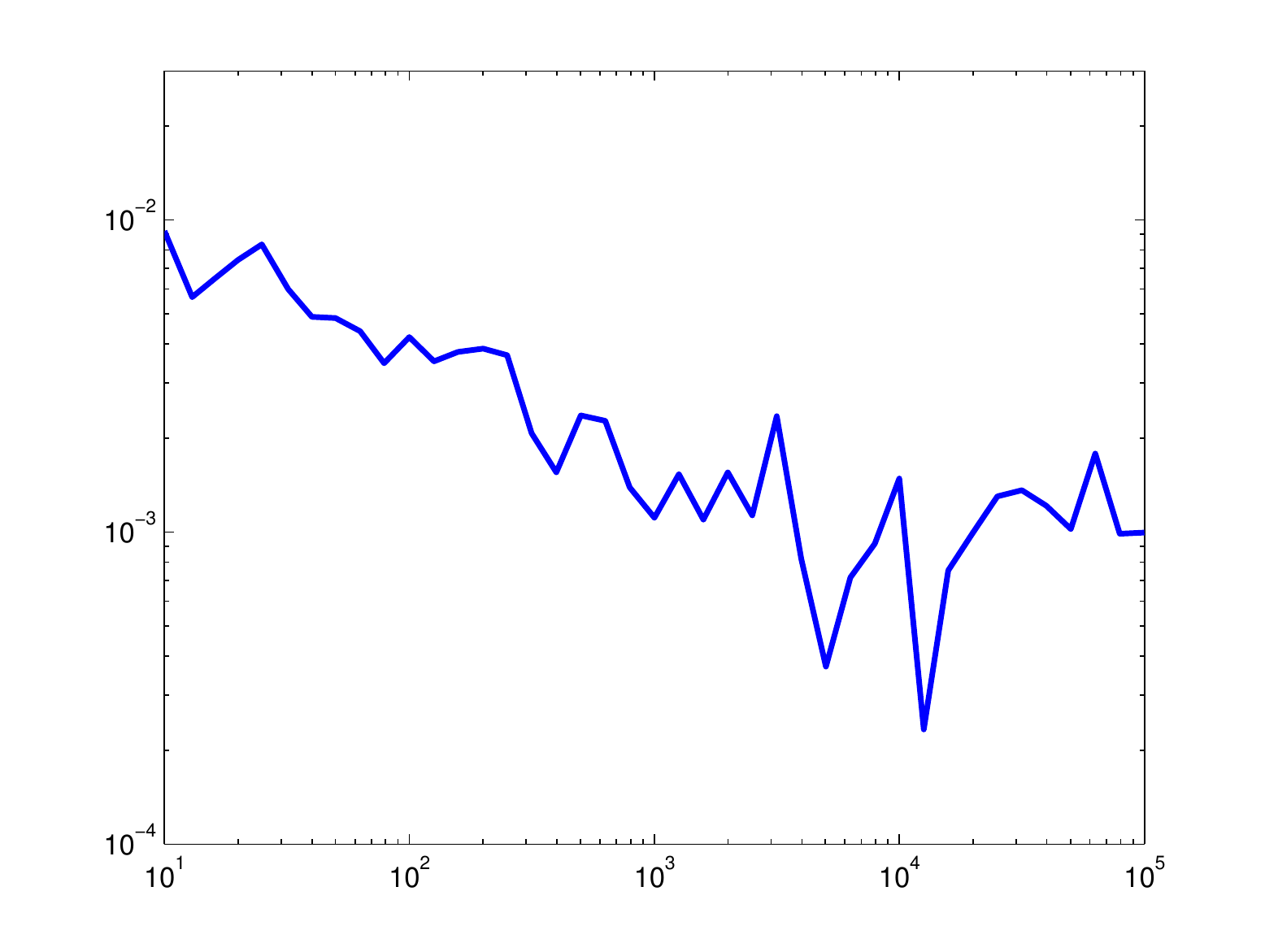}
\label{fig:iterations_1}}
\hfil
\subfloat[]{
\includegraphics[width=0.23\textwidth]{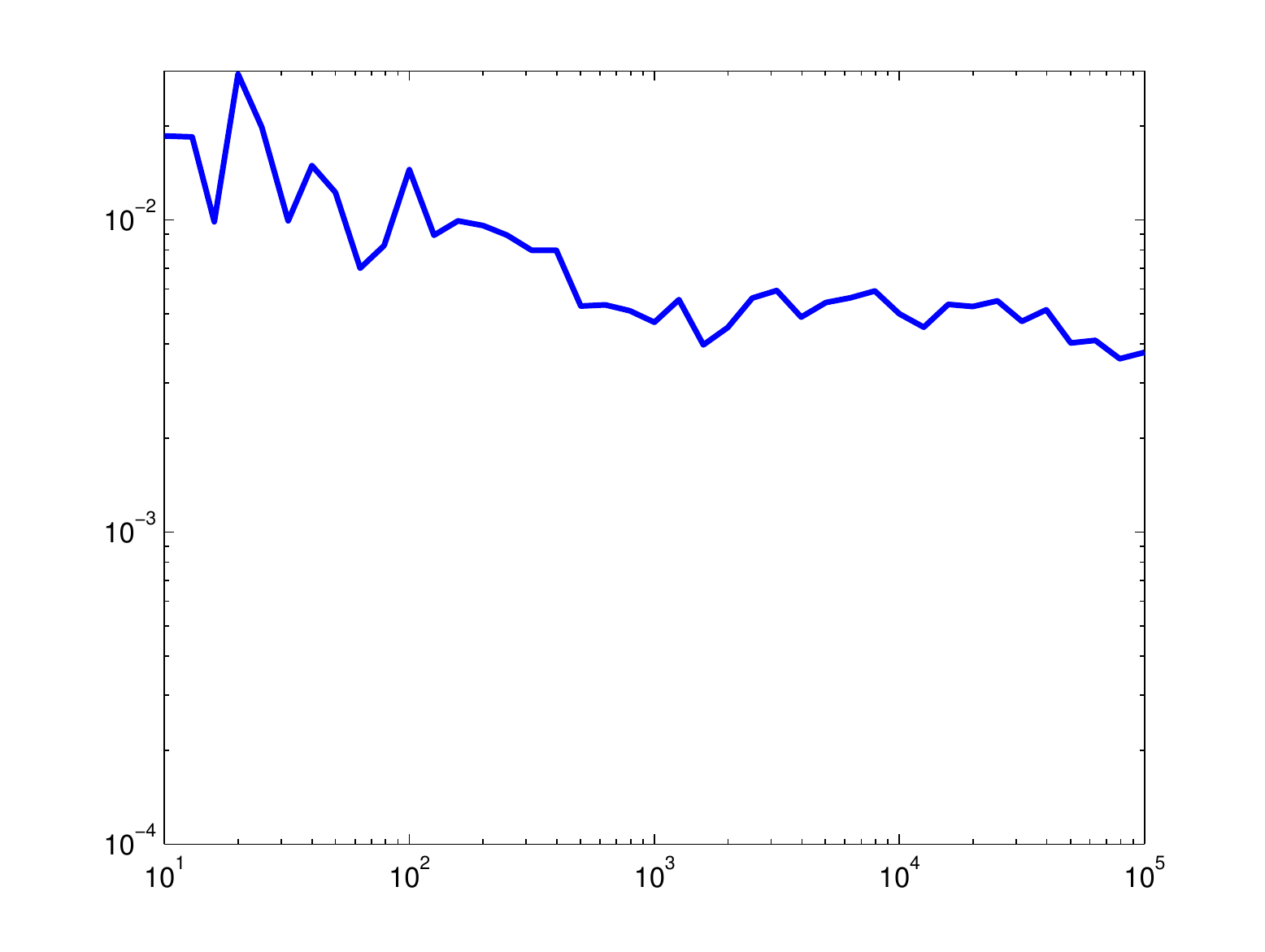}
\label{fig:iterations_2}}
\caption{for the audio signal ``I'm sorry'', reconstruction error as a function of the maximal number of iterations (a) with $0.01\%$ of noise (b) with $0.6\%$ of noise\label{fig:iterations}}
\end{figure}

\section{Conclusion}

We have presented an new iterative algorithm that reconstructs a signal from its scalogram. This algorithm is based on a new reformulation of the reconstruction problem, using the analytic extension of the wavelet transform. It is precise and stable to noise, and has a sufficiently low complexity to be applied to audio signals.

In future works, we plan to investigate further ways to speed up the reconstruction, including parallelization, and to test our algorithm on concrete applications, like source separation.

\section{Acknowledgments}

A large portion of this work has been done while the author was at École Normale Supérieure de Paris, where she has been partially supported by ERC grant InvariantClass 320959.

\appendices

\section{Proof of lemma \ref{lem:extension}\label{app:extension}}

\begin{lem*}[\ref{lem:extension}]
For any $f$ satisfying the analyticity condition \eqref{eq:analyticity_discrete},
\begin{align*}
\forall z\in\C\quad\quad
&P({|f\star\psi_j|^2})(r_jz)=P({(f\star\psi_j^{low})\overline{(f\star\psi_j^{high})}})(z)\\
\mbox{and}\quad&P({g_j^2})(r_jz)=P({Q_j})(z).
\end{align*}
\end{lem*}
\begin{proof}
Recall that, by definition, for any $h\in\C^N$,
\begin{equation*}
\forall z\in\C\quad\quad
P(h)(z)=\underset{k=\left\lfloor\frac{N}{2}\right\rfloor-N+1}{\overset{\left\lfloor\frac{N}{2}\right\rfloor}{\sum}}\hat h[k]z^k.
\end{equation*}
So for any two signals $h,H$, the condition $P(h)(r_jz)=P(H)(z),\forall z\in\C$ is equivalent to
\begin{equation}\label{eq:relation_fourier}
\forall k=\left\lfloor\frac{N}{2}\right\rfloor-N+1,...,\left\lfloor\frac{N}{2}\right\rfloor,\quad\quad
\hat h[k]r_j^k=\hat H[k].
\end{equation}
Applied to $h=g_j^2$ and $H=Q_j$, this property yields the equality $P({g_j^2})(r_jz)=P({Q_j})(z),\forall z\in\C$: by the definition of $Q_j$ in \eqref{eq:def_Qj}, Equation \eqref{eq:relation_fourier} is clearly satisfied.

Let us now show that
\begin{equation*}
P({|f\star\psi_j|^2})(r_jz)=P({(f\star\psi_j^{low})\overline{(f\star\psi_j^{high})}})(z),\,\forall z\in\C.
\end{equation*}
It suffices to prove that \eqref{eq:relation_fourier} holds, that is,
\begin{align*}
\forall k\in\left\lfloor\frac{N}{2}\right\rfloor-N+1&,...,\left\lfloor\frac{N}{2}\right\rfloor,\\
\widehat{|f\star\psi_j|^2} [k]r_j^k&=\widehat{(f\star\psi_j^{low})\overline{(f\star\psi_j^{high})}} [k].
\end{align*}
Indeed, because the analyticity condition \eqref{eq:analyticity_discrete} holds, we have for all $k$,
\begin{align*}
\widehat{|f\star\psi_j|^2}[k]
&=\frac{1}{N}\left(\widehat{f\star\psi_j}\right)\star\left(\widehat{\overline{f\star\psi_j}}\right)[k]\\
&=\frac{1}{N}\underset{l=1}{\overset{\lfloor N/2\rfloor}{\sum}}\hat f[l]\hat\psi_j[l]\overline{\hat f[l-k]}\overline{\hat\psi_j[l-k]}\\
&=\frac{r_j^{-k}}{N}\underset{l=1}{\overset{\lfloor N/2\rfloor}{\sum}}\hat f[l]\hat\psi_j^{low}[l]\overline{\hat f[l-k]}\overline{\hat\psi_j^{high}[l-k]}\\
&=\frac{r_j^{-k}}{N}\left(\widehat{f\star\psi^{low}_j}\right)\star\left(\widehat{\overline{f\star\psi^{high}_j}}\right)[k]\\
&=r_j^{-k}\left(\widehat{(f\star\psi_j^{low})(\overline{f\star\psi_j^{high}})}\right)[k].
\end{align*}
\end{proof}

\section{Proofs of Proposition \ref{prop:critical_points_1} and Theorem \ref{thm:critical_points_2}\label{s:critical_points}}

\begin{prop*}[\ref{prop:critical_points_1}]
Let $f_{rec}$ be the function that we aim at reconstructing. We assume that, for any $j=0,\dots,J$, the functions $f_{rec}\star\psi_j^{low}$ and $f_{rec}\star\psi_j^{high}$ have no zero entry. Then, on the vector space defined by Equation \eqref{eq:additional_constraint_mod}, the critical points of $\mathrm{obj}_1$ that are not strict saddle are
\begin{align*}
\Big\{
\Big(&\gamma (f_{rec}\star \psi_{J}^{low}),\frac{1}{\overline{\gamma}}(f_{rec}\star \psi_{J-1}^{low}),\gamma (f_{rec}\star \psi_{J-2}^{low}),\dots,\\
&\frac{1}{\overline{\gamma}}(f_{rec}\star\psi_J^{high}),\gamma (f_{rec}\star\psi_{J-1}^{high}),\frac{1}{\overline{\gamma}}(f_{rec}\star\psi_{J-2}^{high}),\dots\Big),\\
&\quad\quad\mbox{such that } \gamma\in (\C-\{0\})^N\Big\}
\quad\overset{def}{=}\quad\mathcal{C}_1.
\end{align*}
\end{prop*}
\begin{proof}
We first note that, by Theorem \ref{thm:reformulation},
\begin{equation}\label{eq:dev_Q_j}
Q_j=(f_{rec}\star\psi_j^{low})\overline{(f_{rec}\star\psi_j^{high})}\quad j=0,\dots,J.
\end{equation}
Let $H=(h_J^{low},\dots,h_0^{high})$ be a critical point that is not strict saddle. The derivative of $\mathrm{obj}_1$ at $H$ is
\begin{align*}
d\mathrm{obj}_1(H).&(s_J^{low},\dots,s_0^{low},s_J^{high},\dots,s_0^{high})\\
&=2\sum_{j=0}^J\Re\scal{s_j^{low}}{g_j^{low}}+\Re\scal{s_j^{high}}{g_j^{high}},
\end{align*}
\begin{align*}
\mbox{with }g_j^{low}&=h_j^{high}(h_j^{low}\overline{h_j^{high}}-Q_j)\\
g_j^{high}&=h_j^{low}(\overline{h_j^{low}}h_j^{high}-\overline{Q_j}).
\end{align*}
As $H$ is a critical point on the vector space defined by constraints \eqref{eq:additional_constraint_mod}, we must have
\begin{gather}
g_J^{low}=0;\nonumber\\
\forall j=0,\dots,J-1,\quad
a^{p_1}g_{j+1}^{high}+g_j^{low}=0;\label{eq:null_gradient}\\
g_0^{high}=0.\nonumber
\end{gather}
We set $\gamma=h_J^{low}(f_{rec}\star\psi_J^{low})^{-1}$. Let us show by decreasing induction over $j=J,\dots,0$ that $h_j^{low},h_j^{high}$ satisfy
\begin{subequations}\label{eq:hyp_rec}
\begin{align}
h_j^{low}&=\gamma(f_{rec}\star\psi_j^{low})\\
\mbox{ and }
h_j^{high}&=\frac{1}{\overline{\gamma}}(f_{rec}\star\psi_j^{high})\mbox{ if }j\equiv J[2];\\
h_j^{low}&=\frac{1}{\overline{\gamma}}(f_{rec}\star\psi_j^{low})\\
\mbox{ and }
h_j^{high}&=\gamma(f_{rec}\star\psi_j^{high})\mbox{ if }j\not\equiv J[2].
\end{align}
\end{subequations}
We start with $j=J$. If we can show that $h_J^{high}$ has no zero entry, then the equality $g_J^{low}=0$ implies, by the definition of $g_J^{low}$:
\begin{gather*}
h_J^{low}\overline{h_J^{high}}=Q_J=
(f_{rec}\star\psi_J^{low})\overline{(f_{rec}\star\psi_J^{high})}.
\end{gather*}
As $f_{rec}\star\psi_J^{low}$ and $f_{rec}\star\psi_J^{high}$ have no zero entry (by assumption), from the previous equation and the definition of $\gamma$,
\begin{equation*}
h_J^{low}=\gamma(f_{rec}\star\psi_J^{low})\mbox{ and }
h_J^{high}=\frac{1}{\overline{\gamma}}(f_{rec}\star\psi_J^{high}).
\end{equation*}
So let us just show that $h_J^{high}$ has no zero entry. By contradiction, we assume that $h_J^{high}[k]=0$ for some $k=0,\dots,N-1$. Let $S^1\in (\C^N)^{2(J+1)}$ be defined as
\begin{equation*}
S^1=(\delta_k,0,\dots,0,0,\dots,0),
\end{equation*}
where $\delta_k\in\C^N$ is the vector such that $\delta_k[l]=0$ if $l\ne k$ and $\delta_k[k]=1$. Similarly, we define
\begin{equation*}
S^2=(0,\delta_k,\dots,0,a^{p_1}\delta_k,\dots,0)
\end{equation*}
(where $a^{p_1}\delta_k$ is the $J+2$-th component).

Using the equality $h_J^{high}[k]=0$, we see that
\begin{align}
d^2\mathrm{obj}_1(H).(S^1,S^1)&=0;\label{eq:H_diag_zero}\\
d^2\mathrm{obj}_1(H).(iS^1,iS^1)&=0.\nonumber
\end{align}
However,
\begin{align*}
d^2\mathrm{obj}_1(H).(S^1,S^2)&=-2a^{p_1}\Re(Q_J[k]);\\
d^2\mathrm{obj}_1(H).(iS^1,S^2)&=-2a^{p_1}\Im(Q_J[k]).
\end{align*}
As $Q_J[k]=(f\star\psi_J^{low}[k])\overline{(f\star\psi_J^{high}[k])}\ne 0$, it means that $d^2\mathrm{obj}_1(H).(S^1,S^2)\ne 0$ or $d^2\mathrm{obj}_1(H).(iS^1,S^2)\ne 0$. Combined with Equation \eqref{eq:H_diag_zero}, it implies that $d^2\mathrm{obj}_1(H)$ has a strictly negative eigenvalue, so $H$ is strict saddle. This is a contradiction, and concludes the initialization.

We now assume the inductive hypothesis to hold from $j+1$ to $J$, for some $j<J$, and prove it for $j$. From Equations \eqref{eq:dev_Q_j}, \eqref{eq:hyp_rec} and the definition of $g_{j+1}$,
\begin{equation*}
g_{j+1}^{high}=0.
\end{equation*}
From Equation \eqref{eq:null_gradient},
\begin{equation*}
g_j^{low}=0,
\end{equation*}
so, if $h_j^{high}$ has no zero entry, we must have
\begin{equation*}
h_j^{low}\overline{h_j^{high}}=Q_j.
\end{equation*}
Constraint \eqref{eq:additional_constraint_mod} is satisfied, so, if $j\equiv J[2]$, by Equation \eqref{eq:hyp_rec},
\begin{align*}
h_j^{low}&=a^{-p_1}h_{j+1}^{high}\\
&=a^{-p_1}\gamma(f_{rec}\star\psi_{j+1}^{high})\\
&=\gamma(f_{rec}\star\psi_j^{low}).
\end{align*}
Combined with the previous equation and Equation \eqref{eq:dev_Q_j}, this relation also yields
\begin{equation*}
h_j^{high}=\frac{1}{\overline{\gamma}}(f_{rec}\star\psi_j^{high}).
\end{equation*}
A similar reasoning holds if $j\not\equiv J[2]$, so the induction hypothesis holds at rank $j$, provided that we show that $h_j^{high}$ has no zero entry.

To prove this last affirmation, we reason by contradiction, and assume $h_j^{high}[k]=0$ for some $k=0,\dots,N-1$. Multiplying everything by a suitable global phase, we can assume $\Re(\overline{h_j^{low}[k]}Q_j[k])\ne 0$ (this number is non zero, by the assumptions over $f_{rec}$ and by the inductive hypothesis applied to $h_j^{low}=a^{-p_1}h_{j+1}^{high}$). We define
\begin{align*}
S^1=(&h_J^{low}[k]\delta_k,-h_{J-1}^{low}[k]\delta_k,\dots,\pm h_j^{low}[k]\delta_k,0,\dots,0,\\
-&h_{J}^{high}[k]\delta_k,h_{J-1}^{high}[k]\delta_k,\dots,\pm h_{j+1}^{high}[k]\delta_k,0,\dots,0);\\
S^2=(&0,\dots,\delta_k,\dots,0,0,\dots,a^{p_1}\delta_k,\dots,0).
\end{align*}
(In the second definition, the non-zero components are at positions $J-j+2$ and $2J-j+2$.)

We have
\begin{align*}
d^2\mathrm{obj}_1(H).(S^1,S^1)&=0;\\
d^2\mathrm{obj}_1(H).(S^1,S^2)&=-2(-1)^{J-j}a^{p_1}\Re(\overline{h_j^{low}[k]}Q_j[k])\ne 0.
\end{align*}
So $d^2\mathrm{obj}_1(H)$ has at least a strictly negative eigenvalue; this is a contradiction.

This concludes the induction, and proves that all critical points that are not strict saddle have the desired form. Conversely, any point that has this form is a global minimum of $\textrm{obj}_1$, by Equation \eqref{eq:dev_Q_j}, so it is a critical point, and is not strict saddle.
\end{proof}

\begin{thm*}[\ref{thm:critical_points_2}]
We assume $J\geq 2$, and keep the hypotheses of Proposition \ref{prop:critical_points_1}.

Let $(h_J^{low},\dots,h_0^{low},h_J^{high},\dots,h_0^{high},f)$ be a critical point of $\mathrm{obj}_2$ on the real manifold $\mathcal{C}_1\times \mathcal{A}$. There exists $\alpha\in\C-\{0\},b\in\C$ such that
\begin{equation*}
f=\alpha f_{rec}+b.
\end{equation*}
\end{thm*}
\begin{proof}
The tangent to $\mathcal{C}_1\times \mathcal{A}$ at a point $H=(h_J^{low},\dots,h_J^{high},\dots,f)$ is the real vector space
\begin{align*}
\Big\{\Big(&\delta h_J^{low},-\overline{\delta}h_{J-1}^{low},\delta h_{J-2}^{low},\dots,
-\overline{\delta}h_{J}^{high},\delta h_{J-1}^{high},\dots,g \Big),\\
&\quad\quad \mbox{such that }\delta\in\C^N,g\in\mathcal{A}\Big\}.
\end{align*}
If $H$ is a critical point of $\mathrm{obj}_2$, we evaluate $d\mathrm{obj}_2(H)$ along the vector of the tangent space that corresponds to $\delta=1$ and $g=f$, and we obtain
\begin{equation*}
2\mathrm{obj}_2(H)=0.
\end{equation*}
In particular,
\begin{align*}
h_J^{low}&=f\star\psi_J^{low};\\
h_{J-2}^{low}&=f\star\psi_{J-2}^{low}.
\end{align*}
So, from the definition of $\mathcal{C}_1$, there exists $\gamma\in(\C-\{0\})^N$ such that
\begin{subequations}
\begin{align}
\gamma(f_{rec}\star\psi_J^{low})&=f\star\psi_J^{low};\label{eq:kernel_1}\\
\gamma(f_{rec}\star\psi_{J-2}^{low})&=f\star\psi_{J-2}^{low}.\label{eq:kernel_2}
\end{align}
\end{subequations}
From the definition of Cauchy wavelets and $r$, we deduce that
\begin{align*}
\hat\psi_J^{low}[k]=a^{Jp_1}k^{p_1}&\exp\left(-k a^Jp_2\left(\frac{2 a}{a+1}\right)\right),\\
&\forall k=0,\dots,N-1.
\end{align*}
So if we define
\begin{equation*}
P(f)(X)=\frac{1}{N}\sum_{l=1}^{[N/2]}l^{p_1}\hat f[l]X^l,
\end{equation*}
we see that, for any $k=0,\dots,N-1$,
\begin{align*}
f\star\psi_J^{low}[k]&=\frac{1}{N}\sum_{l=0}^{N-1}\hat f[l]\hat\psi_J[l]\left(e^{\frac{2\pi i k}{N}}\right)^l\\
&=a^{Jp_1}P(f)\left(R_J e^{\frac{2\pi ik}{N}}\right),
\end{align*}
where
\begin{equation*}
R_J = \exp\left(-a^Jp_2\left(\frac{2a}{a+1}\right)\right).
\end{equation*}
Similarly,
\begin{equation*}
f\star\psi_{J-2}^{low}[k]=a^{(J-2)p_1}P(f)(R_{J-2}e^{\frac{2\pi ik}{N}}),
\end{equation*}
where
\begin{equation*}
R_{J-2} = \exp\left(-a^{J-2}p_2\left(\frac{2a}{a+1}\right)\right).
\end{equation*}
In the same way, we have
\begin{align*}
f_{rec}\star\psi_{J}^{low}[k]&=a^{Jp_1}P(f_{rec})(R_{J}e^{\frac{2\pi ik}{N}});\\
f_{rec}\star\psi_{J-2}^{low}[k]&=a^{(J-2)p_1}P(f_{rec})(R_{J-2}e^{\frac{2\pi ik}{N}}),
\end{align*}
if we define
\begin{equation*}
P(f_{rec})(X)=\frac{1}{N}\sum_{l=1}^{[N/2]}l^{p_1}\hat f_{rec}[l]X^l.
\end{equation*}
A consequence of these previous equalities and Equations \eqref{eq:kernel_1} and \eqref{eq:kernel_2} is that
\begin{align}
\frac{P(f)(R_J x)}{P(f_{rec})(R_J x)}=
\frac{P(f)(R_{J-2} x)}{P(f_{rec})(R_{J-2} x)}
\quad\forall x=1,e^{\frac{2\pi i}{N}},e^{\frac{4\pi i}{N}},\dots.
\label{eq:equality_N_points}
\end{align}
(The denominator of these fractions is never zero, because $f_{rec}\star\psi_J^{low}$ and $f_{rec}\star\psi_{J-2}^{low}$ have no zero entry.)

The rational fraction $\frac{P(f)(X)}{P(f_{rec})(X)}=\frac{P(f)(X)/X}{P(f_{rec})(X)/X}$ is the quotient of two polynomials with degree at most $[N/2]-1$, so it is uniquely determined by its values in $N$ complex points, and Equation \eqref{eq:equality_N_points} implies:
\begin{equation*}
\frac{P(f)(x)}{P(f_{rec})(x)}=\frac{P(f)(xR_{J-2}/R_J)}{P(f_{rec})(xR_{J-2}/R_J)},
\quad\quad \forall x\in\C.
\end{equation*}
Applying recursively this equality yields, for any $s\in\N$,
\begin{equation*}
\frac{P(f)(x)}{P(f_{rec})(x)}=\frac{P(f)\left(x(R_{J-2}/R_J)^s\right)}{P(f_{rec})\left(x(R_{J-2}/R_J)^s\right)},
\quad\quad \forall x\in\C.
\end{equation*}
Depending on the exact degrees of $P(f)$ and $P(f_{rec})$, the right-hand side of this equalities converges (uniformly on every compact set) to either $\infty$ or a constant function, when $s\to +\infty$. The convergence to $\infty$ is actually impossible, because $\frac{P(f)}{P(f_{rec})}$ has finite values at almost any point of $\C$, so $\frac{P(f)}{P(f_{rec})}$ is a constant function:
\begin{equation*}
\frac{P(f)(z)}{P(f_{rec})(z)}=\alpha,\quad\quad \forall z\in\C,
\end{equation*}
for some $\alpha\in\C$.

From the definitions of $P(f)$ and $P(f_{rec})$,
\begin{equation*}
\hat f[l]=\alpha \hat f_{rec}[l]\quad\quad\forall l=1,\dots,[N/2].
\end{equation*}
As $f$ and $f_{rec}$ are analytic, we actually have
\begin{equation*}
\hat f[l]=\alpha \hat f_{rec}[l]\quad\quad\forall l=1,\dots,N-1.
\end{equation*}
So the Fourier transform $f-\alpha f_{rec}$ has only zero entries, except maybe in $l=0$, and $f-\alpha f_{rec}$ is a constant function.
\end{proof}

\section{Proof of lemma \ref{lem:rec_small}\label{app:rec_small}}

\begin{lem*} (\ref{lem:rec_small})
Let $m\in\R^N$ and $K\in\N^*$ be fixed. We consider the problem
\begin{align*}
\mbox{Find }g\in\C^N\mbox{ s.t. }&|g|=m\\
\mbox{and }&\mbox{Supp}(\hat g)\subset\{1,...,K\}.
\end{align*}
This problem has at most $2^{K-1}$ solutions, up to a global phase, and there exist a simple algorithm which, from $m$ and $K$, returns the list of all possible solutions.
\end{lem*}

\begin{proof}
We define a polynomial $P(g)$ by
\begin{equation*}
P(g)(X)=\hat g[1]+\hat g[2]X+...+\hat g[K]X^{K-1}.
\end{equation*}
We show that the constraint $|g|=m$ amounts to knowing $P(g)(X)\overline{P(g)}(1/X)$. This is in turn equivalent to knowing the roots of $P(g)$ (and thus knowing $g$) up to inversion with respect to the unit circle. There are in general $K-1$ roots, and each one can be inverted. This gives $2^{K-1}$ solutions.

We set
\begin{align*}
Q(g)(X)&=\overline{P(g)}(1/X)\\
&=\overline{\hat g[K]}X^{-(K-1)}+\overline{\hat g[K-1]}X^{-(k-2)}+...+\overline{\hat g[1]}.
\end{align*}
The equation $|g|^2=m^2$ is equivalent to $\widehat{|g|^2}=\widehat{m^2}$, that is $\frac{1}{N}\hat g\star \hat{\overline{g}}=\widehat{m^2}$. For each $k\in\{-(K-1),...,K-1\}$,
\begin{equation*}
\hat g\star\hat{\overline{g}}[k]
=\underset{s}{\sum}\hat g[k-s]\overline{\hat{g}[-s]}.
\end{equation*}
This number is the coefficient of order $k$ of $P(g)(X)Q(g)(X)$, so $|g|=m$ if and only if
\begin{equation}\label{eq:condition1}
P(g)(X)Q(g)(X)=N\underset{k=-(K-1)}{\overset{K-1}{\sum}}\widehat{m^2}[k]X^k.
\end{equation}
Let us denote by $r_1,...,r_{K-1}$ the roots of $P(g)(X)$, so that
\begin{gather*}
P(g)(X)=\hat g[K](X-r_1)...(X-r_{K-1});\\
Q(g)(X)=\overline{\hat g[K]}(1/X-\overline{r}_1)...(1/X-\overline{r}_{K-1}).
\end{gather*}
From \eqref{eq:condition1}, the equality $|g|=m$ holds if and only if $\hat g[K],r_1,...,r_{K-1}$ satisfy
\begin{align}
|\hat g[K]|^2&\underset{j=1}{\overset{K-1}{\prod}}(X-r_j)(1/X-\overline{r}_j)\nonumber\\
&=N\underset{k=-(K-1)}{\overset{K-1}{\sum}}\widehat{m^2}[k]X^k.
\label{eq:condition2}
\end{align}
If we denote by $s_1,1/\overline{s}_1,...,s_{K-1},1/\overline{s}_{K-1}$ the roots of the polynomial function $\underset{k=-(K-1)}{\overset{K-1}{\sum}}\widehat{m^2}[k]X^k$, then the only possible choices for $r_1,...,r_{K-1}$ are, up to permutation,
\begin{equation*}
r_1=s_1\mbox{ or }1/\overline{s}_1,\quad\quad\quad
r_2=s_2\mbox{ or }1/\overline{s}_2,\quad\quad\quad
\dots
\end{equation*}
So there are $2^{K-1}$ possibilities. Once the $r_j$ have been chosen, $\hat g[K]$ is uniquely determined by \eqref{eq:condition2}, up to multiplication by a unitary complex.

From $r_1,...,r_{K-1},\hat g[K]$, $P(g)$ is uniquely determined and so is $g$. The algorithm is summarized in Algorithm \ref{alg:exhaustive}.
\end{proof}

\begin{algorithm}
\caption{reconstruction by exhaustive search for a small problem\label{alg:exhaustive}}
\begin{algorithmic}[1]
\REQUIRE $K$,$m$
\STATE Compute the roots of $\underset{k=-(K-1)}{\overset{K-1}{\sum}}\widehat{m^2}[k]X^k$.
\STATE Group them by pairs $(s_1,1/\overline{s}_1),...,(s_{K-1},1/\overline{s}_{K-1})$.
\STATE List the $2^{K-1}$ elements $(r_1,...,r_{K-1})$ of $\{s_1,1/\overline{s}_1\}\times ...\times\{s_{K-1},1/\overline{s}_{K-1}\}$.
\FORALL{the elements}
\STATE Compute the corresponding $\hat g[K]$ by \eqref{eq:condition2}.
\STATE Compute the coefficients of $P(g)(X)=\hat g[K](X-r_1)...(X-r_{K-1})$.
\STATE Apply an IFFT to the coefficients to obtain $g$.
\ENDFOR
\ENSURE the list of $2^{K-1}$ possible values for $g$.
\end{algorithmic}
\end{algorithm}

\bibliographystyle{plainnat}
\bibliography{../bib_articles.bib,../bib_proceedings.bib,../bib_livres.bib,../bib_misc.bib}

\end{document}